\documentclass[10pt,reqno]{amsart}

\usepackage{amsmath,amssymb,amsthm}
\usepackage{mathrsfs}
\usepackage{enumitem}
\usepackage{booktabs}
\usepackage{array}
\usepackage{graphicx}
\usepackage{caption}
\usepackage{subcaption}
\usepackage{xcolor}
\usepackage[a4paper,margin=2.6cm]{geometry}
\usepackage{hyperref}
\usepackage{cleveref}

\hypersetup{
  colorlinks=true,
  linkcolor=blue!70!black,
  citecolor=green!50!black,
  urlcolor=blue!70!black,
}

\theoremstyle{plain}
\newtheorem{theorem}{Theorem}[section]
\newtheorem{proposition}[theorem]{Proposition}
\newtheorem{lemma}[theorem]{Lemma}
\newtheorem{corollary}[theorem]{Corollary}
\newtheorem{conjecture}[theorem]{Conjecture}
\newtheorem*{thmA}{Theorem A}
\newtheorem*{thmAp}{Theorem A$'$}
\newtheorem*{thmB}{Theorem B}
\newtheorem*{thmC}{Theorem C}
\newtheorem*{thmD}{Theorem D}
\newtheorem*{thmE}{Theorem E}
\newtheorem*{thmU}{Theorem U}
\newtheorem*{thmM}{Theorem M}

\theoremstyle{definition}

\theoremstyle{remark}
\newtheorem{remark}[theorem]{Remark}

\DeclareMathOperator{\Br}{Br}
\DeclareMathOperator{\NS}{NS}
\DeclareMathOperator{\Pic}{Pic}

\DeclareMathOperator{\disc}{disc}
\DeclareMathOperator{\tr}{tr}
\DeclareMathOperator{\ord}{ord}
\DeclareMathOperator{\Frob}{Frob}

\DeclareMathOperator{\mult}{mult}
\DeclareMathOperator{\Reg}{Reg}
\DeclareMathOperator{\Gal}{Gal}
\DeclareMathOperator{\Hom}{Hom}
\DeclareMathOperator{\Aut}{Aut}
\DeclareMathOperator{\End}{End}
\DeclareMathOperator{\rank}{rank}
\DeclareMathOperator{\Irr}{Irr}
\DeclareMathOperator{\lead}{lead}
\newcommand{\Sha}{\text{\tencyr{Sh}}}
\newcommand{\Fq}{\mathbb{F}_q}

\newcommand{\Fqbar}{\overline{\mathbb{F}}_q}
\newcommand{\Pp}{\mathbb{P}^1}
\newcommand{\Ql}{\mathbb{Q}_\ell}

\newcommand{\Z}{\mathbb{Z}}
\newcommand{\Q}{\mathbb{Q}}
\newcommand{\C}{\mathbb{C}}
\newcommand{\A}{\mathbb{A}}
\newcommand{\Zhat}{\widehat{\mathbb{Z}}}

\newcommand{\calO}{\mathcal{O}}
\newcommand{\calF}{\mathcal{F}}
\newcommand{\Fam}{\mathcal{F}_5}
\newcommand{\E}{\mathbb{E}}
\newcommand{\WE}{W(E_8)}

\newcommand{\Lunit}{L_{\mathrm{unit}}}
\newcommand{\Xbar}{\overline{X}}
\newcommand{\Kbar}{\overline{K}}
\newcommand{\detp}{{\det}'}

\font\tencyr=wncyr10

\newsavebox{\tablebox}
\newcommand{\fitinput}[1]{\sbox{\tablebox}{\input{#1}}\ifdim\wd\tablebox>\textwidth\resizebox{\textwidth}{!}{\usebox{\tablebox}}\else\usebox{\tablebox}\fi}

\title[Frobenius in $W(E_8)$ for $y^2=x^3+x+D(t)$]{Frobenius in $W(E_8)$ for $y^2=x^3+x+D(t)$ over $\mathbb F_q(t)$: ranks, Tate--Shafarevich groups and trace statistics}

\author[D.\ Wachs]{Dane Wachs}
\address{The University of Arizona, Tucson, AZ}
\email{wachs@arizona.edu}

\date{\today}

\subjclass[2020]{11G05, 11G40, 14G10, 14J27, 11R45, 20F55}

\keywords{Elliptic curves, function fields,
  Tate--Shafarevich group, Artin--Tate conjecture, Mordell--Weil lattices,
  Weyl group of $E_8$, monodromy, Chebotarev density}

\begin{document}

\begin{abstract}
Let $q$ be a power of a prime $p\ge5$ and let $\Fam(q)$ be the set of curves
$E_D\colon y^2=x^3+x+D(t)$ over $\Fq(t)$ with $D$ monic of degree~$5$ and
$4+27D^2$ squarefree.  Each $E_D$ is the generic fibre of a rational elliptic
surface with Mordell--Weil lattice $E_8$, Frobenius acts on it by some
$w_D\in\WE$, and $L(E_D,T)=\det(1-qw_DT)$ (Shioda, Oguiso--Shioda, Ulmer).
From known results (Galois descent, Grothendieck's $\Sha\cong\Br$, Artin,
Milne, Urabe) one reads off $E_D(\Fq(t))=E_8^{w_D}$,
$\Sha(E_D)\cong(E_8/(1-w_D)E_8)_{\mathrm{tors}}$ and
$|\Sha|\cdot\Reg=\detp(1-w_D)$; by the class tables of Urabe and Karras,
$|\Sha|$ is not determined by $L(E_D,T)$ at positive rank.  The main new
result is that the monodromy of this $4$-parameter family is $\WE$ in every
characteristic $p\ge5$ (computer-assisted).  With Deligne's theorem, and
Katz's Betti bounds for uniformity in $p$, $w_D$ is equidistributed in $\WE$
with error $O(q^{-1/2})$ as $q\to\infty$ through all prime powers; e.g.\ the
mean degree-$1$ trace over the rank-$r$ curves, which equals
$-\E[\tr w_D\mid\mult_1(w_D)=r]$, tends to $8/29$ at rank~$0$.  The constant is astronomical.  We prove first-order ($1/q$)
formulas for the moments of $\tr w_D$ of order $\le3$; their coefficients
depend on $q$ through $\chi_q(-3)$ and, at order $3$, through a Frobenius trace
on a $2$-dimensional $H^1$.  We conjecture a first-order law for the
whole class distribution; one of its four terms involves the CM newform
$\eta(4z)^6$ (checked for all primes below $3000$).  The law fits
exact class counts at all prime powers $q\le41$, including out-of-sample
predictions at $q=37,41$, some of which (rank statistics at $q=41$) missed.  This version replaces the first, whose computed data were
incorrect.
\end{abstract}

\maketitle

\section{Introduction}\label{sec:intro}

\subsection{The family}\label{sec:family}

Throughout, $q=p^f$ with $p\ge5$, $K=\Fq(t)$, and
\begin{equation}\label{eq:family}
  \Fam(q)\;=\;\bigl\{\,D\in\Fq[t]\ \text{monic},\ \deg D=5,\ \Delta_0:=4+27D^2\ \text{squarefree in }\Fq[t]\,\bigr\},
\qquad
  E_D\colon\ y^2=x^3+x+D(t).
\end{equation}
The discriminant of the Weierstrass equation is $\Delta=-16\Delta_0$, $c_4=-48$,
and $j(E_D)=6912/(4+27D^2)$ is non-constant, so $E_D$ is non-isotrivial with
$\End_{\Kbar}(E_D)=\Z$; no curve in the family has complex multiplication.
The condition that $\Delta_0$ be squarefree is the only one imposed.  Exactly
\begin{equation}\label{eq:counts}
  |\Fam(7)|=12{,}320=7^5-4{,}487,\quad |\Fam(11)|=159{,}841=11^5-1{,}210,\quad |\Fam(13)|=316{,}238=13^5-55{,}055
\end{equation}
(the loss is $\approx 2q^4$ for $q\equiv1\pmod 3$, where $\Delta_0=27(D-a)(D+a)$ with $a^2=-4/27$, $a\in\Fq$; for $q\equiv2\pmod3$ it is only of order $q^3$, because then $a\notin\Fq$ and a repeated factor of $\Delta_0$ needs a point $\beta\in\mathbb F_{q^2}\smallsetminus\Fq$ with $D(\beta)=a$, $D'(\beta)=0$, which for each of the $q^2-q$ such $\beta$ are four independent linear conditions on $D$; by enumeration the loss is $100$, $1{,}210$, $4{,}607$, $11{,}615$ at $q=5,11,17,23$; Lemma~\ref{lem:evendeg}; ancillary \texttt{enum\_family.gp}).

Let $\pi\colon X_D\to\Pp_{\Fq}$ be the minimal regular model, $\Xbar_D$ its base change to $\Fqbar$, $O$ the zero section and $F$ the class of a fibre.  For a place $v$ of $K$ let $a_v\in\Z$ be the local trace ($q_v+1-\#E_v(k_v)$ at good $v$, $\pm1$ at split/non-split multiplicative $v$, $0$ at additive $v$), and
\[
  L(E_D,T)=\prod_{v\ \mathrm{good}}(1-a_vT^{\deg v}+q_vT^{2\deg v})^{-1}\prod_{v\ \mathrm{bad}}(1-a_vT^{\deg v})^{-1},\quad
  \Lunit(z):=z^8L(E_D,1/(qz)).
\]
$E_8$ denotes the positive definite even unimodular lattice of rank $8$; $O(E_8)=\Aut(E_8)=\WE$ has order $696{,}729{,}600$ and $112$ conjugacy classes with $106$ distinct characteristic polynomials.  For $w\in\WE$ we write $\mult_1(w)$ for the multiplicity of the eigenvalue~$1$, $E_8^w=\ker(1-w)$ for the fixed sublattice, $\chi_w(z)=\det(z-w)$, $p_d(w)=\tr(w^d)$, and
$\detp(1-w)=\prod_{\lambda\ne1}(1-\lambda)$ over the eigenvalues $\lambda\neq1$ of $w$.

\subsection{Main results}\label{sec:main}

All statements below are proved in Sections~\ref{sec:structure}--\ref{sec:arithstat} and~\ref{sec:uniform}; the labels (a), (b), \dots\ used in the proofs refer to the numbering here.  Each result is marked \emph{proven}, \emph{computer-assisted} (a finite computation that is part of the proof, with the scripts shipped as ancillary files), \emph{computed} (data, not a theorem) or \emph{conjectured}, and its header says whether it is known, assembled from known results, or new; the attributions are collected in \S\ref{sec:prior}.

\begin{thmA}[structure; proven; known, after Shioda, Oguiso--Shioda and Ulmer]\label{thm:A}
Let $D\in\Fam(q)$.
\begin{enumerate}[label=\textup{(\arabic*)},leftmargin=2em]
\item \textup{(Local structure.)} The finite bad places of $E_D$ are the ten (counted with degree) zeros of $\Delta_0$; at each, $E_D$ has multiplicative reduction of Kodaira type $\mathrm I_1$ with conductor exponent $1$, one component, Tamagawa number $c_v=1$ and $a_v=\pm1$.  At $v=\infty$ the reduction is additive of type II with conductor exponent $2$, $c_\infty=1$, $a_\infty=0$.  Hence $\deg N(E_D)=12$, $\deg L(E_D,T)=8$, the Euler number of $\Xbar_D$ is $12$, $\chi(\calO_{X_D})=1$, $\omega=dx/2y$ has degree $1$, $\Xbar_D$ is a rational surface, all Tamagawa numbers are $1$, and $E_D(\Kbar)_{\mathrm{tors}}=0$.
\item \textup{(N\'eron--Severi and Mordell--Weil lattices; $L$ as a $\WE$ characteristic polynomial.)} $\NS(\Xbar_D)=\Pic(\Xbar_D)$ is free of rank $10$ and unimodular, $\NS(\Xbar_D)=T\oplus T^\perp$ with $T=\langle O,F\rangle$ and $T^\perp\cong E_8(-1)$.  The map $P\mapsto\varphi(P):=\bar P-O-(\bar P\cdot O+1)F$ is a Galois-equivariant isomorphism $E_D(\Kbar)\to T^\perp$ under which Shioda's height pairing $\langle P,Q\rangle=-\varphi(P)\cdot\varphi(Q)$ makes $E_D(\Kbar)\cong E_8$; $\langle P,P\rangle=2+2\bar P\cdot O$, and the $240$ roots are the sections $(x(t),y(t))$ with $x,y\in\Fqbar[t]$, $\deg x\le2$, $\deg y\le3$.  The geometric Frobenius acts on $T^\perp\cong E_8$ by an element $w_D\in O(E_8)=\WE$, well defined up to conjugacy, and
\begin{equation}\label{eq:LdetW}
  L(E_D,T)=\det(1-q\,w_D\,T\mid E_8),\qquad \Lunit(z)=\chi_{w_D}(z),
\end{equation}
so $\Lunit$ is a product of cyclotomic polynomials $\Phi_n$ with $n\in\{1,\dots,10,12,14,15,18,20,24,30\}$ and is one of the $106$ characteristic polynomials of $\WE$.  All of $H^2(\Xbar_D,\Ql(1))$ is algebraic, $H^1=H^3=0$, and the root number is $\varepsilon(E_D)=\det w_D=(-1)^{\mult_1(w_D)}$.
\item \textup{(Rank.)} $E_D(K)\cong E_8^{w_D}$ as groups, with the canonical height pairing equal to the restriction of the $E_8$ form; hence $\rank E_D(K)=\mult_1(w_D)=\ord_{T=1/q}L(E_D,T)$, and $\rho(X_D)=2+\mult_1(w_D)$: the conjectures of Birch--Swinnerton-Dyer (rank) and Tate hold for $E_D$ and $X_D$.
\item \textup{(Rank $0$.)} If $\mult_1(w_D)=0$ then $\Sha(E_D)$ is finite and $|\Sha(E_D)|=L(E_D,1/q)=\Lunit(1)=\det(1-w_D)=[E_8:(1-w_D)E_8]$; in fact $\Sha(E_D)\cong E_8/(1-w_D)E_8$.
\end{enumerate}
Part \textup{(5)} (the local invariants for all $1\le\deg D\le7$) is Theorem~\ref{thm:local} and part \textup{(6)} (the same statements for $y^2=x^3+A(t)x+B(t)$, $\deg A\le3$, $\deg B=5$) is Theorem~\ref{thm:general}.
\end{thmA}

\begin{thmAp}[the $\WE$ dictionary; proven, with a finite computation in \textup{(iv)}; assembled from Shioda, Grothendieck, Artin--Milne and Urabe]\label{thm:Ap}
Let $D\in\Fam(q)$, $w:=w_D$, $r:=\mult_1(w)$, $N:=(E_8^w)^\perp\subset E_8$ and $N^*\subset N\otimes\Q$ its dual lattice.
\begin{enumerate}[label=\textup{(\roman*)},leftmargin=2em]
\item $E_D(K)=E_8^w$ with the height pairing, so $\rank E_D(K)=r$ and $\Reg(E_D)=\disc(E_8^w)$ in the normalisation in which minimal sections have $\langle P,P\rangle=2$.
\item There are isomorphisms of finite abelian groups
\begin{equation}\label{eq:shachain}
  \Sha(E_D)\cong\Br(X_D)\cong H^1(\Gal(\Fqbar/\Fq),\Pic\Xbar_D)\cong H^1(\Zhat,E_8)=(E_8/(1-w)E_8)_{\mathrm{tors}}\cong N/(1-w)N^*,
\end{equation}
the first three canonical.  Every element of $\Sha(E_D)$ is split by the constant field extension $\mathbb F_{q^n}K/K$ with $n=\ord(w_D)\le30$, and $\Sha(E_D\times_K\mathbb F_{q^m}K)\cong(E_8/(1-w^m)E_8)_{\mathrm{tors}}$ for every $m\ge1$.
\item $|\Sha(E_D)|=\detp(1-w)/\disc(E_8^w)$; equivalently
\[
  |\Sha(E_D)|\cdot\Reg(E_D)=\detp(1-w)=\lim_{T\to1/q}\frac{L(E_D,T)}{(1-qT)^r},
\]
the refined Birch--Swinnerton-Dyer formula of Kato--Trihan with $|E_D(K)_{\mathrm{tors}}|=1$, $\prod_vc_v=1$ and Tamagawa factor $\tau=1$.
\item For each of the $112$ conjugacy classes of $\WE$ the group $(E_8/(1-w)E_8)_{\mathrm{tors}}$ is of the form $H\times H$ \textup{(Urabe \cite[Cor.~1.12]{Urabe1996MComp}; recomputed here, in agreement with his table in $111$ of the $112$ classes, \S\ref{sec:HxH})}; hence $\Sha(E_D)\cong H\times H$ for all $D\in\Fam(q)$, as the Cassels--Tate pairing also predicts.  The groups that occur are: for $r=0$: $0$, $(\Z/2)^2$, $(\Z/3)^2$, $(\Z/4)^2$, $(\Z/5)^2$, $(\Z/6)^2$, $(\Z/2)^4$, $(\Z/2)^2\times(\Z/4)^2$, $(\Z/3)^4$, $(\Z/2)^6$, $(\Z/2)^8$, so $|\Sha|\in\{1,4,9,16,25,36,64,81,256\}$; for $r=1$: $0$, $(\Z/2)^2$, $(\Z/3)^2$, $(\Z/4)^2$, $(\Z/2)^4$, $(\Z/2)^6$; for $r=2$: $0$, $(\Z/2)^2$, $(\Z/3)^2$, $(\Z/2)^4$; for $r=3,4$: $0$, $(\Z/2)^2$; for $r\ge5$: $0$.
\item \textup{($|\Sha|$ is a class function but not an $L$-polynomial invariant.)} For $r=0$ the characteristic polynomial determines the conjugacy class of $w$ (the $30$ rank-$0$ classes have $30$ distinct characteristic polynomials), hence $\Sha(E_D)$ as a group.  For $r\ge1$ this fails: the five polynomials $\Phi_1\Phi_2\Phi_4\Phi_8$, $\Phi_1^2\Phi_2^2\Phi_3\Phi_6$, $\Phi_1^2\Phi_2^2\Phi_4^2$, $\Phi_1^3\Phi_2^3\Phi_4$, $\Phi_1^4\Phi_2^4$ are each shared by two or three classes whose discriminants $\disc(E_8^w)$ are not all equal, so that $|\Sha|$ is not determined by the polynomial.
\end{enumerate}
\end{thmAp}

\begin{corollary}[a consequence of Theorem~A$'$ and the class tables of Urabe and Karras; the counts are computed]\label{cor:notLinv}
Two curves $E_D,E_{D'}\in\Fam(q)$ with the same $L$-function can have different $|\Sha|$.  Explicitly, $z^8-1=\Phi_1\Phi_2\Phi_4\Phi_8$ (rank $1$, $\detp(1-w)=8$, $7.8\%$ of $\WE$) is the characteristic polynomial of two classes of $\WE$: Carter's class $(A_7)''$, of size $43{,}545{,}600$, with $E_8^w=\Z v$, $\langle v,v\rangle=8$, $\Sha=0$, and Carter's class $(A_7)'$, of size $10{,}886{,}400$, with $E_8^w=\Z v'$, $v'$ a root, $\Sha\cong(\Z/2)^2$.  Curves with $w_D\in(A_7)'$ are exactly those with an $\Fq$-rational section $(x,y)$, $\deg x\le2$, $\deg y\le3$, which then generates $E_D(K)$.  At $q=7$ the $1{,}274$ curves of type $z^8-1$ split as $1{,}120$ in $(A_7)''$ and $154$ in $(A_7)'$; at $q=11$ a sample of $512$ such curves splits $431:81$ \textup{(Section~\ref{sec:arithstat})}.  At rank $0$, by contrast, $|\Sha(E_D)|=L(E_D,1/q)$ is determined by the $L$-function \textup{(Theorem~A(4))}.
\end{corollary}

Two points put Corollary~\ref{cor:notLinv} in context.  Over $\Fq(t)$, ``the same $L$-function'' means the same polynomial of degree~$8$, not the same local factors, and such curves need not be isogenous.  Even over $\Q$, the Birch--Swinnerton-Dyer formula fixes only $|\Sha|\cdot\Reg\cdot\Omega\cdot\prod_pc_p/|E(\Q)_{\mathrm{tors}}|^2$ across an isogeny class, not $|\Sha|$.

\begin{thmB}[monodromy; computer-assisted proof; new]\label{thm:B}
Let $U=\A^5_{\Z[1/6]}\smallsetminus V(\Delta)$ be the parameter scheme of monic $D=t^5+a_4t^4+\dots+a_0$ with $\Delta=\disc_t(4+27D^2)\neq0$, so that $U(\Fq)=\Fam(q)$, and let $\rho\colon\pi_1(U)\to\Aut(E_8)=\WE$ be the monodromy representation on the geometric Mordell--Weil lattice.  Write $G_{\mathrm{geom}}(p)=\rho(\pi_1(U_{\overline{\mathbb F}_p}))$ and $G_{\mathrm{arith}}(\Fq)=\rho(\pi_1(U_{\Fq}))$.
\begin{enumerate}[label=\textup{(\alph*)},leftmargin=2em]
\item For $p\in\{5,7,11,13,23\}$: $G_{\mathrm{geom}}(p)=G_{\mathrm{arith}}(\mathbb F_{p^f})=\WE$ for all $f\ge1$.
\item $\rho(\pi_1(U_\Q))=\rho(\pi_1(U_{\overline\Q}))=\WE$.
\item For every prime $p\ge5$: $G_{\mathrm{geom}}(p)=\WE$, hence $G_{\mathrm{arith}}(\Fq)=\WE$ for every $q$ of characteristic $\ge5$.
\end{enumerate}
\end{thmB}

The certificate for (a) consists of five or six explicit $D$ per prime whose Frobenius elements $w_D$, computed exactly from the $240$ minimal sections, have characteristic polynomials outside the characteristic-polynomial sets of each of the nine conjugacy classes of maximal subgroups of $\WE$ \cite[\S3]{JKZ}, together with $\det w_D$ of both signs.  This Frobenius-exclusion method is standard: it is used for $\WE$ by V\'arilly-Alvarado--Zywina \cite{VAZ} and by Jouve--Kowalski--Zywina \cite[\S3]{JKZ}.  Part (c) uses two explicit lines in $U$ and Grothendieck's specialisation theorem for the tame fundamental group of curves.

\begin{thmC}[exact trace identities; proven; the Lefschetz trace formula applied to $E_D$]\label{thm:C}
For a place $v$ and $m\ge1$ put $a_v^{(m)}:=\alpha_v^m+\beta_v^m$ at good $v$ (the $\alpha_v,\beta_v$ being the Frobenius eigenvalues of the fibre), $a_v^{(m)}:=a_v^m=(\pm1)^m$ at the $\mathrm I_1$ places, and $a_\infty^{(m)}:=0$.  For $D\in\Fam(q)$ and every $d\ge1$ the Lefschetz sum
\begin{equation}\label{eq:lefschetz}
  S_d(D):=\sum_{t_0\in\Pp(\mathbb F_{q^d})}a_{t_0}^{(d)}=\sum_{v:\ \deg v\mid d}\deg(v)\,a_v^{(d/\deg v)}\qquad\text{satisfies}\qquad S_d(D)=-q^d\,p_d(w_D)=-q^d\tr(w_D^d).
\end{equation}
In particular $\sum_{t_0\in\Fq}a_{t_0}(D)=-q\tr(w_D)$ exactly, the $\mathrm I_1$ points counted as $\pm1$.  For $d\ge2$ the mean of $a_v$ over the places of degree exactly $d$ is \emph{not} a function of $w_D$; for $d=2$,
\begin{equation}\label{eq:d2}
  2\sum_{\deg v=2}a_v=-q^2p_2(w_D)-\sum_{\deg v=1,\ \mathrm{good}}(a_v^2-2q)-\#\{\mathrm I_1\text{ places of degree }1\}.
\end{equation}
\end{thmC}

\begin{thmD}[traces by rank as Chebotarev in $\WE$]\label{thm:D}
Let $\bar a_1(D):=q^{-1}\sum_{\deg v=1,\ v\ne\infty}a_v(D)$, and for $r\ge0$ with $N_r(q):=\#\{D\in\Fam(q):\rank E_D=r\}>0$ let $M_r(q)$ be the mean of $\bar a_1$ over the rank-$r$ curves and $M_r^{(d)}(q)$ the mean of $S_d(D)/q^d$.
\begin{enumerate}[label=\textup{(\roman*)},leftmargin=2em]
\item \textup{(Exact mechanism; proven, immediate from Theorems~A, C.)} $M_r(q)=-\E_q[\tr w_D\mid\mult_1(w_D)=r]$ and $M_r^{(d)}(q)=-\E_q[p_d(w_D)\mid\mult_1(w_D)=r]$, where $\E_q$ is the average over $\Fam(q)$; the same holds with the rank condition replaced by any condition that is a class function of $w_D$ (the $L$-polynomial, $\varepsilon$, $\Sha$ as a group, the Mordell--Weil lattice).
\item \textup{(Equidistribution for fixed $p$; proven, from Theorem~B and Deligne's theorem; the constants $\mu_r$ are a finite computation in $\WE$.)} For every $p\ge5$ and every conjugacy class $C\subset\WE$, as $q=p^f\to\infty$,
\begin{equation}\label{eq:cheb}
  \frac{\#\{D\in\Fam(q):w_D\in C\}}{|\Fam(q)|}=\frac{|C|}{|\WE|}+O_p(q^{-1/2}),
\end{equation}
with a constant depending on $p$ only through the cohomology of the $\WE$-cover of $U_{\overline{\mathbb F}_p}$ (Theorem~\ref{thm:effcheb}).  Consequently every conditional class-function average converges with the same rate; in particular, with $\mu_r:=\E_{\WE}[\tr w\mid\mult_1(w)=r]$,
\begin{gather*}
  M_r(q)\to-\mu_r,\qquad \mu_0=-\tfrac{8}{29}\approx-0.27586,\qquad \mu_1=-\tfrac{230838}{2695895}\approx-0.08563,\\
  \mu_2=\tfrac{2715644}{6504289}\approx+0.41752,\qquad \mu_3=\tfrac{10794}{9799}\approx+1.10154,\qquad\mu_4=\tfrac{122672}{62299}\approx+1.96908,
\end{gather*}
so the rank-$0$ mean of the degree-$1$ trace tends to $+0.27586$ and the rank-$1$ mean to $+0.08563$ (both positive, rank $0$ above rank $1$); the rank distribution tends to $P_W(\mult_1=r)=0.30953$, $0.46432$, $0.18671$, $0.03544$, $0.00376$, \dots\ for $r=0,\dots,4,\dots$, and the average rank to $121/126$; the $|\Sha|$-stratified rank-$0$ means tend to $-\E_W[\tr w\mid\mult_1=0,\det(1-w)=s]=-0.15540$, $+0.69523$, $+1.12821$, $+1.91231$, $+2$, $+3$, $+4.05594$, $+4$, $+8$ for $s=1$, $4$, $9$, $16$, $25$, $36$, $64$, $81$, $256$.  The full tables of constants are in Section~\ref{sec:murm}.
\end{enumerate}
\end{thmD}

Theorem~D(ii) is a statement about $q=p^f\to\infty$ with $p$ fixed, with an implied constant that depends on $p$.  Theorem~U below makes the constant uniform in $p$, so that \eqref{eq:cheb} and the limits of Theorem~D(ii) hold along all prime powers, in particular along the primes.  Its constant is astronomically large, however, so neither theorem makes a statement about any fixed $q$.  In particular neither says anything about the data at $q=7,11,13,23$ in Sections~\ref{sec:murm}--\ref{sec:arithstat}, which remain descriptive (Remark~\ref{rem:uniformp}).  The first-order behaviour of the data is the subject of Theorem~M and Conjecture~\ref{conj:first}.

\begin{thmE}[exact arithmetic statistics]\label{thm:E}
\textup{(Proven; immediate from Theorems~A$'$, D(ii) and~U.)}  The joint law over $\Fam(q)$ of the rank, of $\Sha(E_D)$ as a group, of $E_D(K)$ up to isometry and of $\Reg(E_D)$ is the push-forward of the empirical Frobenius class distribution under
\[
  w\longmapsto\bigl(\mult_1 w,\ (E_8/(1-w)E_8)_{\mathrm{tors}},\ E_8^w,\ \disc E_8^w\bigr),
\]
and converges to the push-forward of the uniform measure on $\WE$, with the rate of \eqref{eq:cheb}, as $q=p^f\to\infty$ for fixed $p$ and, by Theorem~U, as $q\to\infty$ through all prime powers of characteristic $\ge5$; the limiting proportions are read off the class table (Tables~\ref{tab:rank0classes}--\ref{tab:ranksummary}).  \textup{(Computed.)} For all $12{,}320$ curves of $\Fam(7)$ and seeded uniform samples of $7{,}500$ and $2{,}000$ curves of $\Fam(11)$ and $\Fam(13)$, the conjugacy class of $w_D$, $\Sha(E_D)$ as a group, $E_D(K)$ with its Gram matrix and $\Reg(E_D)$ were computed exactly from the Frobenius permutation of the $240$ minimal sections \textup{(Section~\ref{sec:arithstat}, Appendix~\ref{app:comp})}.
\end{thmE}

\begin{thmU}[uniformity in $p$; proven, using Theorem~B(c); a standard application of Katz's bounds on sums of Betti numbers and of Deligne's theorem]\label{thm:U}
Let $\tilde V\to U$ be the $\WE$-torsor over $\Z[1/6]$ of Proposition~\ref{prop:torsor} and $B:=\sup_{p\ge5}\sum_{i\le9}\dim H^i_c(\tilde V_{\overline{\mathbb F}_p},\Ql)$.  Then $B\le3\cdot19^{119}$.  Put $K:=\max(2B,6|\WE|)$.  For every $q=p^f$ with $p\ge5$ and every conjugacy class $C\subset\WE$,
\begin{gather*}
  \Bigl|N_C(q)-\tfrac{|C|}{|\WE|}|\Fam(q)|\Bigr|\le\tfrac{|C|}{|\WE|}\,B\,q^{9/2},\qquad q^5-16q^4\le|\Fam(q)|\le q^5,\\
  \Bigl|\tfrac{N_C(q)}{|\Fam(q)|}-\tfrac{|C|}{|\WE|}\Bigr|\le K\,\tfrac{|C|}{|\WE|}\,q^{-1/2}.
\end{gather*}
Let $d:=Kq^{-1/2}$.  If $d<1$, then for every class function $\varphi$ and every non-empty union $S$ of classes, $N_S(q)>0$ and $|\E_q[\varphi\mid S]-\E_W[\varphi\mid S]|\le2\|\varphi\|_\infty d/(1-d)$.  Hence \eqref{eq:cheb}, the limits of Theorem~D(ii) and the limit laws of Theorem~E hold as $q\to\infty$ through all prime powers of characteristic $\ge5$, in particular along the primes.
\end{thmU}

\begin{thmM}[moments to first order; proven, uniform in $p$; exact moment identities and parametrisations of the boundary, and for the second part of \textup{(f)} the Grothendieck--Ogg--Shafarevich formula and Deligne's purity theorem]\label{thm:M}
Let $q=p^f$ with $p\ge5$ and $c\in\Fq^\times$.  Let $\Fam^{(c)}(q)$ be the set of monic $D$ of degree $5$ with $4c^3+27D^2$ squarefree (so $\Fam^{(1)}=\Fam$), and $w_D\in\WE$ the Frobenius of $y^2=x^3+cx+D(t)$ (Theorem~\ref{thm:general}).  Let $\chi_q$ be the quadratic character of $\Fq$, $\chi:=\chi_q(-3)$, $\chi_Z:=\chi_q(-3c)$, $a_b:=-\sum_{x\in\Fq}\chi_q(x^3+cx+b)$ and $M_3:=\sum_{b\in\Fq}a_b^3$.  Then, with absolute implied constants and sums over $D\in\Fam^{(c)}(q)$:
\begin{enumerate}[label=\textup{(\alph*)},leftmargin=2em]
\item $|\Fam^{(c)}(q)|=q^5-(1+\chi_Z)q^4+O(q^3)$;
\item $\sum\mathrm{fix}(w_D)=q^5-(3+2\chi_Z)q^4+O(q^3)$, where $\mathrm{fix}(w)$ is the number of roots fixed by $w$, i.e.\ of $\Fq$-rational sections $(x,y)$ with $\deg x\le2$, $\deg y\le3$;
\item $\sum\tr w_D=O(q^3)$;
\item $\sum(\tr w_D)^2=q^5-(2+\chi+2\chi_Z)q^4+O(q^3)$;
\item $\sum\tr(w_D^2)=q^5-(2-\chi+2\chi_Z)q^4+O(q^3)$;
\item $\sum(\tr w_D)^3=-q^2M_3+O(q^3)$.  Moreover $M_3=-\tr(\Frob_q\mid H^1_!)-2q\sum_{4c^3+27b^2=0}a_b$, where $H^1_!$, the compactly supported $H^1$ over $\A^1_{\Fqbar}$ of $j_*\operatorname{Sym}^3R^1$ for the family $y^2=x^3+cx+b$ over the $b$-line, is $2$-dimensional and pure of weight $4$.
\end{enumerate}
\end{thmM}

Dividing by (a), the $1/q$ terms of $\E_q[\mathrm{fix}(w_D)]$, $\E_q[(\tr w_D)^2]$ and $\E_q[\tr(w_D^2)]$ depend on $q$ only through two quadratic characters.  The first, $\chi_q(-3c)$, records whether the two components of the discriminant locus of $4c^3+27D^2$ are defined over $\Fq$.  The second, $\chi_q(-3)$, records whether $\Fq$ contains the cube roots of unity.  For $c=1$ both are $\chi_q(-3)$: this is the dependence on $q\bmod3$ visible in the data.  The third moment is not periodic in $q$.  For $c=1$ and all primes $5\le p<3000$, $\tr(\Frob_p\mid H^1_!)=\chi_p(-3)\,p\,\lambda_p$, where $\lambda_p$ is the $p$-th coefficient of the CM newform $\eta(4z)^6$ of weight~$3$ \textup{(computed, not proved; Remark~\ref{rem:cm})}.  Conjecture~\ref{conj:first} extends Theorem~M to the whole class distribution:
\begin{gather*}
  N_C(q)=\frac{|C|}{|\WE|}\,q^5\Bigl(1-\frac{\gamma^0_C(q)}q\Bigr)+O(q^{7/2}),\\
  \gamma^0_C(q)=\pi_{120}(C)+\chi_q(-3c)\,\chi_{\operatorname{Sym}^2V}(C)+\chi_q(-3)\,\chi_{\Lambda^2V}(C)+q^{-2}M_3\,\chi_{\Lambda^3V}(C).
\end{gather*}
Here $V$ is the reflection representation and $\pi_{120}(C)$ is the number of reflections commuting with an element of $C$.  Heuristically (\S\ref{sec:conj}), the first term comes from the $120$ root hyperplanes and the second from the two components of the discriminant locus.  The third comes from an inertia element at infinity that is regular in Springer's sense, and the fourth, the CM term, from transcendental cohomology.  The conjecture is supported by exact class counts at all twelve prime powers $5\le q\le41$ and at six pairs $(q,c)$ with $c\ne1$.  These include out-of-sample predictions at $q=37$ and $41$; several of the $q=41$ predictions for rank statistics missed, and $q=5$ and the $q\equiv1\pmod5$ values fit less well (\S\ref{sec:conj}).  The gap to a theorem is identified there.

\subsection{Two exact symmetries}\label{sec:symm}
Two facts hold exactly at every $q$: for $p\ge7$ every class count $N_C(q)=\#\{D\in\Fam(q):w_D\in C\}$ is divisible by $q$, and for $q\equiv3\pmod4$ one has $N_C(q)=N_{-C}(q)$, where $-C$ is the class of $-w$ for $w\in C$ (\S\ref{sec:twist}, Proposition~\ref{prop:twist}; e.g.\ $\Phi_{15}\leftrightarrow\Phi_{30}$, $\Phi_3\Phi_9\leftrightarrow\Phi_6\Phi_{18}$).

\subsection{Relation to prior work}\label{sec:prior}
\emph{Murmurations.} Murmurations were discovered over $\Q$ by He--Lee--Oliver--Pozdnyakov \cite{HLOP2022} and proved for weight-$2$ newforms by Zubrilina \cite{Zub} and in the weight aspect by Bober--Booker--Lee--Lowry-Duda \cite{BBLLD}; see also Sawin--Sutherland \cite{SS} (curves ordered by height) and Cowan \cite{Cowan} (densities from ratios conjectures or GRH).  In all of these the conductor $N\to\infty$ with $p/N$ fixed.  Over function fields, Lal\'in--Lee--Oliver--Pozdnyakov \cite{LLOP} treat families of characters with growing conductor degree: the quadratic family by homological stability and the cubic family by averages of Gauss sums.  The dependence of the mean trace on the rank studied here is motivated by this literature, but the regime is different: the conductor degree is fixed ($\deg N=12$), the monodromy is the finite group $\WE$, there is no $p/N$ scaling and no density, and the only limit is $q\to\infty$.

\emph{Attributions.} Theorem~A is known.  Part (1) is Tate's algorithm \cite{Tate75}; (2)--(3) are the theory of Mordell--Weil lattices of Shioda \cite{Shioda1990,Shioda1989} in the form of Sch\"utt--Shioda \cite[Cor.~6.13, Cor.~8.4]{SchuettShioda} (see also \cite{SchuettShiodaBook}) and Oguiso--Shioda \cite{OguisoShioda}, together with the fact that the stabiliser of $O$ and $F$ in $O(\NS)$ is $\WE$ (cf.\ \cite[\S1]{HeckmanLooijenga}, where the stabiliser of the fibre class alone is identified with the affine Weyl group $W(\hat E_8)$), and Ulmer's identity $L=P_2/(1-qT)^2$ \cite[Lect.~3, (6.3)]{UlmerPCMI}; complete proofs are included only to pin down every normalisation (Frobenius direction, sign of the pairing, powers of $q$); (4) is the case $r=0$ of Theorem~A$'$.  In Theorem~A$'$, (i) is Galois descent in Shioda's theory; (ii) is assembled from Grothendieck's $\Sha\cong\Br$ \cite[\S4]{Groth}, Artin's $\Br(X)\cong H^1(G,\Pic\Xbar)$ for rational surfaces over finite fields \cite{Milne70} and Urabe's identification of $H^1$ with a lattice quotient \cite[Lemma~1.8, Prop.~1.9 and the Remark after Cor.~1.17]{Urabe1996MComp}; the order formula behind (iii), $|\Br(X)|\cdot|\disc\NS(X)|=\detp(1-w)$ for a rational surface, is Milne's \cite{Milne70}; the alternating pairing \cite[Thm.~1.10]{Urabe1996MComp}, the shape $H\times H$ \cite[Cor.~1.12]{Urabe1996MComp} and the table of these groups for all $112$ classes of $\WE$ in (iv) are Urabe's; and the pairs of classes with the same characteristic polynomial and different $H^1$ behind (v) are in Urabe's table and in Karras \cite[Table~1(b), Tables~II--III]{Karras2026}.  Corollary~\ref{cor:notLinv} is immediate from Theorem~A$'$ and these tables: the classes with characteristic polynomial $z^8-1$ are Carter's $(A_7)'$ and $(A_7)''$ \cite[Table~11]{Carter} (Urabe's Nos.~47 and~112, Karras's types~62 and~63), and only $(A_7)'$ fixes an exceptional vector \cite[Table~2]{Urabe1996MComp}, \cite[p.~14, Tables~II--III]{Karras2026}; the other four pairs of classes with a common Carter symbol in \cite[Table~1(b)]{Karras2026} likewise have $H^1=(\Z/2)^2$ and $0$.  Contracting $O$, $X_D$ becomes a del Pezzo surface $Y$ of degree~$1$, whose $\WE$-valued Frobenius is the object studied by Banwait--Fit\'e--Loughran \cite{BFL}; the map $\lambda\mapsto\lambda+K_Y$ takes exceptional vectors to roots, and Frobenius-fixed exceptional vectors correspond to $\Fq$-rational minimal sections of $E_D$.  Theorem~C is the Grothendieck--Lefschetz trace formula \cite[Rapport]{SGA4h}, \cite[(1.12.1)]{DeligneWeilI} for $X_D$, read along the fibration.  Theorem~D(i) is immediate from Theorems~A(3) and~C, and D(ii) is Deligne's equidistribution theorem in the finite-monodromy case \cite{DeligneWeilII}, \cite[Ch.~9]{KatzSarnak}, applied with Theorem~B.  Its closest antecedent is Banwait--Fit\'e--Loughran's $\WE$-equidistribution for all degree-$1$ del Pezzo surfaces \cite[Thm.~1.9, case $d=1$]{BFL}, with the rate $O(q^{-1/2})$ from Ekedahl's effective Chebotarev lemma \cite[Lem.~1.2]{Ekedahl}, used in \cite[Prop.~6.2 and its proof]{BFL}, and the limiting law of $\tr w$ \cite[Cor.~1.10]{BFL}.  Our family is a $4$-parameter slice of theirs; their theorems do not apply to the slice directly, which is why Theorem~B is needed.  The constants $\mu_r$ and the $|\Sha|$-stratified constants are a finite computation in the class table of $\WE$; \cite[Cor.~1.10]{BFL} gives the unconditional law of $\tr w$.  Up to $t\mapsto t+c$ ($p\ne5$), the family $y^2=x^3+A(t)x+B(t)$, $\deg A\le3$, $\deg B=5$, of Theorem~\ref{thm:general} is Shioda's family \cite{Shioda1991}, which is excellent with group $\WE$ in the sense of Shioda--Usui \cite{ShiodaUsui} (see also Kumar--Shioda \cite{KumarShioda}), and $\Fam(q)$ is its slice $p_2=p_8=p_{14}=0$, $p_{20}=1$, of dimension $4$.  Shioda's results are stated in characteristic~$0$ and, in positive characteristic, exclude finitely many primes \cite{Shioda1991}; Theorem~B does not rely on them.  Full monodromy of a family does not pass to its slices.

\emph{What is new.} New here are Theorem~B, the full monodromy of the slice $\Fam$ in every characteristic $p\ge5$ (Shioda's $\WE$ statement and the equidistribution of \cite{BFL} concern the full family and the full parameter space of degree-$1$ del Pezzo surfaces), whose method is standard (see the paragraph after Theorem~B), and the data, including the Frobenius class of $21{,}820$ curves computed from their minimal sections.  Everything else is known or follows directly from known results: $E_D(K)=E_8^{w_D}$ with $\Reg(E_D)=\disc E_8^{w_D}$ is Galois descent in Shioda's theory, Corollary~\ref{cor:notLinv} is immediate from Theorem~A$'$ and the class tables of Urabe and Karras, and the constants $\mu_r$ are a finite computation.  The author's earlier paper \cite{WachsQ} treats murmurations and $\Sha$ over $\Q$.

New in this version are Theorems~U and~M and Conjecture~\ref{conj:first}.  Theorem~U is a standard application of Katz's bounds on sums of Betti numbers \cite{Katz2001} and of Deligne's theorem \cite{DeligneWeilII}.  Its only input specific to the slice is Theorem~B(c); for the full family of degree-$1$ del Pezzo surfaces, the rate $O(q^{-1/2})$ with a constant independent of $q$ is stated in \cite[Prop.~6.2]{BFL}.  As far as we know, Theorem~M is new, and so is the phenomenon behind Conjecture~\ref{conj:first}: we found no earlier residue-class (or non-periodic) $1/q$ term in a $W(E_n)$ Frobenius-class distribution.  For cubic surfaces and for del Pezzo surfaces of degree~$2$, the per-class counts are polynomials in $q$ \cite{Das}, \cite{Bergvall2024}.  For cubic surfaces the moduli space of marked surfaces is minimally pure \cite[Cor.~5.7]{BergvallGounelas}, and the per-class $1/q$ coefficient is, up to normalisation, $-\chi_{H^1}(C)$, independent of $q$ \cite[\S2.1]{Das}.  For degree~$1$, Bergvall states that the moduli space of geometrically marked del Pezzo surfaces is minimally pure \cite[Cor.~4.3]{Bergvall2020}, but its per-class counts are not known.  Our slice is not of this kind.  Heuristically, its residue-class terms come from boundary components that are not defined over $\Fq$, one of them seen through Springer's theory of regular elements \cite[Thm.~1.9]{Bessis}.

\subsection{The first version}\label{sec:changed}
The first version of this paper (arXiv:2603.13802v1) reconstructed $L(E_D,T)$ from point counts with the wrong degree, so its $L$-polynomials, ranks, $|\Sha|$ tables and $z$-scores were wrong.  The corrections are listed in the ancillary file \texttt{CORRECTIONS.md}.

\subsection{Organisation}
Section~\ref{sec:background} fixes normalisations; Sections~\ref{sec:structure}--\ref{sec:arithstat} prove Theorems~A/A$'$, B, C, D and E and present the data; Section~\ref{sec:local} treats $\deg D\le7$, Section~\ref{sec:uniform} proves Theorems~U and~M and states the first-order conjecture, Section~\ref{sec:open} lists open questions, and Appendix~\ref{app:comp} describes the computations and the ancillary files.

\section{Background and normalisations}\label{sec:background}

\subsection{Conventions}\label{sec:conv}
$q=p^f$ with $p\ge5$; $K=\Fq(t)$; $\Kbar:=\Fqbar(t)=\bigcup_n\mathbb F_{q^n}(t)$; $G:=\Gal(\Fqbar/\Fq)\cong\Zhat$ with arithmetic Frobenius $\sigma\colon x\mapsto x^q$ and geometric Frobenius $\Frob:=\sigma^{-1}$.  Group cohomology of $G$ is continuous cohomology of the profinite group $\Zhat$.  Places $v$ of $K$ are closed points of $\Pp_{\Fq}$, with residue field $k_v$ of cardinality $q_v=q^{\deg v}$.  For an elliptic curve $E/K$ the local invariant $a_v$ is as in \S\ref{sec:family} \cite[(8.3)]{UlmerPCMI}, and $L(E,T)$ is the Euler product of \S\ref{sec:family}, $L(E,s):=L(E,q^{-s})$ \cite[(9.1)]{UlmerPCMI}.  $\Sha(E):=\ker\bigl(H^1(K,E)\to\prod_vH^1(K_v,E)\bigr)$.  For a non-constant $E/K$, $L(E,T)$ is a polynomial of degree $4g_C-4+\deg N(E)$ by Grothendieck--Deligne \cite[Thm.~9.3]{UlmerPCMI}, here $\deg N-4$.

\subsection{Lattices and $\WE$}\label{sec:lattices}
$E_8$ is the unique positive definite even unimodular lattice of rank $8$ \cite[Ch.~V, \S2.3]{SerreCourse}; $E_8(-1)$ is the same lattice with negated form.  Its $240$ vectors of norm $2$ form the root system $E_8$ and span the lattice, so $O(E_8)=\Aut(R)=\WE\rtimes\Aut(\text{Dynkin diagram})=\WE$, of order $696{,}729{,}600$ \cite[Ch.~4, \S8.1]{ConwaySloane}, \cite[Ch.~VI, \S1, no.~5, Prop.~16 and Planche~VII]{Bourbaki}.  $\WE$ has $112$ conjugacy classes \cite{Carter}, whose characteristic polynomials on $E_8\otimes\Q$ are products of cyclotomic polynomials $\Phi_n$, $n\in\{1,\dots,10,12,14,15,18,20,24,30\}$ (the divisors of the element orders); exactly $106$ distinct characteristic polynomials occur, and the five polynomials listed in Theorem~A$'$(v) are the only ones shared by several classes.  All of this, together with the invariants $\mult_1$, $\detp(1-w)$, $E_8^w$, $\disc E_8^w$, the Smith normal form of $1-w$ and the cycle type of $w$ on the $240$ roots, was computed exactly for a representative of each class (ancillary \texttt{we8\_classes.sage}, \texttt{weyl\_e8\_group.py}).  A class is labelled by the pair (characteristic polynomial, cycle type on the $240$ roots), which separates all $112$ classes, and numbered $1,\dots,112$ by sorting by $\mult_1$, then by decreasing class size, then by characteristic polynomial; the ancillary index table \texttt{weyl\_e8\_classes.csv} records the numbering and every invariant used below.

The \emph{trivial lattice} $T=\langle O,F\rangle\subset\NS(\Xbar_D)$ has Gram matrix $\bigl(\begin{smallmatrix}-1&1\\1&0\end{smallmatrix}\bigr)$ since $O^2=-\chi(\calO_X)=-1$; it is odd unimodular, \emph{not} the hyperbolic plane.

\subsection{Elliptic surfaces and Mordell--Weil lattices}\label{sec:MWL}
For $E/K$ non-constant, $\pi\colon X\to\Pp$ denotes the minimal regular model (Kodaira--N\'eron model), a smooth projective geometrically irreducible surface with zero section $O$.  The Kodaira type of a fibre is read off Tate's algorithm \cite{Tate75}, \cite[IV, \S9]{SilvermanATAEC}; for $p\ge5$ the conductor exponent is $1$ at multiplicative and $2$ at additive places \cite[(8.5)]{UlmerPCMI}, the Euler number of a fibre is $0$ (smooth), $m$ ($\mathrm I_m$) or $m+1$ (additive with $m$ components) \cite[\S6.7]{SchuettShioda}, and $e(\Xbar)=\sum_v\deg(v)\,e(X_v)$ (Ogg's formula, $e(\Xbar)=\deg\Delta_{\min}$, in the tame case; \cite[Thm.~6.10]{SchuettShioda}, \cite{CossecDolgachev}).  The Hodge bundle $\omega=(R^1\pi_*\calO_X)^{-1}$ satisfies $\deg\omega=\chi(\calO_X)$, $12\deg\omega=\deg\Delta_{\min}$, $K_X=\pi^*(K_{\Pp}\otimes\omega)$ \cite[Lect.~3, \S2--\S3, Prop.~4.3]{UlmerPCMI}, \cite[Thm.~6.8, Cor.~6.11]{SchuettShioda}, and $O^2=-\chi(\calO_X)$ \cite[Cor.~6.9]{SchuettShioda}.  A relatively minimal elliptic surface over $\Pp$ with $\chi(\calO_X)=1$ is rational, with $\chi=2$ a K3 surface \cite[Lect.~3, \S4]{UlmerPCMI}, \cite[\S8, \S12]{SchuettShioda}.

\emph{Shioda--Tate.} $P\mapsto\bar P\bmod T_{\mathrm{full}}$ is an isomorphism $E(\Kbar)\cong\NS(\Xbar)/T_{\mathrm{full}}$ where $T_{\mathrm{full}}$ is spanned by $O$, $F$ and the fibre components not meeting $O$ \cite[Thm.~1.3]{Shioda1990}, \cite[Thm.~6.3]{SchuettShioda}, \cite[Lect.~3, Thm.~5.1]{UlmerPCMI}; $E(\Kbar)_{\mathrm{tors}}\cong T'_{\mathrm{full}}/T_{\mathrm{full}}$ with $T'$ the primitive closure \cite[Cor.~7.2]{SchuettShioda}.  The height pairing $\langle P,Q\rangle=-\varphi(P)\cdot\varphi(Q)$, $\varphi(P)$ the orthogonal projection of $\bar P$ to $T_{\mathrm{full}}^\perp$, is Shioda's pairing \cite[Def.~8.5, Thm.~8.6]{Shioda1990}, \cite[\S11]{SchuettShioda}: explicitly $\langle P,Q\rangle=\chi+\bar P\cdot O+\bar Q\cdot O-\bar P\cdot\bar Q-\sum_v\mathrm{contr}_v(P,Q)$, with no fibre contributions when all fibres are irreducible.  Shioda proves that this pairing is the canonical (N\'eron--Tate) height pairing on $E(\Kbar)$ \cite[Thm.~8.4]{Shioda1990}; in the normalisation used here a minimal section (one with $\bar P\cdot O=0$) has $\langle P,P\rangle=2\chi$.  A rational elliptic surface with no reducible fibres has $E(\Kbar)\cong E_8$ \cite[Cor.~8.4]{SchuettShioda}, \cite[Main Thm., entry No.~1]{OguisoShioda}.

\subsection{$L$-functions and zeta functions}\label{sec:Lzeta}
Deligne's convention: $F^*$ on $H^i(\Xbar,\Ql)$ is the action of the geometric Frobenius, and $Z(X,T)=\prod_i\det(1-F^*T\mid H^i)^{(-1)^{i+1}}$ \cite[Rapport]{SGA4h}, \cite[(1.14.3)]{DeligneWeilI}, \cite[VI.12--13]{MilneEC}.  The zeta function of an elliptic surface decomposes fibrewise, $Z(X,T)=\prod_{y\in\Pp\ \mathrm{closed}}Z(X_y,T^{\deg y})$, and by \cite[Lect.~3, Prop.~6.1 and (6.3)]{UlmerPCMI}
\begin{equation}\label{eq:ZL}
  Z(X,T)=\frac{Z(\Pp,T)\,Z(\Pp,qT)}{L(E,T)\cdot\prod_v(1-q_vT^{\deg v})^{\nu_v-1}(1+q_vT^{\deg v})^{g_v}},
\end{equation}
where $\nu_v$ is the number of irreducible components of the fibre at $v$ as a $k_v$-scheme, i.e.\ of $\mathrm{Frob}_v$-orbits on its geometric components (Ulmer's $f_v$; $\nu_v=1$ for an irreducible fibre), and $g_v$ the number of these orbits of length $2$.  This form assumes that no orbit is longer, as holds for all $\mathrm I_m$ fibres and all fibres in this paper; in general the product is $\prod(1-(q_vT^{\deg v})^{\ell})$ over the orbits, of lengths $\ell$, on the components not meeting the zero section.  The functional equation $c_{8-i}=\varepsilon q^{8-2i}c_i$ of $L(E_D,T)=\sum c_iT^i$ and the Riemann hypothesis $|\alpha_j|=q$ hold by Deligne \cite{DeligneWeilII}; in our case both are visible directly from \eqref{eq:LdetW}.

\subsection{Birch--Swinnerton-Dyer and Artin--Tate, with normalisations}\label{sec:BSDAT}
\emph{Kato--Trihan.} For $E/K$ with $L(E,s)$ vanishing to order $r$ at $s=1$ and $\rank E(K)=r$, $\Sha(E)$ is finite and \cite{KT}, \cite[Conj.~6.2.5, Thm.~6.2.6, \S6.2.3]{UlmerCRM}
\begin{equation}\label{eq:KT}
  L(E,s)\sim\frac{R\cdot|\Sha(E)|\cdot\tau}{|E(K)_{\mathrm{tors}}|^2}(s-1)^r\ (s\to1),\qquad
  \tau=q^{g_X(1-g_C)}\prod_vq^{g_Xa_v}c_v,
\end{equation}
where $g_X=1$, $g_C=0$, $c_v=|\Phi_v(k_v)|$, $\pi_v^{a_v}\omega$ is a N\'eron differential at $v$ for a chosen invariant differential $\omega$ (so $\sum_va_v\deg v=-\deg\omega=-\chi(\calO_X)$), and $R$ is the regulator of the N\'eron--Tate pairing in Ulmer's normalisation, valued in $\Q\cdot\log q$.  Since $\sum a_v=-\chi$, $\tau=q^{1-\chi(\calO_X)}\prod_vc_v$.  For rank $0$:
\begin{equation}\label{eq:BSD0}
  |\Sha(E)|=L(E,1/q)\cdot|E(K)_{\mathrm{tors}}|^2\cdot q^{\chi(\calO_X)-1}\Big/\prod_vc_v .
\end{equation}

\emph{Artin--Tate.} For a smooth projective surface $X/\Fq$ satisfying Tate's conjecture $T^2$, $\Br(X)$ is finite and \cite[Thm.~6.1]{Milne1975} (with \cite[Thm.~4.1]{Milne1975} for the finiteness of $\Br(X)(\ell)$; stated there for $p\ne2$), \cite[Conj.~(C)]{TateBourbaki}, \cite[Conj.~6.1.3, Thm.~6.1.4]{UlmerCRM}
\begin{equation}\label{eq:AT}
  \lim_{s\to1}\frac{P_2(X,q^{-s})}{(1-q^{1-s})^{\rho(X)}}=\frac{|\Br(X)|\cdot|\disc\NS(X)|}{q^{\alpha(X)}\,|\NS(X)_{\mathrm{tors}}|^2},\qquad \alpha(X)=\chi(\calO_X)-1+\dim\mathrm{PicVar}(X).
\end{equation}
For a rational surface $\alpha=0$; for rational surfaces, the case used here, \eqref{eq:AT} and $\Br(X)\cong H^1(G,\Pic\Xbar)$ (due to Artin) are in \cite{Milne70}.

\emph{$\Sha$ and $\Br$.} If $E/K$ has a $K$-rational point and $k$ is finite, there is a canonical isomorphism $\Br(X)\cong\Sha(E)$ \cite[\S4]{Groth}, \cite[Thm.~3.1]{TateBourbaki}, \cite[Prop.~5.3.1]{UlmerCRM}, \cite[Lect.~3, Thm.~7.1]{UlmerPCMI}.

\subsection{Cohomology of $\Zhat$}\label{sec:Zhat}
We use: $H^1(\Zhat,\Fqbar^\times)=0$ (Hilbert 90), $H^2(\Zhat,\Fqbar^\times)=\Br(\Fq)=0$ (Wedderburn), $H^3(\Zhat,\Fqbar^\times)=0$ (because $\Fqbar^\times$ is torsion and $\mathrm{cd}(\Zhat)=1$ \cite[I \S3.2 Prop.~11, II \S3]{SerreGC}), $\Hom_{\mathrm{cont}}(\Zhat,\Z)=0$; and, for $v$ a place of $K$ with N\'eron component group $\Phi_v$, $H^1(\Gal(K_v^{\mathrm{nr}}/K_v),E(K_v^{\mathrm{nr}}))=H^1(k_v,\Phi_v)$ \cite[Prop.~I.3.8]{MilneADT}.

\section{The structure theorem: proofs of Theorems A and A$'$}\label{sec:structure}

\subsection{Local structure: proof of Theorem A(1)}\label{sec:local1}
\emph{Finite places.} The equation $y^2=x^3+x+D$ has coefficients in $\Fq[t]$, hence is integral at every finite $v$, with $v(c_4)=v(-48)=0$.  If $v(\Delta_0)=0$ the reduction is good.  If $v(\Delta_0)\ge1$ then, $\Delta_0$ being squarefree, $v(\Delta)=v(\Delta_0)=1<12$, so the model is minimal at $v$; the reduced cubic is singular with $c_4$ a unit, hence nodal: multiplicative reduction, of type $\mathrm I_n$ with $n=v(\Delta)=1$ by Tate's algorithm \cite{Tate75}, \cite[IV, \S9]{SilvermanATAEC}.  An $\mathrm I_1$ fibre is an irreducible nodal cubic (the Weierstrass model is regular at the node because $v(\Delta)=1$), so $m_v=1$, $\Phi_v=0$, $c_v=1$, $f_v=1$ ($p\ge5$), and $a_v=\pm1$ according as the node is split or not.  Counting with degree, $\sum_{v\ \mathrm{bad\ finite}}\deg v=\deg\Delta_0=10$.

\emph{The place $\infty$.} Put $t=1/s$, $D(t)=s^{-5}u(s)$ with $u(s):=s^5D(1/s)\in\Fq[s]$, $u(0)=1$.  The admissible change $x=s^{-2}X$, $y=s^{-3}Y$ gives
\begin{equation}\label{eq:modelinf}
  Y^2=X^3+s^4X+s\,u(s),
\end{equation}
with $a_4=s^4$, $a_6=su(s)$, discriminant $\Delta'=s^{12}\Delta=-16(4s^{12}+27s^2u^2)$ of valuation $v_\infty(\Delta')=2<12$: the model is minimal.  The reduction $Y^2=X^3$ is cuspidal (additive, $s\mid b_2=0$), $s\mid a_3,a_4,a_6$ and $s^2\nmid a_6$: this is the criterion for Kodaira type II \cite{Tate75}, \cite[IV, \S9]{SilvermanATAEC}, so the fibre is an irreducible cuspidal cubic, $m_\infty=1$, $\Phi_\infty=0$, $c_\infty=1$, $a_\infty=0$, and $f_\infty=2$ ($p\ge5$, tame).

\emph{Conductor, $L$-degree, Euler number, $\chi$, $\omega$.} $\deg N=10+2=12$, so $\deg L=0-4+12=8$ (\S\ref{sec:conv}).  $e(\Xbar_D)=10\cdot1+2=12$ (\S\ref{sec:MWL}).  The minimal discriminant has total degree $10+2=12=12\deg\omega$, so $\deg\omega=\chi(\calO_{X_D})=1$ (consistently with Noether's formula $12\chi=K^2+e=0+12$), and $\Xbar_D$ is rational.  Equivalently, in Ulmer's terms the height of $E_D$ is the least $h$ with $\deg a_4\le4h$, $\deg a_6=5\le6h$, i.e.\ $h=1$.  The differential $\omega=dx/2y$ is the N\'eron differential of the minimal model at every finite $v$, and at $\infty$, in the model \eqref{eq:modelinf}, $dx/2y=s^{-2}dX/(2s^{-3}Y)=s\cdot dX/2Y$; thus $a_\infty=-1$ in the notation of \eqref{eq:KT} and $\sum_va_v\deg v=-1=-\deg\omega$, as it must be.

\emph{Torsion.} $E_D(\Kbar)_{\mathrm{tors}}=0$ is proved in \S\ref{sec:NSMWL} from the lattice structure.  Two elementary confirmations: a point of order $2$ is $(x_0,0)$ with $x_0\in K$ a root of $x^3+x+D$; writing $x_0=f/g$ in lowest terms, $f^3+fg^2+Dg^3=0$ forces $g\mid f^3$, so $g\in\Fq^\times$ and $x_0\in\Fq[t]$, and then $3\deg x_0=\deg D=5$ is impossible.  A point of order $p$ would force $j(E_D)\in K^p$ \cite[Prop.~7.3]{UlmerPCMI}, but $j=6912/\Delta_0$ with $\Delta_0$ squarefree of degree $10$, so $\Delta_0$ is not a $p$-th power in $K$ and, $\Fq$ being perfect, $j\notin K^p$. \qed

\subsection{N\'eron--Severi and Mordell--Weil lattices: proof of Theorem A(2)}\label{sec:NSMWL}
\emph{Step 1: $\NS(\Xbar_D)$ and $T$.} $\Xbar:=\Xbar_D$ is a smooth projective rational surface, so $\Pic^0(\Xbar)=0$, $\NS(\Xbar)=\Pic(\Xbar)$ is free of finite rank and the intersection form is unimodular of signature $(1,\rho-1)$ (a rational elliptic surface with a section is the blow-up of $\mathbb P^2$ in the base points of a cubic pencil \cite[Prop.~8.1]{SchuettShioda}, so $\NS(\Xbar)$ is $\Z^{1,9}$ with the standard form; cf.\ \cite[\S8.8]{SchuettShioda}, \cite{Milne70}).  $b_2(\Xbar)=e(\Xbar)-2=10$ ($b_1=0$), and $\rho=10$ by Step~3.  $O^2=-\chi(\calO_X)=-1$, $O\cdot F=1$, $F^2=0$, so $T=\langle O,F\rangle$ is unimodular and $\NS(\Xbar)=T\oplus T^\perp$ (orthogonal projection to $T_\Q$ is integral on $\NS$).  $T^\perp$ has rank $8$ and is negative definite (Hodge index: $T$ carries the positive direction).  It is even: by Riemann--Roch $x^2\equiv K_X\cdot x\pmod 2$ for $x\in\NS(\Xbar)$, and $K_X=\pi^*\calO(-2+\deg\omega)=-F$ (\S\ref{sec:MWL}), so $x^2\equiv0\pmod 2$ for $x\perp F$.  An even negative definite unimodular lattice of rank $8$ is $E_8(-1)$ (\S\ref{sec:lattices}).  Hence $\NS(\Xbar)\cong\langle O,F\rangle\oplus E_8(-1)$; this is also the entry ``no reducible fibres'' of the Oguiso--Shioda classification \cite{OguisoShioda}.

\emph{Step 2: Mordell--Weil lattice.} By Shioda--Tate (\S\ref{sec:MWL}; here $T_{\mathrm{full}}=T$ since all fibres are irreducible), $P\mapsto\bar P\bmod T$ is an isomorphism $E_D(\Kbar)\cong\NS(\Xbar)/T$, and composing with the projection $\NS(\Xbar)=T\oplus T^\perp\to T^\perp$ gives $E_D(\Kbar)\cong T^\perp\cong E_8(-1)$; explicitly $\varphi(P)=\bar P-O-(\bar P\cdot O+1)F$, as $\varphi(P)\cdot F=1-1=0$ and $\varphi(P)\cdot O=\bar P\cdot O+1-(\bar P\cdot O+1)=0$.  The pairing $\langle P,Q\rangle=-\varphi(P)\cdot\varphi(Q)$ is therefore an isometry $E_D(\Kbar)\cong E_8$, positive definite, so $E_D(\Kbar)$ is torsion-free (equivalently: $T$ is unimodular, hence primitive, so $T'/T=0$).  With $c=\bar P\cdot O+1$: $\varphi(P)^2=\bar P^2+O^2+2c\cdot0-2\bar P\cdot O-2c+2c=-1-1-2\bar P\cdot O$, using $\bar P^2=O^2=-1$, $F^2=0$, $\bar P\cdot F=O\cdot F=1$; so $\langle P,P\rangle=2+2\bar P\cdot O$, and $\langle P,P\rangle=2$ iff $\bar P\cdot O=0$ iff $P$ does not meet $O$ over any point of $\Pp$.  Over a finite place this means $x(P),y(P)\in\Fqbar[t]$; over $\infty$, in the chart $X=s^2x$, $Y=s^3y$, it means that $s^2x(1/s)$ and $s^3y(1/s)$ are regular at $s=0$, i.e.\ $\deg x\le2$, $\deg y\le3$.  Hence the $240$ roots are the $240$ sections with polynomial coordinates of degrees $\le2,\le3$ (cf.\ \cite[\S10]{Shioda1990}); in particular, for every $D\in\Fam(q)$ the system $y^2=x^3+x+D$, $\deg x\le2$, $\deg y\le3$ has exactly $240$ solutions over $\Fqbar$.

\emph{Step 3: $H^2$ is algebraic; Frobenius on $\NS$.} $\Xbar$ is rational, so there is a dominant rational map $\Pp\times\Pp\dashrightarrow\Xbar$; resolving its indeterminacy by blowing up finitely many points gives a generically finite dominant morphism $f\colon Y\to\Xbar$ of degree $d$ with $Y$ a blow-up of $\Pp\times\Pp$.  By the blow-up formula $H^1(Y,\Ql)=0$ and $H^2(Y,\Ql(1))$ is spanned by algebraic classes.  Since $f_*f^*=d$ on $H^*(\Xbar)$ and $f_*$ preserves algebraic classes, every class in $H^2(\Xbar,\Ql(1))$ is algebraic and $H^1(\Xbar,\Ql)\hookrightarrow H^1(Y,\Ql)=0$; by Poincar\'e duality $H^3(\Xbar,\Ql)=0$.  (This is the argument of \cite[Lect.~2, Prop.~11.1, Prop.~13.1]{UlmerPCMI}.)  So the cycle map $\NS(\Xbar)\otimes\Ql\to H^2(\Xbar,\Ql(1))$ is a Galois-equivariant isomorphism (injective by non-degeneracy of the intersection form), $\rho(\Xbar)=b_2=10$, Tate's conjecture $T^1$ holds for $X_D$ over every finite extension of $\Fq$, and $P_2(X_D,T)=\det(1-F^*T\mid\NS(\Xbar)\otimes\Q(-1))$.

Let $F=F_X\times\mathrm{id}\colon\Xbar\to\Xbar$ be the base change of the $q$-power Frobenius endomorphism of $X$.  For an integral curve $E\subset\Xbar$, $F(E)=\sigma(E)$ (the curve with $\sigma$-twisted equations), $F|_E\colon E\to\sigma(E)$ is the relative $q$-Frobenius of degree $q$, and $F$ is finite flat of degree $q^2$; hence $F_*[E]=q[\sigma(E)]$ and, from $F_*F^*=q^2$, $F^*[\sigma(E)]=q[E]$.  Thus $F^*=q\cdot\sigma^{-1}=q\cdot\Frob$ on $\NS(\Xbar)$, where $\Frob$ acts through the transport-of-structure action of $G$ on $\Xbar=X\times_{\Fq}\Fqbar$.  This action is by isometries fixing $O$ and $F$ (both defined over $\Fq$), hence preserves $T$ and $T^\perp$ and induces $w_D\in O(T^\perp)=O(E_8)=\WE$, and
\begin{equation}\label{eq:P2}
  P_2(X_D,T)=\det(1-q\Frob\,T\mid\NS(\Xbar)\otimes\Q)=(1-qT)^2\det(1-qw_DT\mid E_8).
\end{equation}
(The choice $\Frob=\sigma^{-1}$ rather than $\sigma$ is immaterial: $w$ and $w^{-1}$ have the same characteristic polynomial, the same fixed lattice, and $(1-w^{-1})E_8=-w^{-1}(1-w)E_8=(1-w)E_8$.)

\emph{Step 4: $L(E_D,T)=P_2/(1-qT)^2$.} Apply \eqref{eq:ZL}: every fibre is an irreducible plane cubic, so all these factors are $1$ (a split nodal cubic has $Q$ points over $\mathbb F_Q$, a non-split one $Q+1-(-1)^m$ for $Q=q_y^m$, a cuspidal one $Q+1$; this gives $Z(X_y,T)=(1-a_yT+\epsilon_yq_yT^2)/((1-T)(1-q_yT))$ with $\epsilon_y=1$ for good and $0$ for bad $y$, which is \cite[Prop.~6.1]{UlmerPCMI} in the rows $\mathrm I_1$, II).  Hence $Z(X_D,T)=L(E_D,T)^{-1}Z(\Pp,T)Z(\Pp,qT)$, while the trace formula and Step~3 give $Z(X_D,T)=1/((1-T)P_2(X_D,T)(1-q^2T))$.  Comparing, $L(E_D,T)=P_2(X_D,T)/(1-qT)^2=\det(1-qw_DT\mid E_8)$, a polynomial of degree $8$ with constant term $1$ and leading coefficient $q^8\det w_D$; comparison with $c_8=\varepsilon q^8$ gives $\varepsilon(E_D)=\det w_D$.  Since the non-real eigenvalues of $w_D$ come in conjugate pairs, $\det w_D=(-1)^{\mult_{-1}(w_D)}=(-1)^{\mult_1(w_D)}$, consistent with $\varepsilon=(-1)^{\ord_{s=1}L}$.

\emph{Step 5: $\Lunit$.} $\Lunit(z)=z^8L(E_D,1/(qz))=z^8\det(1-w_D/z)=\det(z-w_D)=\chi_{w_D}(z)$.  As $w_D$ has finite order, $\chi_{w_D}\in\Z[z]$ is monic with all roots roots of unity, hence a product of $\Phi_n$'s, with $n$ dividing an element order of $\WE$ (\S\ref{sec:lattices}). \qed

\begin{remark}[the rational-surface hypothesis]\label{rem:prop13}
The argument uses that $H^2(\Xbar_D)$ is entirely algebraic, which is special to rational surfaces and fails for every elliptic surface with $p_g>0$.  For $\deg D=7$ (K3 surfaces, \S\ref{sec:local}) $\Lunit$ is in general not even integral: at $q=7$, $517{,}034$ of the $603{,}225$ monic $D$ of degree $7$ with $4+27D^2$ squarefree have $c_1=\sum_{t_0\in\mathbb F_7}a_{t_0}\not\equiv0\pmod 7$, so $c_1/q\notin\Z$ (exhaustive; ancillary \texttt{deg7\_c1.gp}).
\end{remark}

\subsection{Rank and the rank-$0$ case: proof of Theorem A(3), (4)}\label{sec:rank}
$E_D(K)=E_D(\Kbar)^G$.  The isomorphism $\varphi$ of Step~2 is $G$-equivariant ($T$ is $G$-stable and orthogonal projection commutes with isometries preserving $T$), so $E_D(K)\cong(T^\perp)^{\Frob}\cong E_8^{w_D}$ and $\rank E_D(K)=\mult_1(w_D)$.  The canonical height pairing on $E_D(K)$ is the restriction of that on $E_D(\Kbar)$ (heights are computed over $\Fqbar(t)$ with the same degree normalisation), which is Shioda's pairing (\S\ref{sec:MWL}), i.e.\ the $E_8$ form.  By \eqref{eq:LdetW}, $\ord_{T=1/q}L(E_D,T)=\mult_1(w_D)$.  Moreover $\NS(X_D)=\Pic(X_D)=\Pic(\Xbar_D)^G$ (Hochschild--Serre with $H^1(\Zhat,\Fqbar^\times)=0$ and $\Br(\Fq)=0$, \S\ref{sec:Zhat}) $=T\oplus E_8(-1)^{w_D}$, so $\rho(X_D)=2+\mult_1(w_D)$, which by \eqref{eq:P2} is the multiplicity of the eigenvalue $q$ of $F^*$ on $H^2$: Tate's conjecture $T^2$ holds for $X_D$, and the Shioda--Tate formula $\rank E_D(K)=\rho(X_D)-2-\sum_v(\nu_v-1)$ \cite[(5.2)]{UlmerPCMI} reads $\rank=\rho-2$.  This proves (3); it is also a case of \cite[Thm.~12.3]{UlmerPCMI}.

(4) is the case $r=0$ of Theorem~A$'$(ii)--(iii) proved below, where $E_8^w=0$, $\disc=1$, $\detp(1-w)=\det(1-w)=\chi_w(1)=\Lunit(1)=L(E_D,1/q)$ and $(E_8/(1-w)E_8)_{\mathrm{tors}}=E_8/(1-w)E_8$, of order $|\det(1-w)|$.  Directly from \eqref{eq:BSD0}: $|\mathrm{tors}|=1$, $\prod c_v=1$, $\chi=1$, so $|\Sha|=L(E_D,1/q)$, the power of $q$ being $q^{1-g_C}q^{\sum a_v}=q^{1-\chi}=1$. \qed

\subsection{A lattice lemma}\label{sec:latticelemma}
The following lemma is essentially Urabe's: (a), (b) and (d) are \cite[Lemma~1.6, Lemma~1.8, Prop.~1.9]{Urabe1996MComp} (stated there for his lattices $U_r$; the proofs hold for any unimodular lattice), and the order in (c) is Manin's Conjecture~31.7, proved by Milne \cite{Milne70} (cf.\ \cite[p.~250]{Urabe1996MComp}).  We include the short proofs to fix notation.
\begin{lemma}\label{lem:lattice}
Let $M$ be a unimodular integral lattice of any signature, $w\in O(M)$ of finite order, $M^w=\ker(1-w)$, $N=(M^w)^\perp$, $N^*=\Hom(N,\Z)\subset N_\Q$.  Then:
\begin{enumerate}[label=\textup{(\alph*)},leftmargin=2em]
\item the orthogonal projection $\pi_N\colon M_\Q\to N_\Q$ maps $M$ onto $N^*$, and likewise $\pi_{M^w}(M)=(M^w)^*$;
\item $(1-w)M=(1-w)N^*\subset N$;
\item $(M/(1-w)M)_{\mathrm{tors}}=N/(1-w)N^*$, of order $|\det(1-w\mid N_\Q)|/[N^*:N]=\detp(1-w)/|\disc N|=\detp(1-w)/|\disc M^w|$;
\item if $\Zhat$ acts on $M$ through $w$ (via $\Z/n$, $n=\ord w$), then $H^1(\Zhat,M)\cong N/(1-w)N^*=(M/(1-w)M)_{\mathrm{tors}}$.
\end{enumerate}
\end{lemma}
\begin{proof}
(a) $M^w$ is a kernel, hence primitive, so $N$ is primitive and $M/N$ torsion-free.  Dualising $0\to N\to M\to M/N\to0$ gives $M^*\to N^*\to\mathrm{Ext}^1(M/N,\Z)=0$, so restriction $M^*\to N^*$ is onto; as $M^*=M$ via the form and the restriction of $(x,\cdot)$ to $N$ is $(\pi_N(x),\cdot)$, $\pi_N(M)=N^*$.  Same for $M^w$.  (b) $1-w$ kills $M^w_\Q$ and preserves $N_\Q$, so $(1-w)M=(1-w)\pi_N(M)=(1-w)N^*$, and $(1-w)M\subset M\cap N_\Q=N$ by primitivity.  (c) $w$ is diagonalisable over $\C$ and $1$ is not an eigenvalue on $N_\Q$, so $1-w$ is invertible on $N_\Q$ and $(1-w)M\otimes\Q=N_\Q$; the torsion of $M/(1-w)M$ is $(M\cap N_\Q)/(1-w)M=N/(1-w)N^*$, of order $[N^*:(1-w)N^*]/[N^*:N]=|\det(1-w\mid N)|/|\disc N|$, and $\det(1-w\mid N)=\prod_{\lambda\ne1}(1-\lambda)=\detp(1-w)>0$ (conjugate pairs contribute $|1-\lambda|^2$, $\lambda=-1$ contributes $2$).  Finally $|\disc N|=|\disc N^\perp|=|\disc M^w|$ because $N$ and $N^\perp=M^w$ are mutually orthogonal primitive sublattices of a unimodular lattice \cite[Prop.~1.6.1, Cor.~1.6.2]{Nikulin}.  (d) Inflation--restriction: $0\to H^1(\Z/n,M)\to H^1(\Zhat,M)\to H^1(n\Zhat,M)=\Hom_{\mathrm{cont}}(n\Zhat,M)=0$ ($M$ discrete torsion-free), so $H^1(\Zhat,M)=H^1(\Z/n,M)=\ker(N_w)/(1-w)M$ with $N_w=\sum_{i<n}w^i$.  On $M_\Q$, $N_w=n\cdot(\text{projection to }M^w_\Q)$ (an eigenvalue $\lambda\ne1$ of order dividing $n$ has $\sum\lambda^i=0$), so $\ker N_w=M\cap N_\Q=N$ and $H^1=N/(1-w)M=N/(1-w)N^*$.
\end{proof}

\subsection{$\Sha$ as a group: proof of Theorem A$'$(ii)}\label{sec:shagroup}
We give the argument (for attributions see \S\ref{sec:prior}) in the form needed for $\Sha(E_D)$.

\emph{$\Sha(E_D)\cong\Br(X_D)$}, canonically, since $E_D$ has the rational point $O$ (\S\ref{sec:BSDAT}).

\emph{$\Br_1(X_D)\cong H^1(G,\Pic\Xbar_D)$.} The Hochschild--Serre spectral sequence $E_2^{a,b}=H^a(G,H^b(\Xbar_D,\mathbb G_m))\Rightarrow H^{a+b}(X_D,\mathbb G_m)$ gives the exact sequence of low-degree terms
\[
  0\to H^1(G,\Fqbar^\times)\to\Pic(X_D)\to\Pic(\Xbar_D)^G\to H^2(G,\Fqbar^\times)\to\Br_1(X_D)\to H^1(G,\Pic\Xbar_D)\to H^3(G,\Fqbar^\times),
\]
with $\Br_1(X_D):=\ker(\Br(X_D)\to\Br(\Xbar_D))$ and $H^0(\Xbar_D,\mathbb G_m)=\Fqbar^\times$.  By \S\ref{sec:Zhat} the outer terms vanish, so $\Pic(X_D)=\Pic(\Xbar_D)^G$ and $\Br_1(X_D)\cong H^1(G,\Pic\Xbar_D)$ (this is the standard sequence, written out in \cite[pp.~247--248]{Urabe1996MComp}).

\emph{$H^1(G,\Pic\Xbar_D)\cong H^1(G,E_D(\Kbar))\cong(E_8/(1-w)E_8)_{\mathrm{tors}}$.} $\Pic\Xbar_D=T\oplus T^\perp$ as $G$-modules with $G$ trivial on $T\cong\Z^2$, so
\[
  H^1(G,\Pic\Xbar_D)=H^1(\Zhat,\Z^2)\oplus H^1(G,T^\perp)=0\oplus H^1(\Zhat,E_8),
\]
and $T^\perp\cong E_D(\Kbar)$ $G$-equivariantly (\S\ref{sec:rank}), $G$ acting through $\Frob\mapsto w$.  Lemma~\ref{lem:lattice}(d) gives $H^1(\Zhat,E_8)=(E_8/(1-w)E_8)_{\mathrm{tors}}=N/(1-w)N^*$.

\emph{$\Br_1=\Br$: counting.} By \S\ref{sec:shaorder} below, $|\Br(X_D)|=\detp(1-w)/\disc(E_8^w)$, which by Lemma~\ref{lem:lattice}(c) equals $|(E_8/(1-w)E_8)_{\mathrm{tors}}|=|\Br_1(X_D)|$; since $\Br_1\subset\Br$ and both are finite of the same order, $\Br_1(X_D)=\Br(X_D)$.  (Alternatively $\Br(\Xbar_D)=0$, the Brauer group of a smooth projective surface over an algebraically closed field being a birational invariant with $\Br(\mathbb P^2)=0$; the counting argument avoids needing this.)  This proves the chain \eqref{eq:shachain}.

\emph{The direct description and the constant-field statement.} Let $H:=\Gal(K^{\mathrm{sep}}/\Kbar)$, so $G=\Gal(\Kbar/K)$.  Inflation--restriction gives
\[
  0\to H^1(G,E_D(\Kbar))\to H^1(K,E_D)\to H^1(\Kbar,E_D)^G .
\]
The image of an inflated class in $H^1(K_v,E_D)$ comes from
\[
  H^1(\Gal(\Fqbar K_v/K_v),E_D(\Fqbar K_v))=H^1(\Gal(K_v^{\mathrm{nr}}/K_v),E_D(K_v^{\mathrm{nr}}))=H^1(k_v,\Phi_v)=0
\]
($\Fqbar K_v=K_v^{\mathrm{nr}}$ as $k_v$ is finite; \S\ref{sec:Zhat}; $\Phi_v=0$ at every $v$ by Theorem~A(1)).  So $H^1(G,E_D(\Kbar))\hookrightarrow\Sha(E_D)$ canonically, and by the order count it is an isomorphism: \emph{$\Sha(E_D)$ consists exactly of the torsors split by the constant field extension $\Fqbar(t)$}, and the constant extension of degree $\ord(w_D)$ kills it.  Applying everything to $E_D$ over $\mathbb F_{q^m}(t)$, whose Frobenius is $w^m$, gives the last statement of (ii). \qed

\subsection{The order of $\Sha$ and the regulator: proof of Theorem A$'$(iii), (i)}\label{sec:shaorder}
Tate's conjecture $T^2$ holds for $X_D$ (\S\ref{sec:rank}), so by \eqref{eq:AT}, which for rational surfaces is \cite{Milne70}, $\Br(X_D)$ is finite and
\[
  \lim_{s\to1}\frac{P_2(X_D,q^{-s})}{(1-q^{1-s})^{\rho(X_D)}}=\frac{|\Br(X_D)|\,|\disc\NS(X_D)|}{q^{\alpha(X_D)}|\NS(X_D)_{\mathrm{tors}}|^2},
\]
with $\alpha(X_D)=\chi(\calO_{X_D})-1+\dim\mathrm{PicVar}(X_D)=1-1+0=0$ ($b_1=0$).  $\NS(X_D)=T\oplus E_8(-1)^w$ is torsion-free with $|\disc\NS(X_D)|=|\disc T|\cdot\disc(E_8^w)=\disc(E_8^w)$ and $\rho(X_D)=2+r$.  By \eqref{eq:P2}, $P_2(X_D,T)=(1-qT)^{2+r}\det(1-qwT\mid N_\Q)$ and $\det(1-qwq^{-s}\mid N_\Q)\to\det(1-w\mid N_\Q)=\detp(1-w)$.  Therefore
\begin{equation}\label{eq:ATresult}
  |\Br(X_D)|=\detp(1-w)/\disc(E_8^w),
\end{equation}
and $|\Sha(E_D)|=|\Br(X_D)|$ gives (iii); by Lemma~\ref{lem:lattice}(c) the same number is $|(E_8/(1-w)E_8)_{\mathrm{tors}}|$, closing the counting step of \S\ref{sec:shagroup}.

(i): $E_D(K)=E_8^w$ with the $E_8$ form was shown in \S\ref{sec:rank}; $\Reg(E_D)=\disc(E_8^w)$ is then the definition of the regulator in Shioda's normalisation.  For $r=1$, $E_8^w=\Z v$ with $v$ primitive and $\Reg=\langle v,v\rangle\in\{2,4,6,8,12,14,20,30\}$ according to the class (Table~\ref{tab:ranksummary} and the ancillary class table).

\emph{Consistency with Kato--Trihan and the $\log q$ normalisation.} By \eqref{eq:KT}, with $\tau=q^{1-1}\cdot1=1$ and trivial torsion, $L(E_D,s)\sim R\,|\Sha(E_D)|\,(s-1)^r$, where $R$ is the regulator of Ulmer's pairing valued in $\Q\cdot\log q$.  On the analytic side $L(E_D,s)=\det(1-q^{1-s}w\mid E_8)=(1-q^{1-s})^r\det(1-q^{1-s}w\mid N_\Q)\sim(\log q)^r(s-1)^r\detp(1-w)$.  Hence \eqref{eq:KT} is consistent with \eqref{eq:ATresult} if and only if $R=(\log q)^r\disc(E_8^w)$, i.e.\ iff Ulmer's pairing is $\log q$ times Shioda's pairing in the normalisation $\langle P,P\rangle=2$ for minimal sections (the pairing $\langle P,Q\rangle=\hat h(P+Q)-\hat h(P)-\hat h(Q)$ customarily used in the BSD regulator); the factor $\log q$ converts the degree normalisation $|\cdot|_v=q_v^{-\ord_v}$ into Ulmer's.  A check independent of conventions: for $w=1$ (rank $8$, $L=(1-qT)^8$), (ii) gives $\Sha=0$ and $E_D(K)=E_8$ of discriminant $1$, so \eqref{eq:KT} forces $R=(\log q)^8$; any normalisation with $R$ smaller by $2^8$ would force $|\Sha|=256$, contradicting (ii).  We stress that (ii) and (iii) rest on Grothendieck's $\Br\cong\Sha$, Milne's theorem \cite{Milne70} and Lemma~\ref{lem:lattice} only; Kato--Trihan is used as a cross-check, and it is what the point-count pipeline uses to compute $|\Sha|$ at rank $0$, where $R=1$ and no ambiguity arises. \qed

\subsection{$H\times H$ and the failure of $L$-invariance: proof of Theorem A$'$(iv), (v)}\label{sec:HxH}
(iv) is a finite computation (ancillary \texttt{we8\_classes.sage}): for each of the $112$ classes of $\WE$ a representative $w$ is taken as an integer matrix on the root lattice in the basis of simple roots (Gram matrix = Cartan matrix $C$, unimodular), $w^{\mathsf T}Cw=C$ is verified, and the Smith normal form of $1-w$ is computed; its invariant factors $\ne0,1$ give $(E_8/(1-w)E_8)_{\mathrm{tors}}$, and $(\text{torsion order})\cdot\disc(E_8^w)=\detp(1-w)$ is checked for every class (Lemma~\ref{lem:lattice}(c)).  In all $112$ classes the invariant factors $>1$ come in pairs $\{d,d\}$, and the list of groups in (iv) is read off the table.  The shape $H\times H$ is Urabe's theorem (\S\ref{sec:prior}); his groups are $H^1(G,U_8)=H^1(G,E_8)$, where $U_8$ is the rank-$9$ Picard lattice of a degree-$1$ del Pezzo surface.  The Cassels--Tate pairing gives the same shape from the arithmetic side: it is alternating, since the polarization of an elliptic curve comes from a rational divisor \cite[\S1, Cor.~7]{PoonenStoll}, and non-degenerate on the quotient of $\Sha$ by its maximal divisible subgroup \cite[II.5.6(a), III.9.5]{MilneADT}, which here is all of the finite group $\Sha(E_D)$.

Urabe's Table~2 lists the groups for all $112$ classes, and our table agrees with it in $111$ of them.  The exception is his No.~1.  On p.~256 of \cite{Urabe1996MComp}, discussing the classes for which the order alone does not determine the group (resolved there by \cite[Prop.~1.9]{Urabe1996MComp}), he writes: ``The exceptions are the following three cases (all are in the case of $W(R_8)=W(E_8)$): No.1 $A_7+A_1$, No.10 $2A_3+2A_1$, No.17 $2D_4(a_1)$.''  Row No.~1 of Table~2 in the supplement reads: ``$A_3+4A_1$ $|$ Index $0$ $|$ Measure$^{-1}$ $15360$ $|$ Frame symbol $1^{-3}\cdot2^4\cdot4^1$ $|$ Q $A_3+4A_1$ $|$ M $D_7$ $|$ H1 $\Z_4+\Z_4$''.  So No.~1 is the class $A_3+4A_1$ (``$A_7+A_1$'' on p.~256 is a misprint), of size $|W|/15360=45{,}360$, whose frame symbol on the rank-$9$ lattice gives characteristic polynomial $\Phi_1\Phi_2^5\Phi_4$ on $E_8$ (rank~$1$; class~$60$ of the ancillary class table); Urabe lists $H^1=\Z/4\oplus\Z/4$ for it.  An explicit element: in Bourbaki's numbering, let
\begin{gather*}
  \beta_1=\alpha_8,\qquad \beta_2=\alpha_6,\qquad \beta_3=\alpha_1+2\alpha_2+2\alpha_3+3\alpha_4+2\alpha_5+\alpha_6,\\
  \beta_4=\alpha_1+2\alpha_2+2\alpha_3+3\alpha_4+2\alpha_5+2\alpha_6+2\alpha_7+\alpha_8,
\end{gather*}
four mutually orthogonal roots orthogonal to $\langle\alpha_1,\alpha_3,\alpha_4\rangle$, and $w=s_{\alpha_1}s_{\alpha_3}s_{\alpha_4}\cdot s_{\beta_1}s_{\beta_2}s_{\beta_3}s_{\beta_4}$.  Then $\chi_w(z)=(z-1)(z+1)^5(z^2+1)$, the class of $w$ has $45{,}360$ elements, and the Smith normal form of $1-w$ on the root lattice is $(1,1,1,2,2,2,2,0)$, so $E_8/(1-w)E_8\cong\Z\oplus(\Z/2)^4$ and $H^1(\Zhat,E_8)=(\Z/2)^4$, the group used here (ancillary \texttt{checks/scripts/urabe\_no1\_explicit.sage}).

(v): the $30$ classes with $r=0$ have $30$ distinct characteristic polynomials, and the five shared polynomials all have $r\ge1$ (Table~\ref{tab:shared}).  For $\Phi_1\Phi_2\Phi_4\Phi_8$, in the model $E_8=D_8\cup(D_8+(\tfrac12)^8)$: the Coxeter element of the $A_7$ spanned by $e_1-e_2,\dots,e_7-e_8$ has fixed line $\Z\cdot(\tfrac12,\dots,\tfrac12)$, a root ($\Sha\cong(\Z/2)^2$, $\Reg=2$), while the Coxeter element of the $A_7$ spanned by $e_1-e_2,\dots,e_6-e_7,e_7+e_8$ has fixed line $\Z\cdot(1,\dots,1,-1)$ of norm $8$ (the half-vector $(\tfrac12,\dots,\tfrac12,-\tfrac12)$ has an odd number of minus signs and is not in $E_8$), so $\Sha=0$, $\Reg=8$.  In the simple-root model (ancillary \texttt{shared\_charpoly\_classes.py}), with $\alpha_0=-(\text{highest root})$, the Coxeter elements $s_1s_3s_4s_5s_6s_7s_8$ of the subsystem $\{\alpha_1,\alpha_3,\dots,\alpha_8\}$ of type $A_7$, which extends to an $A_8$ by $\alpha_0$, and $s_2s_4\cdots s_8s_0$ of $\{\alpha_2,\alpha_4,\dots,\alpha_8,\alpha_0\}$ both have characteristic polynomial $z^8-1$, with $(\disc E_8^w,\Sha)=(8,0)$ and $(2,(\Z/2)^2)$ respectively.  Carter defines the pair $(A_7)'$, $(A_7)''$ through $W(D_8)$ \cite[Lemma~27]{Carter} and lists $(A_7)'$ with class size $10{,}886{,}400$ and $(A_7)''$ with class size $43{,}545{,}600$ \cite[Table~11]{Carter}; the description in the simple-root model is ours, and it is matched with Carter's classes by the class sizes (ancillary class table): the first element lies in $(A_7)''$, the second in $(A_7)'$ (Urabe's Nos.~112 and~47; Dynkin's accents are the opposite ones \cite[pp.~257--258]{Urabe1996MComp}, and the ancillary script uses Dynkin's).  The arithmetic translation is (i)--(iii): the two curves have the same $L$-function, rank $1$, and $(|\Sha|,\Reg)=(1,8)$ versus $(4,2)$, both with $|\Sha|\cdot\Reg=8=\detp(1-w)$; $w_D$ lies in $(A_7)'$ iff it fixes a root, i.e.\ iff $E_D$ has an $\Fq$-rational minimal section. \qed

\begin{remark}[capitulation]\label{rem:capitulation}
By (ii), $\Sha(E_D/\mathbb F_{q^m}(t))\cong(E_8/(1-w^m)E_8)_{\mathrm{tors}}$.  For $m=\ord(w)$ it is trivial.  For $w=-1$ (the unique class $\Phi_2^8$: rank $0$, $\Sha\cong(\Z/2)^8$), over $\mathbb F_{q^2}(t)$ the rank becomes $8$ and $\Sha=0$.
\end{remark}

\subsection{The generalisation: proof of Theorem A(6)}\label{sec:general}
\begin{theorem}[Theorem A(6)]\label{thm:general}
Let $A,B\in\Fq[t]$ with $\deg A\le3$, $\deg B=5$, $B$ monic and $4A^3+27B^2$ squarefree.  Then Theorem~A(1)--(4), Theorem~A$'$, Theorem~C and Theorem~D(i) hold verbatim, with the same proofs, for $y^2=x^3+A(t)x+B(t)$, and no curve in this family has complex multiplication.
\end{theorem}
\begin{proof}
$\deg(4A^3+27B^2)=10$ since $\deg4A^3\le9<10$.  At a zero $v$ of $\Delta_0=4A^3+27B^2$ we have $A(v)\ne0$: otherwise $B(v)=0$ too, $A^3$ vanishes to order $\ge3$ and $B^2$ to order $\ge2$ at $v$, and $\Delta_0$ would have a multiple root.  Hence $c_4=-48A$ is a unit at every bad finite place, $v(\Delta)=1$, type $\mathrm I_1$, as in \S\ref{sec:local1}.  At $\infty$ with $x=s^{-2}X$, $y=s^{-3}Y$: $a_4=s^4A(1/s)$ has $v_\infty=4-\deg A\ge1$, $a_6=s\cdot s^5B(1/s)$ has $v_\infty=1$, $v_\infty(\Delta)=2$: type II.  Everything in \S\S\ref{sec:local1}--\ref{sec:HxH} and in Sections~\ref{sec:traces}--\ref{sec:murm} (part (i)) used only: ten $\mathrm I_1$ fibres and one II fibre, $\chi=1$, and the Weierstrass form.  Every member has non-constant $j$, hence $\End_{\Kbar}=\Z$.
\end{proof}

\section{Monodromy: proof of Theorem B}\label{sec:monodromy}

\subsection{The monodromy representation}\label{sec:monrep}
Let $Z'=\Z[1/6]$.  For $a=(a_0,\dots,a_4)$ put $D_a=t^5+a_4t^4+\dots+a_0$ and $\Delta(a)=\disc_t(4+27D_a^2)\in\Z[a]$, the discriminant of a polynomial of degree $10$ with leading coefficient $27$; $U=\A^5_{Z'}\smallsetminus V(\Delta)$.  For a field $\kappa$ of characteristic $\ne2,3$ and $a\in U(\kappa)$, $4+27D_a^2$ is squarefree of degree $10$, so $E_{D_a}/\kappa(t)$ has ten $\mathrm I_1$ places (over $\bar\kappa$) and a type II place at $\infty$, and $U(\Fq)=\Fam(q)$.  Theorem~A(2) holds over any algebraically closed field of characteristic $0$ or $\ge5$ (its proof used only the fibre configuration), so for $a\in U(\kappa)$ the geometric Mordell--Weil lattice of $E_{D_a}$ is $E_8$ and $\Gal(\bar\kappa/\kappa)$ acts on it through $\WE$.

Let $f\colon\mathcal X\to U$ be the family of Weierstrass models (a hypersurface in a $\mathbb P(1,2,3)$-bundle over $U\times\Pp$, smooth and projective over $U$).  For $\ell$ invertible on the base the sheaf $\calF_\ell=R^2f_*\Z_\ell(1)$ is lisse on $U[1/\ell]$ with stalk $H^2(\mathcal X_{\bar u},\Z_\ell(1))$.  Let $\eta$ be the generic point of $U$, $K_U=\Q(a_0,\dots,a_4)$, and consider the Galois module $\NS(\mathcal X_{\bar\eta})=T\oplus M$, $M\cong E_8(-1)$.  The action of $\Gal(K_U^{\mathrm{sep}}/K_U)$ preserves the intersection form and the classes $O,F$, hence factors through $\Aut(M)=\WE$.  For a rational surface over an algebraically closed field the cycle map $\NS\otimes\Z_\ell\to H^2(\Z_\ell(1))$ is a Galois-equivariant isomorphism (both sides free of rank $10$ and $\Br(\Xbar)=0$, by Tate's sequence $0\to\NS\otimes\Z_\ell\to H^2(\Z_\ell(1))\to T_\ell\Br\to0$), so the Galois action on $\NS(\mathcal X_{\bar\eta})\otimes\Z_\ell$ is the action on the generic stalk of $\calF_\ell|U[1/\ell]$ and factors through $\pi_1(U[1/\ell],\bar\eta)$; taking two different $\ell$ and using Zariski--Nagata purity on the regular scheme $U$, the action on $\NS(\mathcal X_{\bar\eta})$ factors through $\pi_1(U,\bar\eta)$.  This defines
\[
  \rho\colon\pi_1(U,\bar\eta)\to\Aut(\NS(\mathcal X_{\bar\eta}),\langle\,,\rangle,O,F)=O(E_8)=\WE .
\]
For a $Z'$-scheme $S$ write $\rho_S$ for the composite $\pi_1(U_S)\to\pi_1(U)\to\WE$ and $G(S)=\rho_S(\pi_1(U_S))$ (well defined up to conjugacy); $G_{\mathrm{arith}}(\Fq)=G(\Fq)$, $G_{\mathrm{geom}}(p)=G(\overline{\mathbb F}_p)$, and $G_{\mathrm{geom}}(p)\subseteq G_{\mathrm{arith}}(\Fq)$ for every $q=p^f$.  For $u=a\in U(\Fq)$ the geometric Frobenius class $\mathrm{Fr}_u\subset\pi_1(U_{\Fq})$ acts on the stalk $\NS(\mathcal X_{a,\Fqbar})\otimes\Z_\ell$ as the geometric Frobenius of $\mathcal X_a/\Fq$ (compatibility of cospecialisation with Frobenius, \cite[Rapport]{SGA4h}, \cite[Ch.~9]{KatzSarnak}), hence $\rho(\mathrm{Fr}_u)$ is the class of $w_a$ of Theorem~A, with $\chi_{\rho(\mathrm{Fr}_u)}=\Lunit(E_{D_a})$ and $\det\rho(\mathrm{Fr}_u)=\varepsilon(E_{D_a})$.

\subsection{Group-theoretic input}\label{sec:groupinput}
$\WE$ is realised exactly in GAP (inside Sage 10.8) as the group generated by the simple reflections on the root lattice in the basis of simple roots and, faithfully, as a permutation group $P$ of degree $240$ on the roots; every $g\in P$ is converted back to an integer matrix $M_g$ and $M_g^{\mathsf T}CM_g=C$ is checked.  The normal subgroups of $\WE$ are $1$, $\{\pm1\}$, $W^+(E_8)=\ker\det$ (index $2$) and $\WE$; the only proper normal subgroup with cyclic quotient is $W^+(E_8)$, and $\WE^{\mathrm{ab}}=\Z/2$.  GAP's \texttt{MaximalSubgroupClassReps} returns nine conjugacy classes of maximal subgroups, of indices $12096$, $11200$, $2025$, $1575$, $1120$, $960$, $135$, $120$, $2$; Jouve--Kowalski--Zywina \cite[\S3]{JKZ} also find nine classes, with the same indices (computed in Magma; cf.\ the entry for $O_8^+(2)$ in the ATLAS \cite{ATLAS}).  For each representative $M_j$ the multiset $\{\chi_g:g\in M_j\}$ was computed from \texttt{ConjugacyClasses($M_j$)}; Table~\ref{tab:maximal} lists them.  Every maximal subgroup contains $-1$ (as it must: $-1\in W^+$, so a maximal $M\not\ni-1$ would satisfy $M\{\pm1\}=\WE$, forcing $M=W^+$).  No maximal subgroup realises all $106$ characteristic polynomials, so characteristic polynomials alone can exclude each of them: a Coxeter element ($\chi=\Phi_{30}$) lies in no maximal subgroup except $W^+$, and an element with $\det=-1$ lies outside $W^+$.

\begin{table}[ht]\centering\small

\begin{tabular}{@{}rrrrrrl@{}}
\toprule
$M$ & order & index & \#classes & $\chi$'s realised & $\chi$'s missing & description \\
\midrule
$M_1$ & 348,364,800 & 2 & 83 & 61 & 45 & $W^+(E_8)=2.O_8^+(2)$ (rotation subgroup) \\
$M_2$ & 5,806,080 & 120 & 120 & 80 & 26 & $W(E_7)\times W(A_1)$ \\
$M_3$ & 5,160,960 & 135 & 100 & 72 & 34 & $W(D_8)=2^7{:}S_8$ \\
$M_4$ & 725,760 & 960 & 60 & 55 & 51 & $\{\pm1\}\times W(A_8)=2\times S_9$ \\
$M_5$ & 622,080 & 1,120 & 150 & 84 & 22 & $\{\pm1\}\times W(E_6)\times W(A_2)$ \\
$M_6$ & 442,368 & 1,575 & 148 & 79 & 27 & $N(W(D_4\times D_4))=(W(D_4)\times W(D_4)).(S_3\times 2)$ \\
$M_7$ & 344,064 & 2,025 & 79 & 44 & 62 & $N(W(A_1^8))=2^8.2^3.L_3(2)$ \\
$M_8$ & 62,208 & 11,200 & 87 & 62 & 44 & $N(W(A_2^4))$, order $2^8 3^5$ \\
$M_9$ & 57,600 & 12,096 & 70 & 53 & 53 & $N(W(A_4\times A_4))=(A_5\times A_5){:}((4\times 2){:}2)$ \\
\bottomrule
\end{tabular}

\caption{The nine conjugacy classes of maximal subgroups of $\WE$ (GAP; cf.\ \cite[\S3]{JKZ}), with the number of the $106$ characteristic polynomials of $\WE$ realised in each.}\label{tab:maximal}
\end{table}

\subsection{The certificate}\label{sec:certificate}
Table~\ref{tab:certificate} lists, for each $q\in\{5,7,11,13,23\}$, five or six curves $D\in\Fam(q)$ with the characteristic polynomial and determinant of $w_D$ and the maximal subgroups they exclude.  For all $26$ curves of Table~\ref{tab:certificate} and all $22$ points of Table~\ref{tab:lines}, $w_D$ was computed in exact arithmetic from the $240$ minimal sections (the method of \S\ref{sec:sectionsmethod}, in an independent implementation: a separating linear form, the shape lemma, explicit intersection numbers; ancillary \texttt{sections\_w.sage}); the characteristic polynomials and determinants in the tables are those of these integer matrices, and the $L$-polynomials printed in Table~\ref{tab:certificate} are $\det(1-qw_DT)$ (Theorem~A(2)).  The curves were found with a floating-point Lefschetz-sum computation (Appendix~\ref{app:pipelines}), on which the proof does not rely.  The ``excludes'' column is read off the characteristic-polynomial sets of \S\ref{sec:groupinput}.

\begin{table}[ht]\centering\scriptsize

\begin{tabular}{@{}rllrl>{\raggedright\arraybackslash}p{6.4cm}l@{}}
\toprule
$q$ & $D$ & $\chi_{w_D}$ & $\det w_D$ & $\varepsilon$ from & $L(E_D,T)$: $c_1,\dots,c_8$ ($c_0=1$) & excludes \\
\midrule
5 & $t^5 + t^2 + 3$ & $\Phi_{30}$ & $+1$ & complete & $5$, $0$, $-125$, $-625$, $-3125$, $0$, $78125$, $390625$ & 2,3,4,5,6,7,8,9 \\
5 & $t^5 + 2t$ & $\Phi_{24}$ & $+1$ & complete & $0$, $0$, $0$, $-625$, $0$, $0$, $0$, $390625$ & 2,3,4,5,7,9 \\
5 & $t^5 + t^3 + 3t$ & $\Phi_{20}$ & $+1$ & complete & $0$, $-25$, $0$, $625$, $0$, $-15625$, $0$, $390625$ & 2,3,4,5,6,7,8 \\
5 & $t^5 + 2t^2 + 3$ & $\Phi_{15}$ & $+1$ & complete & $-5$, $0$, $125$, $-625$, $3125$, $0$, $-78125$, $390625$ & 2,3,4,5,6,7,8,9 \\
5 & $t^5 + t^2 + 1$ & $\Phi_{1}\Phi_{2}\Phi_{4}\Phi_{8}$ & $-1$ & complete & $0$, $0$, $0$, $0$, $0$, $0$, $0$, $-390625$ & 1,5 \\
\addlinespace[2pt]
7 & $t^5 + 5t + 2$ & $\Phi_{30}$ & $+1$ & complete & $7$, $0$, $-343$, $-2401$, $-16807$, $0$, $823543$, $5764801$ & 2,3,4,5,6,7,8,9 \\
7 & $t^5 + t^2 + 2t$ & $\Phi_{24}$ & $+1$ & complete & $0$, $0$, $0$, $-2401$, $0$, $0$, $0$, $5764801$ & 2,3,4,5,7,9 \\
7 & $t^5 + 5t$ & $\Phi_{20}$ & $+1$ & complete & $0$, $-49$, $0$, $2401$, $0$, $-117649$, $0$, $5764801$ & 2,3,4,5,6,7,8 \\
7 & $t^5 + 5t + 5$ & $\Phi_{15}$ & $+1$ & complete & $-7$, $0$, $343$, $-2401$, $16807$, $0$, $-823543$, $5764801$ & 2,3,4,5,6,7,8,9 \\
7 & $t^5 + 4t + 1$ & $\Phi_{1}\Phi_{2}\Phi_{4}\Phi_{8}$ & $-1$ & complete & $0$, $0$, $0$, $0$, $0$, $0$, $0$, $-5764801$ & 1,5 \\
\addlinespace[2pt]
11 & $t^5 + 5$ & $\Phi_{30}$ & $+1$ & complete & $11$, $0$, $-1331$, $-14641$, $-161051$, $0$, $19487171$, $214358881$ & 2,3,4,5,6,7,8,9 \\
11 & $t^5 + t + 2$ & $\Phi_{24}$ & $+1$ & complete & $0$, $0$, $0$, $-14641$, $0$, $0$, $0$, $214358881$ & 2,3,4,5,7,9 \\
11 & $t^5 + t$ & $\Phi_{20}$ & $+1$ & complete & $0$, $-121$, $0$, $14641$, $0$, $-1771561$, $0$, $214358881$ & 2,3,4,5,6,7,8 \\
11 & $t^5 + 3$ & $\Phi_{15}$ & $+1$ & complete & $-11$, $0$, $1331$, $-14641$, $161051$, $0$, $-19487171$, $214358881$ & 2,3,4,5,6,7,8,9 \\
11 & $t^5 + 2t + 2$ & $\Phi_{1}\Phi_{2}\Phi_{4}\Phi_{8}$ & $-1$ & complete & $0$, $0$, $0$, $0$, $0$, $0$, $0$, $-214358881$ & 1,5 \\
\addlinespace[2pt]
13 & $t^5 + t + 6$ & $\Phi_{30}$ & $+1$ & complete & $13$, $0$, $-2197$, $-28561$, $-371293$, $0$, $62748517$, $815730721$ & 2,3,4,5,6,7,8,9 \\
13 & $t^5 + 2$ & $\Phi_{24}$ & $+1$ & complete & $0$, $0$, $0$, $-28561$, $0$, $0$, $0$, $815730721$ & 2,3,4,5,7,9 \\
13 & $t^5 + 6t + 6$ & $\Phi_{20}$ & $+1$ & complete & $0$, $-169$, $0$, $28561$, $0$, $-4826809$, $0$, $815730721$ & 2,3,4,5,6,7,8 \\
13 & $t^5 + 10t + 3$ & $\Phi_{15}$ & $+1$ & complete & $-13$, $0$, $2197$, $-28561$, $371293$, $0$, $-62748517$, $815730721$ & 2,3,4,5,6,7,8,9 \\
13 & $t^5 + 3t + 4$ & $\Phi_{1}\Phi_{2}\Phi_{4}\Phi_{8}$ & $-1$ & complete & $0$, $0$, $0$, $0$, $0$, $0$, $0$, $-815730721$ & 1,5 \\
\addlinespace[2pt]
23 & $t^5 + t^3 + 14t^2 + 22t + 6$ & $\Phi_{30}$ & $+1$ & $c_4\neq0$ & $23$, $0$, $-12167$, $-279841$, $-6436343$, $0$, $3404825447$, $78310985281$ & 2,3,4,5,6,7,8,9 \\
23 & $t^5 + 18t^2 + 7t + 4$ & $\Phi_{24}$ & $+1$ & $c_4\neq0$ & $0$, $0$, $0$, $-279841$, $0$, $0$, $0$, $78310985281$ & 2,3,4,5,7,9 \\
23 & $t^5 + 3t^3 + 10t^2 + 11t + 1$ & $\Phi_{20}$ & $+1$ & $c_4\neq0$ & $0$, $-529$, $0$, $279841$, $0$, $-148035889$, $0$, $78310985281$ & 2,3,4,5,6,7,8 \\
23 & $t^5 + t^3 + 13t^2 + 18t + 2$ & $\Phi_{15}$ & $+1$ & $c_4\neq0$ & $-23$, $0$, $12167$, $-279841$, $6436343$, $0$, $-3404825447$, $78310985281$ & 2,3,4,5,6,7,8,9 \\
23 & $t^5 + 12t + 4$ & $\Phi_{1}\Phi_{2}\Phi_{4}\Phi_{8}$ & $-1$ & $W(E_8)$ list & $0$, $0$, $0$, $0$, $0$, $0$, $0$, $-78310985281$ & 1,5 \\
23 & $t^5 + 7t^2 + 19t + 15$ & $\Phi_{1}\Phi_{2}\Phi_{4}\Phi_{12}$ & $-1$ & pair $(2,6)$ & $0$, $-529$, $0$, $0$, $0$, $148035889$, $0$, $-78310985281$ & 1,2,4,5 \\
\bottomrule
\end{tabular}

\caption{The certificate of Theorem B(a): $\chi_{w_D}$ and $\det w_D$ of $w_D$ computed exactly from the $240$ minimal sections, and $L(E_D,T)=\det(1-qw_DT)$.  ``$\varepsilon$ from'' records how the root number was fixed in the floating-point Lefschetz-sum computation by which the curves were found (Appendix~\ref{app:pipelines}); the proof uses $\det w_D$ from the sections.  ``Excludes'' lists the $j$ with $\chi_{w_D}\notin\{\chi_g:g\in M_j\}$.}\label{tab:certificate}
\end{table}

\subsection{Proof of Theorem B(a) and (b)}\label{sec:proofBab}
\emph{(a), arithmetic monodromy.} Fix $q\in\{5,7,11,13,23\}$ and let $G=G_{\mathrm{arith}}(\Fq)$.  If $G\ne\WE$ then $G\le M$ for some maximal $M$, conjugate to one of $M_1,\dots,M_9$; for every $u\in U(\Fq)$ the class $\rho(\mathrm{Fr}_u)$ meets $G$, so $\chi_{w_u}\in\{\chi_g:g\in M_j\}$ (a class function).  Table~\ref{tab:certificate} exhibits for each $j$ a curve with $\chi_{w_D}$ outside that set (the $\Phi_{30}$ row alone excludes $M_2,\dots,M_9$; any $\det=-1$ row excludes $M_1=W^+$).  Hence $G=\WE$.

\emph{(a), geometric monodromy.} $G_{\mathrm{geom}}(p)\trianglelefteq G_{\mathrm{arith}}(\Fq)=\WE$ and the quotient is a quotient of $\Gal(\Fqbar/\Fq)=\Zhat$, hence cyclic, generated by the image of any $\mathrm{Fr}_u$, $u\in U(\Fq)$.  The normal subgroups with cyclic quotient are $W^+$ and $\WE$.  If $G_{\mathrm{geom}}(p)=W^+$, all $\mathrm{Fr}_u$ would lie in one coset of $W^+$, i.e.\ $\det w_u$ would be constant on $U(\Fq)$; the certificate has rows of both signs.  So $G_{\mathrm{geom}}(p)=\WE$, and $G_{\mathrm{arith}}(\mathbb F_{p^f})\supseteq G_{\mathrm{geom}}(p)$ for all $f$.  (Frobenius elements are defined only up to conjugacy; the argument uses only that $\chi$ and $\det$ are class functions.)

\emph{(b).} Let $S=\operatorname{Spec}Z'$ and $S^\circ\subseteq S$ the open image of $U\to S$; it contains every $p\ge5$ with $U(\overline{\mathbb F}_p)\ne\emptyset$.  $U$ is normal and integral with generic fibre $U_\Q$, so $\rho(\pi_1(U_\Q))=\rho(\pi_1(U))$: the connected $\WE$-torsor $V\to U$ attached to $\rho$ ($V$ connected because $\rho$ is onto by (a)) has integral, hence connected, generic fibre.  The Frobenius classes of the points $u\in U(\mathbb F_{11})$ therefore lie in $G_{\mathrm{arith}}(\Q)$, and the $q=11$ certificate gives $G_{\mathrm{arith}}(\Q)=\WE$ as in (a).  Put $G_0=\rho(\pi_1(U_{\overline\Q}))\trianglelefteq\WE$, $Q_0=\WE/G_0$.  By the homotopy exact sequence $1\to\pi_1(U_{\overline\Q})\to\pi_1(U_\Q)\to\Gal(\overline\Q/\Q)\to1$ the surjection $\phi\colon\pi_1(U)\to Q_0$ restricted to $\pi_1(U_\Q)$ factors as $\psi\circ(\pi_1(U_\Q)\to\Gal(\overline\Q/\Q))$ with $\psi\colon\Gal(\overline\Q/\Q)\twoheadrightarrow Q_0$, kernel $\Gal(\overline\Q/k_0)$, $k_0/\Q$ finite Galois.  The $Q_0$-torsor $V_0\to U$ attached to $\phi$ is connected, hence integral, with generic fibre $U_\Q\otimes_\Q k_0$, so $V_0$ is the normalisation of $U$ in $k_0(U)$.  Let $S_1=\operatorname{Spec}\calO_{k_0}[1/6]$ and $S_1^\circ$ its preimage of $S^\circ$.  $U_{S^\circ}\times_{S^\circ}S_1^\circ$ is smooth over the normal integral $S_1^\circ$, hence normal; connected (open surjective with geometrically connected fibres onto a connected base), hence integral; finite over $U_{S^\circ}$ with function field $k_0(U)$.  By uniqueness of normalisation it equals $V_0|_{U_{S^\circ}}$, which is finite \'etale over $U_{S^\circ}$.  \'Etaleness descends along the faithfully flat $U_{S^\circ}\to S^\circ$ \cite[17.7.3(ii)]{EGAIV}, so $S_1^\circ\to S^\circ$ is \'etale and $\phi|\pi_1(U_{S^\circ})$ factors through $\pi_1(S^\circ)\to\Gal(k_0/\Q)$.  Consequently, for every $u\in U(\mathbb F_p)$, $p\in S^\circ$, $\phi(\mathrm{Fr}_u)$ is the Frobenius class of $p$ in $\Gal(k_0/\Q)$ -- the same class for all $u\in U(\mathbb F_p)$.  If $G_0=W^+$ then $\det w_u=\psi(\mathrm{Frob}_{11})$ would be constant on $U(\mathbb F_{11})$, contradicting the certificate; if $G_0\subseteq\{\pm1\}$ then all $w_u$, $u\in U(\mathbb F_{11})$, would lie in the preimage of a single class of $\WE/\{\pm1\}$, so at most two characteristic polynomials ($\chi_g$, $\chi_{-g}$) would occur on $U(\mathbb F_{11})$, contradicting the five types of the certificate ($82$ types occur).  Hence $G_0=\WE$. \qed

\subsection{Proof of Theorem B(c): every $p\ge5$}\label{sec:proofBc}
For $p\in\{5,7\}$ part (a) applies.  For $p\ge11$ we use one-parameter subfamilies.  For $D_0\in\Z[t]$ monic of degree $5$ and $D_1\in\Z[t]$ of degree $\le4$ let $\iota\colon\A^1\to\A^5$, $s\mapsto D_s=D_0+sD_1$, $\delta(s)=\Delta(\iota(s))\in\Z[s]$, $U_C=\iota^{-1}(U)=\A^1\smallsetminus V(\delta)$, and $N=\lead(\delta)\cdot\disc_s(\delta)\in\Z$.  Over $S'=\operatorname{Spec}\Z[1/6N]$ the divisor $B=V(\delta)\cup\{\infty\}\subset\Pp_{S'}$ is finite \'etale over $S'$, so $(\Pp_{S'},B)$ is a smooth proper curve with a relative \'etale boundary divisor and $U_C=\Pp\smallsetminus B$.  Let $\rho_C=\rho\circ\iota_*\colon\pi_1(U_C)\to\WE$; then $\rho_C(\pi_1(U_{C,\overline{\mathbb F}_p}))\subseteq G_{\mathrm{geom}}(p)$ for $p\in S'$.

\begin{lemma}[tame specialisation for a line]\label{lem:tame}
Let $p\nmid6N|\WE|$, i.e.\ $p\notin\{2,3,5,7\}$ and $p\nmid N$.  Then $\rho_C(\pi_1(U_{C,\overline{\mathbb F}_p}))=\rho_C(\pi_1(U_{C,\overline\Q}))$.
\end{lemma}
\begin{proof}
Let $R$ be the strict henselisation of $\Z_{(p)}$, with residue field $\overline{\mathbb F}_p$ and fraction field $K_R$, and consider $U_{C,R}\subset\Pp_R$ with the \'etale boundary $B_R$.  The image of $\rho_C$ is a subgroup of $\WE$, whose order $2^{14}3^55^27$ is prime to $p$; the inertia group at a closed point of $B_R$ in the special fibre maps onto a subgroup of that image, so its wild part (a pro-$p$ group, the kernel of the passage to the tame quotient) acts trivially: $\rho_C|\pi_1(U_{C,R})$ factors through the tame fundamental group $\pi_1^t(U_{C,R})$ (tame along $B_R$, \cite[Exp.~XIII, \S2]{SGA1}).  By Grothendieck's specialisation theorem for the tame fundamental group of a smooth proper curve with an \'etale boundary divisor over a strictly henselian trait \cite[Exp.~XIII, \S2]{SGA1} -- the statement used, that the specialisation map $\pi_1^t(U_{C,\overline{K}_R})\to\pi_1^t(U_{C,\overline{\mathbb F}_p})$ is surjective and induces an isomorphism on the maximal prime-to-$p$ quotients, is established in the course of the proof of \cite[Exp.~XIII, Cor.~2.12]{SGA1} (via Cor.~2.8 and the specialisation construction of 2.10 there) rather than enunciated as a separate corollary -- both $\rho_C(\pi_1(U_{C,\overline{\mathbb F}_p}))$ and $\rho_C(\pi_1(U_{C,\overline K_R}))$ equal $\rho_C(\pi_1^t(U_{C,R}))$.  Finally $\pi_1(U_{C,\overline K_R})\to\pi_1(U_{C,\overline\Q})$ is surjective (a connected finite \'etale cover of the $\overline\Q$-variety $U_{C,\overline\Q}$ is geometrically connected, so its base change to the algebraically closed $\overline K_R$ is connected).
\end{proof}

\begin{lemma}[the lines $A$ and $B$]\label{lem:lines}
Let
\[
  A\colon\ D_s=t^5+st^4+t^2+6t,\qquad B\colon\ D_s=t^5+4st^3+(3s+2)t^2+2st+s+1 .
\]
For each line, $11\nmid N$, $U_C(\mathbb F_{11})$ has all $11$ points, and among the $\mathbb F_{11}$-points one finds the Coxeter class $\Phi_{30}$ (at $s=0$), both signs of $\det$, and at least three distinct types (Table~\ref{tab:lines}).  Consequently $\rho_C(\pi_1(U_{C,\Q}))=\rho_C(\pi_1(U_{C,\overline\Q}))=\WE$.  Moreover $\gcd(N_A,N_B)$ is supported on $\{2,3,5,7\}$: $\deg\delta_A=10$, $N_A$ has $648$ digits and its prime factors below $10^6$ are $2,3,13,19,43,2137,4729$; $\deg\delta_B=14$, $N_B$ has $926$ digits and its prime factors below $10^6$ are $2,3,5,7,89,12923,112291$; the gcd with all factors $2,3,5,7$ removed equals $1$.
\end{lemma}
\begin{proof}
Exactly as in \S\ref{sec:proofBab} with $U$ replaced by $U_C$ over $S'$ ($U_C\to S'$ is smooth with geometrically connected fibres and $(11)$ is in its image): the Coxeter Frobenius excludes $M_2,\dots,M_9$, a $\det=-1$ point excludes $W^+$, so $\rho_C(\pi_1(U_{C,\Q}))=\WE$; the coset argument with $\ge3$ types and both signs of $\det$ among the $\mathbb F_{11}$-points gives $\rho_C(\pi_1(U_{C,\overline\Q}))\notin\{1,\{\pm1\},W^+\}$.  The arithmetic statements are the output of \texttt{curve\_specialization.py} ($N_A,N_B$ themselves, of $\approx600$--$1000$ digits, are not factored; only the gcd computation is used).
\end{proof}

\begin{table}[ht]\centering\small

\begin{tabular}{@{}rll|ll@{}}
\toprule
$s$ & line $A$: $\chi_w$ & $\det$ & line $B$: $\chi_w$ & $\det$ \\
\midrule
0 & $\Phi_{30}$ & $+1$ & $\Phi_{30}$ & $+1$ \\
1 & $\Phi_{1}\Phi_{2}\Phi_{3}\Phi_{8}$ & $-1$ & $\Phi_{1}\Phi_{2}\Phi_{4}\Phi_{12}$ & $-1$ \\
2 & $\Phi_{15}$ & $+1$ & $\Phi_{1}\Phi_{2}\Phi_{14}$ & $-1$ \\
3 & $\Phi_{1}^{2}\Phi_{2}^{2}\Phi_{5}$ & $+1$ & $\Phi_{6}^{2}\Phi_{12}$ & $+1$ \\
4 & $\Phi_{1}\Phi_{2}\Phi_{4}\Phi_{12}$ & $-1$ & $\Phi_{1}\Phi_{2}^{3}\Phi_{3}\Phi_{6}$ & $-1$ \\
5 & $\Phi_{1}^{2}\Phi_{2}^{2}\Phi_{10}$ & $+1$ & $\Phi_{1}^{2}\Phi_{2}^{2}\Phi_{4}\Phi_{6}$ & $+1$ \\
6 & $\Phi_{24}$ & $+1$ & $\Phi_{1}\Phi_{2}\Phi_{3}\Phi_{12}$ & $-1$ \\
7 & $\Phi_{1}^{2}\Phi_{7}$ & $+1$ & $\Phi_{1}^{2}\Phi_{3}\Phi_{5}$ & $+1$ \\
8 & $\Phi_{1}\Phi_{2}\Phi_{18}$ & $-1$ & $\Phi_{2}^{2}\Phi_{14}$ & $+1$ \\
9 & $\Phi_{1}\Phi_{2}\Phi_{4}\Phi_{12}$ & $-1$ & $\Phi_{15}$ & $+1$ \\
10 & $\Phi_{24}$ & $+1$ & $\Phi_{1}\Phi_{2}\Phi_{9}$ & $-1$ \\
\bottomrule
\end{tabular}

\caption{Frobenius types along the two lines at the $\mathbb F_{11}$-points $s=0,\dots,10$ ($\chi_{w_D}$ and $\det w_D$ of $w_D$ computed exactly from the minimal sections).}\label{tab:lines}
\end{table}

\emph{Proof of (c).} Let $p\ge11$.  Then $p\nmid N_A$ or $p\nmid N_B$; for that line Lemmas~\ref{lem:tame} and \ref{lem:lines} give $G_{\mathrm{geom}}(p)\supseteq\rho_C(\pi_1(U_{C,\overline{\mathbb F}_p}))=\rho_C(\pi_1(U_{C,\overline\Q}))=\WE$.  With (a) for $p=5,7$ this proves (c). \qed

\emph{Why two lines.} Every line through a fixed $D_0$ has $\delta(0)=\disc_t(4+27D_0^2)$; e.g.\ for $D_0=t^5+5$, $\delta(0)\propto5^{10}(7\cdot97)^4$, and these primes (and a $30$-digit factor) divide $N$ for every line through $t^5+5$, so a single line cannot cover all $p$.  Lines through two different Coxeter curves have exceptional integers coprime outside $\{2,3,5,7\}$, which is all that is needed.

\begin{remark}[what is and is not verified]\label{rem:Bcaveats}
Nothing in Theorem~B is conditional.  Its proof relies on: the list of the nine classes of maximal subgroups of $\WE$ (GAP; Jouve--Kowalski--Zywina \cite[\S3]{JKZ} also find nine classes, with the same indices, in Magma; cf.\ the entry for $O_8^+(2)$ in the ATLAS \cite{ATLAS}), the characteristic polynomials and determinants of Tables~\ref{tab:certificate} and~\ref{tab:lines} (computed exactly from the minimal sections, \texttt{sections\_w.sage}), and Grothendieck's tame specialisation theorem used in Lemma~\ref{lem:tame}.  Theorem~A is used as input.
\end{remark}

\section{Exact trace identities: proof of Theorem C}\label{sec:traces}

\begin{proof}[Proof of Theorem C]
Take logarithms in the Euler product defining $L(E_D,T)$:
\[
  \log L(E_D,T)=\sum_{v\ \mathrm{good}}\sum_{m\ge1}(\alpha_v^m+\beta_v^m)\frac{T^{m\deg v}}{m}+\sum_{v\ \mathrm{I}_1}\sum_{m\ge1}a_v^m\frac{T^{m\deg v}}{m},
\]
using $-\log(1-aT+qT^2)=\sum_m(\alpha^m+\beta^m)T^m/m$ and $-\log(1-aT)=\sum_ma^mT^m/m$; the additive place contributes $\log1=0$.  The coefficient of $T^d$ is $\sum_{v:\deg v\mid d}a_v^{(d/\deg v)}\deg(v)/d=S_d(D)/d$.  On the other hand, by \eqref{eq:LdetW}, $\log L(E_D,T)=\sum_{j=1}^8\log(1-q\lambda_jT)=-\sum_dq^d\bigl(\sum_j\lambda_j^d\bigr)T^d/d=-\sum_dq^dp_d(w_D)T^d/d$.  Comparing coefficients gives $S_d(D)=-q^dp_d(w_D)$.  For $d=1$ the only places of degree dividing $1$ are the $q$ finite degree-$1$ places and $\infty$, whence $\sum_{t_0\in\Fq}a_{t_0}=-q\tr w_D$.  For $d\ge2$, $d\sum_{\deg v=d}a_v=S_d-\sum_{e\mid d,e<d}e\sum_{\deg v=e}a_v^{(d/e)}$; for $d=2$, $a_v^{(2)}=a_v^2-2q$ at good degree-$1$ places and $=1$ at the $\mathrm I_1$ places of degree $1$, giving \eqref{eq:d2}.  The right-hand side of \eqref{eq:d2} involves $\sum_{\deg v=1}a_v^2$, which is not determined by $w_D$ (nor by $L(E_D,T)$): in the single class $\Phi_{24}$ at $q=11$ the mean of $a_v$ over the degree-$2$ places has mean $0.95$ and standard deviation $0.42$ across the $7{,}590$ curves of $\Fam(11)$ in that class (per-curve table for $\Fam(11)$, column \texttt{mav2}; see Appendix~\ref{app:data}).
\end{proof}

\begin{corollary}[the degree-$1$ mean is a class function]\label{cor:deg1}
$\bar a_1(D)=q^{-1}\sum_{\deg v=1,v\ne\infty}a_v(D)=-\tr(w_D)$ for every $D\in\Fam(q)$.  If the additive place $\infty$ ($a_\infty=0$) is included and the mean is taken over the $q+1$ places of degree $1$ of $\Pp$, both sides acquire a factor $q/(q+1)$.
\end{corollary}

\begin{remark}[which places carry the identity]\label{rem:whichplaces}
For $q\equiv2\pmod3$ there are no $\mathrm I_1$ places of degree $1$ (Lemma~\ref{lem:evendeg}), so $\sum_{t_0\in\Fq}a_{t_0}$ runs over good points only; for $q\equiv1\pmod3$ the number of degree-$1$ $\mathrm I_1$ places is $\#\{t_0\in\Fq:D(t_0)=\pm a\}$ ($a^2=-4/27$), of mean $2$ over all monic $D$ of degree $5$ (Lemma~\ref{lem:evendeg}(b)) and of mean $22{,}036/12{,}320\approx1.79$ resp.\ $590{,}876/316{,}238\approx1.87$ over $\Fam(7)$, $\Fam(13)$ (ancillary \texttt{enum\_family.gp}).
\end{remark}

\section{Trace statistics as Chebotarev in $\WE$: proof of Theorem D and the data}\label{sec:murm}

\subsection{Statistics}\label{sec:stats}
For $D\in\Fam(q)$, $d\ge1$ and $r\in\{0,\dots,8\}$: $\bar a_1(D)=-\tr(w_D)$ (Corollary~\ref{cor:deg1}); $S_d(D)=-q^dp_d(w_D)$ (Theorem~C); $\bar a_d(D)$ is the mean of $a_v$ over the $N_d$ finite places of degree $d$ (for $d\ge2$ not a class function); $N_r(q)$, $N_C(q)$, $N_\lambda(q)$, $N_{r,H}(q)$ are the numbers of curves of rank $r$, with $w_D\in C$, with $\chi_{w_D}=\lambda$, and of rank $r$ with $\Sha(E_D)\cong H$; $M_r(q)$, $M_r^{(d)}(q)$ are as in Theorem~D, and $M_{0,H}(q)$ is the mean of $\bar a_1$ over the rank-$0$ curves with $\Sha\cong H$.  For a class function $\varphi$ on $\WE$, $\E_q[\varphi]=|\Fam(q)|^{-1}\sum_D\varphi(w_D)$ and $\E_W[\varphi]=|\WE|^{-1}\sum_w\varphi(w)$, with conditional versions defined in the obvious way.

\subsection{The exact mechanism: proof of Theorem D(i)}\label{sec:mechanism}
Substituting Corollary~\ref{cor:deg1}, Theorem~C, Theorem~A(3) ($\rank=\mult_1$) and Theorem~A$'$(ii) ($\Sha$ is determined by the class) into the definitions gives, for every $q$ and every $r$ with $N_r(q)>0$,
\begin{gather*}
  M_r(q)=-\E_q[\tr w_D\mid\mult_1(w_D)=r],\qquad M_r^{(d)}(q)=-\E_q[p_d(w_D)\mid\mult_1(w_D)=r],\\
  M_{0,H}(q)=-\E_q[\tr w_D\mid\mult_1(w_D)=0,\ (E_8/(1-w_D)E_8)_{\mathrm{tors}}\cong H],
\end{gather*}
and more generally every average of a class function of $w_D$ (trace, power sums, $\det$, $L$-polynomial, $\varepsilon$) over the curves satisfying a condition that is itself a class function of $w_D$ is the average of the first function over the union of classes cut out by the condition, weighted by the empirical class counts $N_C(q)$. \qed

Over $\Fq(t)$ the difference between the rank-$0$ and rank-$1$ mean traces is thus the statement that the character $\tr w$ of the reflection representation has different conditional means on the elements of $\WE$ with no eigenvalue $1$ and on those with exactly one; nothing analytic is involved.  As a function of the place degree $d$, $-\E[p_d(w)\mid\mult_1(w)=r]$ is periodic with period the $\operatorname{lcm}$ of the element orders, $2520$.

\subsection{Equidistribution: proof of Theorem D(ii)}\label{sec:equid}
Let $Z\subset\A^5_{\Fq}$ be the hypersurface $\disc_t(4+27D^2)=0$ and $U=\A^5\smallsetminus Z$ as in \S\ref{sec:monrep}: an open, smooth, geometrically irreducible $\Fq$-variety of dimension $5$ with $U(\Fq)=\Fam(q)$.  Over $\Fq(\alpha)$, $27\alpha^2=-4$, one has $4+27D^2=27(D-\alpha)(D+\alpha)$, hence $\disc_t(4+27D^2)=-2^{20}3^{39}\disc_t(D-\alpha)\disc_t(D+\alpha)$, and each discriminant of a monic quintic has total degree $8$ in $a_0,\dots,a_4$; so $\deg Z=16$, $\#Z(\Fq)\le16q^4$ (Schwartz--Zippel) and $|\Fam(q)|=q^5(1+O(1/q))$, consistent with \eqref{eq:counts}.

\begin{lemma}[the $\WE$-cover is \'etale over $U$]\label{lem:etale}
Let $V\subset U\times\A^7$ be the closed subscheme of points $(D;x_0,x_1,x_2,y_0,\dots,y_3)$ satisfying $y^2=x^3+x+D$, i.e.\ the seven coefficient equations in $t^0,\dots,t^6$.  Then $V\to U$ is finite \'etale of degree $240$, and the normalisation $\tilde V$ of $U$ in the splitting field of the $240$ roots of the generic curve is finite \'etale Galois over $U$ with group $G_{\mathrm{arith}}(\Fq)$.  For $D\in U(\Fq)$ the Frobenius class $\mathrm{Frob}_D\subset G_{\mathrm{arith}}$ (geometric convention) is the class of $w_D$.
\end{lemma}
\begin{proof}
The argument works over $\operatorname{Spec}\Fq$ and, verbatim, over $\operatorname{Spec}\Z[1/6]$, which is used in Section~\ref{sec:uniform}.  In both cases $U$ is of finite type over the base, hence a Jacobson scheme whose closed points have finite residue fields.

\emph{Fibres.} Let $\Omega$ be an algebraically closed field of characteristic $0$ or $\ge5$, $D\in U(\Omega)$ and $(x,y)\in V_D(\Omega)$.  If $\deg x\le1$ then $x^3+x+D$ has odd degree $5$ and cannot be $y^2$; so $\deg x=2$ and $\deg y=3$.  Next, $y$ and $3x^2+1$ are coprime.  Otherwise, at a common zero $t_0$, differentiating $y^2=x^3+x+D$ gives $D'(t_0)=0$, and $D(t_0)=-x(t_0)(x(t_0)^2+1)=-2x(t_0)/3$ gives $27D(t_0)^2=-4$; so $t_0$ would be a multiple root of $\Delta_0$.  A tangent vector $(x',y')$ to $V_D$ at $(x,y)$, with $\deg x'\le2$ and $\deg y'\le3$, satisfies $2yy'=(3x^2+1)x'$.  So $y\mid x'$, hence $x'=0$, and then $y'=0$.  Thus every fibre of $V\to U$ is a finite disjoint union of spectra of finite separable field extensions.

\emph{Flatness.} $V$ is cut out by $7$ equations in the regular scheme $U\times\A^7$, and its fibres over the regular $U$ have dimension $0$.  So the equations cut the dimension by $7$, $V$ is Cohen--Macaulay, and it is flat over $U$ \cite[Thm.~23.1]{Matsumura}, \cite[Tag~00R4]{Stacks}.  Together with the fibres, this shows that $V\to U$ is \'etale \cite[17.6.1]{EGAIV}.

\emph{Degree.} $V\to U$ is flat, of finite presentation, separated and quasi-finite.  So $U$ is the disjoint union of locally closed strata $U_r$ such that, for every $g\colon T\to U$, the base change $V\times_UT\to T$ is finite locally free of degree $r$ iff $g(T)\subset U_r$ \cite[Tag~07RZ]{Stacks}.  At a closed point, with residue field $\mathbb F_{q'}$, the geometric fibre has exactly $240$ points, by Theorem~A(2) over $\mathbb F_{q'}$.  Every non-empty locally closed subset of the Jacobson scheme $U$ contains a closed point, so $U=U_{240}$.  Taking $T=U$, $V\to U$ is finite locally free of degree $240$, hence finite \'etale.  (Equivalently: it is \'etale, separated and of finite type with all geometric fibres of cardinality $240$ over a connected base, hence finite \cite[18.2.8--9]{EGAIV}.)

Its Galois closure $\tilde V\to U$ is finite \'etale Galois \cite[Exp.~V]{SGA1}.  For $D\in U(\Fq)$ the geometric Frobenius acts on $V_D$ (the $240$ roots of $E_D$ over $\Fqbar(t)$) as $\sigma^{-1}$ acts on $E_D(\Fqbar(t))$, i.e.\ as $w_D$ on the roots of $E_8$.
\end{proof}

\begin{lemma}[$G_{\mathrm{geom}}$ from Theorem B]\label{lem:Ggeom}
For every $q=p^f$ with $p\ge5$, $G_{\mathrm{arith}}(\Fq)=G_{\mathrm{geom}}(p)=\WE$.
\end{lemma}
\begin{proof}
$G_{\mathrm{geom}}(p)\subseteq G_{\mathrm{arith}}(\Fq)\subseteq\WE$, and $G_{\mathrm{geom}}(p)=\WE$ by Theorem~B(c).  (It would not suffice that the classes of the $w_D$ generate $\WE$: $G_{\mathrm{arith}}$ need only meet each of them, and the proper subgroup $\langle s_1\rangle$ meets the class of reflections, which generates $\WE$.  For $q\in\{5,7,11,13,23\}$ the certificate of \S\ref{sec:certificate} shows that no maximal subgroup meets all the classes $w_D$.)
\end{proof}

\begin{theorem}[effective Chebotarev for finite \'etale covers; Deligne]\label{thm:effcheb}
Let $U/\Fq$ be smooth, geometrically irreducible of dimension $n$, $\tilde V\to U$ finite \'etale Galois with group $\Gamma$, and assume $\tilde V$ geometrically irreducible ($G_{\mathrm{geom}}=\Gamma$).  For a conjugacy class $C\subset\Gamma$,
\[
  \Bigl|\#\{u\in U(\Fq):\mathrm{Frob}_u\in C\}-\frac{|C|}{|\Gamma|}\#U(\Fq)\Bigr|\le\frac{|C|}{|\Gamma|}\sum_{\chi\in\Irr(\Gamma),\chi\ne1}|\chi(C)|\,\sigma_c(\chi)\,q^{n-1/2},
\]
where $\sigma_c(\chi)=\sum_{i\le2n-1}\dim H^i_c(U_{\Fqbar},\calF_\chi)$ and $\calF_\chi$ is the lisse $\overline\Q_\ell$-sheaf on $U$ attached to the irreducible representation $\rho_\chi$ of $\Gamma$.  The constant depends only on the cover $\tilde V/U$ over $\mathbb F_p$, not on $f$.
\end{theorem}
\begin{proof}
Expand $1_C=\sum_\chi(|C|\overline{\chi(C)}/|\Gamma|)\chi$ and sum over $u\in U(\Fq)$: $\#\{\mathrm{Frob}_u\in C\}=(|C|/|\Gamma|)\#U(\Fq)+\sum_{\chi\ne1}(|C|\overline{\chi(C)}/|\Gamma|)\sum_u\chi(\mathrm{Frob}_u)$.  By the Grothendieck--Lefschetz trace formula \cite[Rapport]{SGA4h}, \cite[(1.12.1)]{DeligneWeilI}, $\sum_{u\in U(\Fq)}\tr(\mathrm{Frob}_u\mid\calF_\chi)=\sum_i(-1)^i\tr(\mathrm{Frob}\mid H^i_c(U_{\Fqbar},\calF_\chi))$.  $\calF_\chi$ has finite monodromy, hence is pure of weight $0$, and by Deligne \cite[Thm.~3.3.1]{DeligneWeilII} $H^i_c(U_{\Fqbar},\calF_\chi)$ is mixed of weights $\le i$.  Moreover $H^{2n}_c(U_{\Fqbar},\calF_\chi)\cong(\calF_\chi)_{\pi_1(U_{\Fqbar})}(-n)=(\rho_\chi)_\Gamma(-n)=0$, since $\rho_\chi|_{G_{\mathrm{geom}}}=\rho_\chi$ is irreducible and non-trivial.  Hence $|\sum_u\chi(\mathrm{Frob}_u)|\le\sum_{i\le2n-1}h^i_cq^{i/2}\le\sigma_c(\chi)q^{n-1/2}$.
\end{proof}
This is Deligne's equidistribution theorem in the finite-monodromy case (\S\ref{sec:prior}; see also \cite[Ch.~8]{KowalskiSieve}); the proof above is complete given \cite{DeligneWeilII}.

\begin{proof}[Proof of Theorem D(ii)]
Apply Theorem~\ref{thm:effcheb} with Lemma~\ref{lem:etale} ($\Gamma=G_{\mathrm{arith}}$), Lemma~\ref{lem:Ggeom} and Theorem~B ($\Gamma=\WE=G_{\mathrm{geom}}$, so $\tilde V$ is geometrically irreducible), and $\#U(\Fq)=q^5(1+O(1/q))$; this gives \eqref{eq:cheb} with a constant depending on $p$ through the Betti numbers of the cover $\tilde V_{\overline{\mathbb F}_p}$.  The conditional statements follow by dividing two such estimates, and the limit constants are finite sums over the $112$ classes (Table~\ref{tab:limits}).
\end{proof}

\begin{remark}[scope: fixed $p$, and uniformity in $p$]\label{rem:uniformp}
For each fixed $p\ge5$, Theorem~D(ii) is complete along $q=p^f$, $f\to\infty$.  The constant in the error term $O_p(q^{-1/2})$ depends on $p$ through the Betti numbers of the $\WE$-cover $\tilde V_{\overline{\mathbb F}_p}$, a $5$-dimensional cover of degree $696{,}729{,}600$.  Theorem~U (\S\ref{sec:thmU}) bounds these Betti numbers uniformly in $p$, so \eqref{eq:cheb} also holds along all prime powers, in particular along the primes.  But its constant, $K\le6\cdot19^{119}$, is astronomically large, so neither theorem makes a statement about any fixed $q$.  In particular neither says anything about the data of \S\ref{sec:murmdata} and Section~\ref{sec:arithstat} ($q=7,11,13,23$), which remain descriptive; the first-order law of Conjecture~\ref{conj:first} addresses them, conjecturally.  Note also that the characteristic-$0$ statement (Theorem~B(b)) and the constructibility of geometric irreducibility of the fibres of $\tilde V\to\operatorname{Spec}\Z[1/6N]$ \cite[9.7.7]{EGAIV} give $G_{\mathrm{geom}}(p)=\WE$ for all but finitely many $p$ without any line computation.  Theorem~B(c) makes the exceptional set empty, which is what Theorem~U needs.
\end{remark}

\subsection{The limit constants}\label{sec:limits}
Table~\ref{tab:limits} gives the exact $\WE$ constants of Theorem~D(ii); Table~\ref{tab:rank0classes} and the ancillary class table give the per-class data from which they are computed.  Two checks: $\E_W[\tr w]=0$, $\E_W[p_2]=\E_W[p_4]=1$, $\E_W[p_3]=0$ (character orthogonality), and $\det w=+1$ for exactly half of $\WE$, so $\E_W[\det]=0$ and the root number is equidistributed in the limit.  The rank law has a closed form independent of the class table:

\begin{proposition}[Shephard--Todd, Solomon]\label{prop:solomon}
Let $e_1,\dots,e_8=1,7,11,13,17,19,23,29$ be the exponents of $E_8$.  Then $\sum_{w\in\WE}t^{\mult_1(w)}=\prod_{i=1}^8(t+e_i)$.  Consequently $|\WE|=\prod_i(1+e_i)$, $P_W(\mult_1=r)$ is the coefficient of $t^r$ in $\prod_i(t+e_i)/|\WE|$, and
\[
  \E_W[\mult_1]=\sum_{i=1}^8\frac1{1+e_i}=\frac12+\frac18+\frac1{12}+\frac1{14}+\frac1{18}+\frac1{20}+\frac1{24}+\frac1{30}=\frac{121}{126}.
\]
\end{proposition}
\begin{proof}
The identity $\sum_{w\in W}t^{\dim V^w}=\prod_i(t+e_i)$ for a finite reflection group $W$ with exponents $e_i$ acting on $V$ is due to Shephard--Todd \cite{ShephardTodd} and Solomon \cite{Solomon1963}; here $V^w=E_8^w\otimes\Q$ has dimension $\mult_1(w)$.  Setting $t=1$ gives $|\WE|$; dividing by it gives the distribution of $\mult_1$; differentiating $\log\prod(t+e_i)$ at $t=1$ gives the mean.
\end{proof}
The nine coefficients of $\prod_i(t+e_i)$, namely $215{,}656{,}441$, $323{,}507{,}400$, $130{,}085{,}780$, $24{,}693{,}480$, $2{,}616{,}558$, $163{,}800$, $6{,}020$, $120$, $1$ for $r=0,\dots,8$, are exactly the numbers $\#\{w:\mult_1(w)=r\}$ obtained from the class table (Table~\ref{tab:limits}), a check of the table independent of GAP.

\begin{table}[ht]\centering\small

\begin{tabular}{@{}rrrrrrrr@{}}
\toprule
$r$ & $P_W(\mathrm{mult}_1=r)$ & $\mu_r=\mathbb E_W[\mathrm{tr}\,w\mid r]$ & exact & $\mathbb E_W[p_2\mid r]$ & $\mathbb E_W[p_3\mid r]$ & $\mathbb E_W[p_4\mid r]$ & $\mathbb E_W[p_5\mid r]$ \\
\midrule
0 & 0.309527 & -0.27586 & $-\tfrac{8}{29}$ & +0.57634 & -0.43182 & +0.29169 & -0.55855 \\
1 & 0.464323 & -0.08563 & $-\tfrac{230838}{2695895}$ & +0.86913 & -0.13872 & +0.78444 & -0.17536 \\
2 & 0.186709 & +0.41752 & $\tfrac{2715644}{6504289}$ & +1.64247 & +0.65199 & +2.08415 & +0.86384 \\
3 & 0.035442 & +1.10154 & $\tfrac{10794}{9799}$ & +2.68517 & +1.78733 & +3.78733 & +2.27717 \\
4 & 0.003755 & +1.96908 & $\tfrac{122672}{62299}$ & +3.96270 & +3.16333 & +5.46513 & +3.07857 \\
5 & 0.000235 & +3.04615 & $\tfrac{198}{65}$ & +5.41538 & +4.52308 & +6.52308 & +3.04615 \\
6 & 0.000009 & +4.37209 & $\tfrac{188}{43}$ & +6.88372 & +5.48837 & +6.88372 & +4.37209 \\
7 & 0.000000 & +6.00000 & $6$ & +8.00000 & +6.00000 & +8.00000 & +6.00000 \\
8 & 0.000000 & +8.00000 & $8$ & +8.00000 & +8.00000 & +8.00000 & +8.00000 \\
\bottomrule
\end{tabular}

\caption{Exact $\WE$ limits: $P_W(\mult_1=r)$, $\mu_r=\E_W[\tr w\mid\mult_1=r]$ (the rank-$r$ mean degree-$1$ trace tends to $-\mu_r$), and $\E_W[p_d(w)\mid\mult_1=r]$ for $d\le5$ (the Lefschetz constants $-\E[S_d/q^d\mid r]$ are their negatives; $d\le8$ in the ancillary CSV).}\label{tab:limits}
\end{table}

\subsection{Two exact symmetries of the class counts}\label{sec:twist}
First, $t\mapsto t+c$ ($c\in\Fq$) is an isomorphism $E_D\cong E_{D(t+c)}$ preserving $w_D$; for $p\ge7$ it acts freely on $\Fam(q)$ (the leading term of $D(t+c)-D(t)$ is $5ct^4\ne0$), so every class count $N_C(q)$ is divisible by $q$ (for $p=5$ there are fixed points, e.g.\ $D=t^5-t+a_0$).  Second:
\begin{proposition}[twist symmetry; proven]\label{prop:twist}
Let $\tilde D(t):=-D(-t)$ (again monic of degree $5$, with $4+27\tilde D^2$ squarefree iff $4+27D^2$ is), let $i\in\Fqbar$ with $i^2=-1$, let $\tau$ be the automorphism $t\mapsto-t$ of $\Pp_{\Fq}$, and let $E_D^\tau\colon y^2=x^3+x+D(-t)$ be the base change of $E_D$ along $\tau$; since $\tau$ permutes the places of $K$ preserving degrees, $L(E_D^\tau,T)=L(E_D,T)$ and $w_{E_D^\tau}\sim w_D$.  The map $(x,y)\mapsto(-x,iy)$ is an isomorphism $E_D^\tau\cong E_{\tilde D}$ over $\Fq(i)(t)$.  Hence:
\begin{enumerate}[label=\textup{(\alph*)},leftmargin=2em]
\item if $q\equiv1\pmod4$, $E_{\tilde D}\cong E_D^\tau$ over $K$, so $L(E_{\tilde D},T)=L(E_D,T)$ and $w_{\tilde D}$ is conjugate to $w_D$;
\item if $q\equiv3\pmod4$, $E_{\tilde D}$ is the quadratic twist of $E_D^\tau$ by the constant extension $\mathbb F_{q^2}/\Fq$, so $a_v(E_{\tilde D})=(-1)^{\deg v}a_{\tau(v)}(E_D)$, $L(E_{\tilde D},T)=L(E_D,-T)$, $w_{\tilde D}$ is conjugate to $-w_D$, and the class counts satisfy $N_C(q)=N_{-C}(q)$ exactly for every conjugacy class $C$ of $\WE$ (where $-C$ is the class of $-w$; on characteristic polynomials, $\Phi_n\leftrightarrow\Phi_{2n}$ for odd $n$).
\end{enumerate}
\end{proposition}
\begin{proof}
Substituting $(x,y)=(-x',iy')$ into $y^2=x^3+x+D(-t)$ gives $-y'^2=-x'^3-x'+D(-t)$, i.e.\ $y'^2=x'^3+x'+\tilde D(t)$.  The fibre of $E_{\tilde D}$ over $t_0$ is thus the $(-1)$-twist of the fibre of $E_D$ over $-t_0$, so $a_v(E_{\tilde D})=\chi(v)\,a_{\tau(v)}(E_D)$ with $\chi(v)=1$ if $-1$ is a square in $k_v$ and $\chi(v)=-1$ otherwise.  If $q\equiv1\pmod4$ then $i\in\Fq$, $\chi\equiv1$, and $\phi\colon(x,y)\mapsto(-x,iy)$ is defined over $K$.  If $q\equiv3\pmod4$ then $\sigma(i)=-i$ and $\phi^\sigma=[-1]\circ\phi$: $E_{\tilde D}$ is the quadratic twist of $E_D^\tau$ by the character of $\mathbb F_{q^2}/\Fq$, so $\chi(v)=(-1)^{\deg v}$ and $L(E_{\tilde D},T)=L(E_D^\tau,-T)=L(E_D,-T)$.  For $P\in E_D^\tau(\Kbar)$, $\sigma(\phi(P))=\phi^\sigma(\sigma P)=-\phi(\sigma P)$, so under $\phi\colon E_D^\tau(\Kbar)\cong E_{\tilde D}(\Kbar)$ the Frobenius of $E_{\tilde D}$ on the Mordell--Weil lattice is $-1$ times that of $E_D^\tau$, whose class is $w_D$; hence $w_{\tilde D}\sim-w_D$.  Since $D\mapsto\tilde D$ is an involution of $\Fam(q)$, $N_C(q)=N_{-C}(q)$.
\end{proof}
On the data, $N_C(q)=N_{-C}(q)$ holds for all class counts of $\Fam(7)$ and all $82$ characteristic polynomials occurring in $\Fam(11)$, and every class count at $q=7,11,13$ is divisible by $q$; at $q=13$, where (b) does not apply, $67$ of the $84$ observed types $\lambda$ have $N_\lambda\ne N_{\lambda(-z)}$.

\begin{remark}[the rank-$0$ mean]\label{rem:q13}
Among the $30$ rank-$0$ classes, $17$ have no eigenvalue $-1$ either (weight $0.7393$ of the rank-$0$ mass: $\Phi_{20}$, $\Phi_{24}$, $\Phi_{15}$, $\Phi_{30}$, $\Phi_3\Phi_9$, $\Phi_6\Phi_{18}$, $\Phi_4^2\Phi_{12}$, $\Phi_8^2$, $\Phi_{12}^2$, $\Phi_3^2\Phi_{12}$, $\Phi_6^2\Phi_{12}$, $\Phi_{10}^2$, $\Phi_5^2$, $\Phi_3^2\Phi_6^2$, $\Phi_4^4$, $\Phi_3^4$, $\Phi_6^4$); this set is closed under $C\mapsto-C$, and since $|C|=|{-C}|$ and $\tr(-w)=-\tr w$ its contribution to $\E_W[\tr w\mid\mult_1=0]$ is exactly zero.  The remaining $13$ classes (those with an eigenvalue $-1$; weight $0.2607$) carry the whole of $\mu_0=-8/29$.  For $q\equiv3\pmod4$ the same cancellation holds in $M_0(q)$ at every such $q$, by Proposition~\ref{prop:twist}(b): the $17$ classes contribute exactly zero to $M_0(q)$.  For $q\equiv1\pmod4$ Proposition~\ref{prop:twist} gives no such relation.
\end{remark}

\subsection{The data}\label{sec:murmdata}
The data in this subsection and in Section~\ref{sec:arithstat} are descriptive: neither Theorem~D(ii) nor Theorem~U makes a statement about them (Remark~\ref{rem:uniformp}).  Their first-order behaviour is the subject of Section~\ref{sec:uniform}.  Table~\ref{tab:datasummary} summarises the data sets: complete enumerations of $\Fam(q)$ for $q=7,11,13$ and a uniform sample of $10{,}000$ monic degree-$5$ $D$ at $q=23$ (seed $20260926$), of which $9{,}981$ lie in $\Fam(23)$.  For every curve the $L$-polynomial was computed from point counts (Appendix~\ref{app:pipelines}); the $|\Sha|$ of a rank-$0$ curve is $L(E_D,1/q)$ (Theorem~A(4)).

\begin{table}[ht]\centering\small
\begin{tabular}{@{}rrrrrrrrrrrrr@{}}
\toprule
$q$ & deg-5 $D$ & $|\mathcal F_5|$ & excluded & types & $r=0$ & $r=1$ & $r=2$ & $r=3$ & $r=4$ & $r\ge5$ & $\varepsilon=+1$ & $\varepsilon=-1$ \\
\midrule
7 & 16,807 & 12,320 & 4,487 & 62 & 4,291 & 5,775 & 2,065 & 189 & 0 & 0 & 6,356 & 5,964 \\
11 & 161,051 & 159,841 & 1,210 & 82 & 55,946 & 71,060 & 28,050 & 4,180 & 605 & 0 & 84,601 & 75,240 \\
13 & 371,293 & 316,238 & 55,055 & 84 & 103,077 & 149,422 & 55,913 & 7,514 & 312 & 0 & 159,302 & 156,936 \\
23 & 10,000 & 9,981 & 19 & 91 & 3,260 & 4,516 & 1,825 & 345 & 32 & 3 & 5,117 & 4,864 \\
\bottomrule
\end{tabular}

\caption{The data sets.  ``Excluded'': monic degree-$5$ $D$ with $4+27D^2$ not squarefree (outside $\Fam$).  ``Types'': distinct $\Lunit$ observed (of $106$).}\label{tab:datasummary}
\end{table}

Among $q=7,11,13$ and the $q=23$ sample, rank $4$ first occurs at $q=11$ ($605$ curves) and rank $5$ at $q=23$ ($3$ curves in the sample). The exact counts of \S\ref{sec:conj} show more: $\Fam(5)$ already contains $2$ curves of rank $4$, rank $5$ first occurs at $q=17$ ($136$ curves) and rank $6$ at $q=19$ ($19$ curves) \textup{(Computed.)}  Of the $106$ characteristic polynomials of $\WE$, $62$, $82$, $84$, $91$ occur at $q=7,11,13,23$; the per-type counts, and the per-class counts of the $30$ rank-$0$ types with the class constants $S_1=-q\tr w$, $S_2=-q^2p_2(w)$, are in the ancillary tables \texttt{type\_freq\_q*.tex} and \texttt{classes\_rank0\_q*.tex}.

Table~\ref{tab:deg1} gives the rank-conditioned mean degree-$1$ trace $\E[\bar a_1\mid r]$, which by Theorem~D(i) equals $-\E_q[\tr w_D\mid\mult_1(w_D)=r]$.  At each of the four $q$ the means are ordered $M_0>M_1>0>M_2>M_3>M_4$ (as far as defined).  The Lefschetz statistics $\E[p_d(w_D)\mid r]=-\E[S_d/q^d\mid r]$, $\E[S_d/N_d\mid r]$ and $\E[\bar a_d\mid r]$ for $d\le4$, $r\le2$ at all four $q$ are in the ancillary table \texttt{lefschetz.tex}.  Within rank $0$, Table~\ref{tab:shastrat} and Figure~\ref{fig:shastrat} give the means $\E[\bar a_1\mid\rank 0,|\Sha|=s]$.  The strata $|\Sha|=25,36,81$ consist of a single class each ($\Phi_5^2$, $\Phi_2^2\Phi_3^2\Phi_6$, $\Phi_3^4$), so their means are exactly $2$, $3$, $4$ at every $q$; $|\Sha|=64$ first occurs at $q=23$ ($2$ curves, class $\Phi_2^4\Phi_4^2$).

\begin{table}[ht]\centering\small
\begin{tabular}{@{}rrrrrrrrr@{}}
\toprule
 & \multicolumn{2}{c}{$q=7$} & \multicolumn{2}{c}{$q=11$} & \multicolumn{2}{c}{$q=13$} & \multicolumn{2}{c}{$q=23$} \\
\cmidrule(lr){2-3}\cmidrule(lr){4-5}\cmidrule(lr){6-7}\cmidrule(lr){8-9}
$r$ & $N_r$ & $\mathbb E[\bar a_1\mid r]$ & $N_r$ & $\mathbb E[\bar a_1\mid r]$ & $N_r$ & $\mathbb E[\bar a_1\mid r]$ & $N_r$ & $\mathbb E[\bar a_1\mid r]$ \\
\midrule
0 & 4,291 & +0.1436 & 55,946 & +0.2782 & 103,077 & +0.2031 & 3,260 & +0.2666 \\
1 & 5,775 & +0.0315 & 71,060 & +0.0642 & 149,422 & +0.0597 & 4,516 & +0.0853 \\
2 & 2,065 & -0.2983 & 28,050 & -0.5059 & 55,913 & -0.3320 & 1,825 & -0.4685 \\
3 & 189 & -0.9630 & 4,180 & -1.0921 & 7,514 & -0.9239 & 345 & -1.1739 \\
4 & 0 & -- & 605 & -2.2727 & 312 & -1.6667 & 32 & -2.2188 \\
5 & 0 & -- & 0 & -- & 0 & -- & 3 & -3.3333 \\
\bottomrule
\end{tabular}

\caption{The rank-conditioned mean degree-$1$ trace $\E[\bar a_1\mid\rank=r]=-\E_q[\tr w_D\mid\mult_1(w_D)=r]$: for $q=7,11,13$ the exact value $M_r(q)$ over all of $\Fam(q)$, for $q=23$ the mean over the random sample.  The $\WE$ constants $-\mu_r$ are in Table~\ref{tab:limits}.}\label{tab:deg1}
\end{table}

\begin{table}[ht]\centering\small
\begin{tabular}{@{}rrrrrrrrr@{}}
\toprule
 & \multicolumn{2}{c}{$q=7$} & \multicolumn{2}{c}{$q=11$} & \multicolumn{2}{c}{$q=13$} & \multicolumn{2}{c}{$q=23$} \\
\cmidrule(lr){2-3}\cmidrule(lr){4-5}\cmidrule(lr){6-7}\cmidrule(lr){8-9}
$|\Sha|$ & $N$ & $\mathbb E[\bar a_1\mid 0,|\Sha|]$ & $N$ & $\mathbb E[\bar a_1\mid 0,|\Sha|]$ & $N$ & $\mathbb E[\bar a_1\mid 0,|\Sha|]$ & $N$ & $\mathbb E[\bar a_1\mid 0,|\Sha|]$ \\
\midrule
1 & 2,926 & -0.0742 & 34,133 & -0.1798 & 66,131 & -0.1176 & 1,956 & -0.1636 \\
4 & 1,134 & +0.4877 & 14,410 & +0.7214 & 27,885 & +0.6117 & 876 & +0.6438 \\
9 & 196 & +1.1071 & 4,455 & +1.1605 & 6,682 & +1.1012 & 278 & +1.1511 \\
16 & 35 & +1.8000 & 2,090 & +1.9211 & 1,976 & +1.7105 & 111 & +1.8649 \\
25 & 0 & -- & 462 & +2.0000 & 299 & +2.0000 & 21 & +2.0000 \\
36 & 0 & -- & 385 & +3.0000 & 104 & +3.0000 & 16 & +3.0000 \\
64 & 0 & -- & 0 & -- & 0 & -- & 2 & +4.0000 \\
81 & 0 & -- & 11 & +4.0000 & 0 & -- & 0 & -- \\
\bottomrule
\end{tabular}

\caption{Rank-$0$ curves stratified by $|\Sha|=\det(1-w_D)$: number $N$ of curves and mean degree-$1$ trace, over all of $\Fam(q)$ for $q=7,11,13$ and over the random sample for $q=23$.}\label{tab:shastrat}
\end{table}

\begin{figure}[ht]\centering
\includegraphics[width=0.7\textwidth]{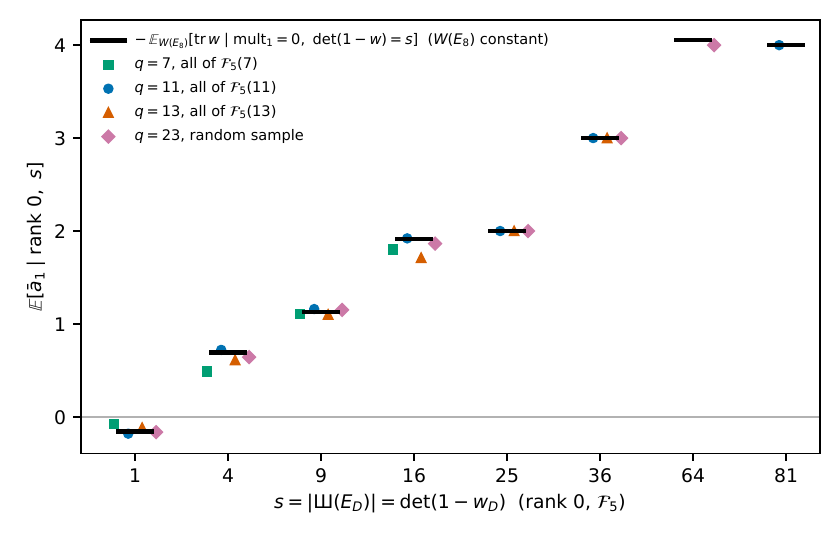}
\caption{$|\Sha|$-stratified degree-$1$ means of rank-$0$ curves: all of $\Fam(q)$ for $q=7,11,13$ and the random sample for $q=23$.  Horizontal marks: the $\WE$ constants $-\E_W[\tr w\mid\mult_1=0,\det(1-w)=s]$ of Theorem~D(ii), for reference; they are limits as $q\to\infty$ (for fixed $p$ by Theorem~D(ii), through all prime powers by Theorem~U), not predictions at fixed $q$.}\label{fig:shastrat}
\end{figure}

\subsection{Comparison with the murmuration literature}\label{sec:comparison}
The rank dependence of the mean trace in Theorem~D is motivated by the murmurations of He--Lee--Oliver--Pozdnyakov \cite{HLOP2022} and the literature following it (\S\ref{sec:prior}); here there is no $p/N$ scaling and no density, the conductor degree is fixed, and the only limit is $q=p^f\to\infty$.

\section{Exact arithmetic statistics: Theorem E and the class-level data}\label{sec:arithstat}

\subsection{The class table}\label{sec:classtable}
Theorem~E's first sentence is Theorem~A$'$ (which identifies rank, $\Sha$, $E_D(K)$ and $\Reg$ with class functions of $w_D$) combined with Theorem~D(ii); the limiting proportions are the $\WE$-masses of the fibres of $w\mapsto(\mult_1w,(E_8/(1-w)E_8)_{\mathrm{tors}},E_8^w,\disc E_8^w)$.  Table~\ref{tab:rank0classes} lists the $30$ rank-$0$ classes, for which the $L$-polynomial determines everything; Table~\ref{tab:shared} lists the eleven classes sharing one of the five ambiguous characteristic polynomials, with the invariants that distinguish them; Table~\ref{tab:ranksummary} summarises all $112$ classes by rank.  The complete $112$-row table (class size, order, characteristic polynomial, $\mult_1$, $\det$, $\detp(1-w)$, $E_8^w$ with rank, discriminant and root type, Smith normal form of $1-w$, $\Sha$, cycle type on the $240$ roots, and a word in the simple reflections) is in the ancillary files (its $\Sha$ column recomputes Urabe's table; \S\ref{sec:HxH}).  Some consequences read off the table: the $\WE$-masses of the pairs $(\rank,\Sha)$ include $P(\rank 0,\Sha=0)=0.18547$, $P(\rank0,\Sha\cong(\Z/2)^2)=0.08781$, $P(\rank1,\Sha=0)=0.41556$, $P(\rank1,\Sha\cong(\Z/2)^2)=0.04407$, $P(\rank2,\Sha=0)=0.18097$; $|\Sha|=16$ is $(\Z/4)^2$ for two rank-$0$ classes (mass $0.00955$) and $(\Z/2)^4$ for three (mass $0.00124$); at rank $1$, $E_D(K)=\Z v$ with $\langle v,v\rangle=2$ (a rational minimal section) for $12$ classes and $\langle v,v\rangle\in\{4,6,8,12,14,20,30\}$ for the other $19$.

\begin{table}[ht]\centering\small

\begin{tabular}{@{}rlrrlr@{\hspace{1.2em}}rlrrlr@{}}
\toprule
\# & $\chi_w$ & $|C|$ & $|\Sha|$ & $\Sha(E_D)$ & $\operatorname{ord} w$ & \# & $\chi_w$ & $|C|$ & $|\Sha|$ & $\Sha(E_D)$ & $\operatorname{ord} w$ \\
\midrule
1 & $\Phi_{20}$ & 34,836,480 & 1 & $0$ & 20 & 16 & $\Phi_{6}^{2}\Phi_{12}$ & 2,419,200 & 1 & $0$ & 12 \\
2 & $\Phi_{24}$ & 29,030,400 & 1 & $0$ & 24 & 17 & $\Phi_{2}^{2}\Phi_{4}^{2}\Phi_{6}$ & 1,209,600 & 16 & $(\mathbb Z/4)^{2}$ & 12 \\
3 & $\Phi_{2}^{2}\Phi_{14}$ & 24,883,200 & 4 & $(\mathbb Z/2)^{2}$ & 14 & 18 & $\Phi_{10}^{2}$ & 1,161,216 & 1 & $0$ & 10 \\
4 & $\Phi_{15}$ & 23,224,320 & 1 & $0$ & 15 & 19 & $\Phi_{5}^{2}$ & 1,161,216 & 25 & $(\mathbb Z/5)^{2}$ & 5 \\
5 & $\Phi_{30}$ & 23,224,320 & 1 & $0$ & 30 & 20 & $\Phi_{2}^{2}\Phi_{3}^{2}\Phi_{6}$ & 806,400 & 36 & $(\mathbb Z/6)^{2}$ & 6 \\
6 & $\Phi_{3}\Phi_{9}$ & 12,902,400 & 9 & $(\mathbb Z/3)^{2}$ & 9 & 21 & $\Phi_{2}^{4}\Phi_{10}$ & 580,608 & 16 & $(\mathbb Z/2)^{4}$ & 10 \\
7 & $\Phi_{6}\Phi_{18}$ & 12,902,400 & 1 & $0$ & 18 & 22 & $\Phi_{3}^{2}\Phi_{6}^{2}$ & 403,200 & 9 & $(\mathbb Z/3)^{2}$ & 6 \\
8 & $\Phi_{2}^{2}\Phi_{6}\Phi_{10}$ & 11,612,160 & 4 & $(\mathbb Z/2)^{2}$ & 30 & 23 & $\Phi_{2}^{4}\Phi_{6}^{2}$ & 268,800 & 16 & $(\mathbb Z/2)^{4}$ & 6 \\
9 & $\Phi_{4}^{2}\Phi_{12}$ & 9,676,800 & 4 & $(\mathbb Z/2)^{2}$ & 12 & 24 & $\Phi_{2}^{2}\Phi_{6}^{3}$ & 89,600 & 4 & $(\mathbb Z/2)^{2}$ & 6 \\
10 & $\Phi_{2}^{2}\Phi_{18}$ & 6,451,200 & 4 & $(\mathbb Z/2)^{2}$ & 18 & 25 & $\Phi_{2}^{4}\Phi_{4}^{2}$ & 37,800 & 64 & $(\mathbb Z/2)^{2}\times(\mathbb Z/4)^{2}$ & 4 \\
11 & $\Phi_{2}^{2}\Phi_{4}\Phi_{8}$ & 5,443,200 & 16 & $(\mathbb Z/4)^{2}$ & 8 & 26 & $\Phi_{4}^{4}$ & 15,120 & 16 & $(\mathbb Z/2)^{4}$ & 4 \\
12 & $\Phi_{2}^{2}\Phi_{6}\Phi_{12}$ & 4,838,400 & 4 & $(\mathbb Z/2)^{2}$ & 12 & 27 & $\Phi_{3}^{4}$ & 4,480 & 81 & $(\mathbb Z/3)^{4}$ & 3 \\
13 & $\Phi_{8}^{2}$ & 3,628,800 & 4 & $(\mathbb Z/2)^{2}$ & 8 & 28 & $\Phi_{6}^{4}$ & 4,480 & 1 & $0$ & 6 \\
14 & $\Phi_{12}^{2}$ & 2,419,200 & 1 & $0$ & 12 & 29 & $\Phi_{2}^{6}\Phi_{6}$ & 2,240 & 64 & $(\mathbb Z/2)^{6}$ & 6 \\
15 & $\Phi_{3}^{2}\Phi_{12}$ & 2,419,200 & 9 & $(\mathbb Z/3)^{2}$ & 12 & 30 & $\Phi_{2}^{8}$ & 1 & 256 & $(\mathbb Z/2)^{8}$ & 2 \\
\bottomrule
\end{tabular}

\caption{The $30$ conjugacy classes of $\WE$ with $\mult_1=0$.  For these the characteristic polynomial determines the class, $\Sha(E_D)\cong E_8/(1-w)E_8$ and $|\Sha|=\det(1-w)=L(E_D,1/q)$.  The $\Sha$ column is $H^1(\Zhat,E_8)$, tabulated for all $112$ classes by Urabe \cite{Urabe1996MComp} and recomputed here; $|C|/|W|=|C|/696{,}729{,}600$.}\label{tab:rank0classes}
\end{table}

\begin{table}[ht]\centering\scriptsize

\begin{tabular}{@{}rlrrrrrlll@{}}
\toprule
\# & $\chi_w$ & $r$ & $|C|$ & $|C|/|W|$ & $\det'(1-w)$ & $\disc E_8^{w}$ & $\Sha$ & roots of $E_8^w$ & cycle type on roots \\
\midrule
31 & $\Phi_{1}\Phi_{2}\Phi_{4}\Phi_{8}$ & 1 & 43,545,600 & 0.06250 & 8 & 8 & $0$ & -- & $8^{29}\,4^{2}$ \\
45 & $\Phi_{1}\Phi_{2}\Phi_{4}\Phi_{8}$ & 1 & 10,886,400 & 0.01562 & 8 & 2 & $(\mathbb Z/2)^{2}$ & $A_{1}$ & $8^{27}\,4^{5}\,2\,1^{2}$ \\
\addlinespace[2pt]
63 & $\Phi_{1}^{2}\Phi_{2}^{2}\Phi_{3}\Phi_{6}$ & 2 & 14,515,200 & 0.02083 & 12 & 12 & $0$ & $A_{1}$ & $6^{37}\,3^{2}\,2^{5}\,1^{2}$ \\
74 & $\Phi_{1}^{2}\Phi_{2}^{2}\Phi_{3}\Phi_{6}$ & 2 & 3,225,600 & 0.00463 & 12 & 12 & $0$ & -- & $6^{34}\,3^{8}\,2^{6}$ \\
77 & $\Phi_{1}^{2}\Phi_{2}^{2}\Phi_{3}\Phi_{6}$ & 2 & 1,612,800 & 0.00231 & 12 & 3 & $(\mathbb Z/2)^{2}$ & $A_{2}$ & $6^{35}\,3^{6}\,2^{3}\,1^{6}$ \\
\addlinespace[2pt]
71 & $\Phi_{1}^{2}\Phi_{2}^{2}\Phi_{4}^{2}$ & 2 & 5,443,200 & 0.00781 & 16 & 16 & $0$ & -- & $4^{58}\,2^{4}$ \\
81 & $\Phi_{1}^{2}\Phi_{2}^{2}\Phi_{4}^{2}$ & 2 & 680,400 & 0.00098 & 16 & 4 & $(\mathbb Z/2)^{2}$ & $2A_{1}$ & $4^{54}\,2^{10}\,1^{4}$ \\
\addlinespace[2pt]
89 & $\Phi_{1}^{3}\Phi_{2}^{3}\Phi_{4}$ & 3 & 2,721,600 & 0.00391 & 16 & 16 & $0$ & $2A_{1}$ & $4^{45}\,2^{28}\,1^{4}$ \\
96 & $\Phi_{1}^{3}\Phi_{2}^{3}\Phi_{4}$ & 3 & 151,200 & 0.00022 & 16 & 4 & $(\mathbb Z/2)^{2}$ & $A_{3}$ & $4^{45}\,2^{24}\,1^{12}$ \\
\addlinespace[2pt]
102 & $\Phi_{1}^{4}\Phi_{2}^{4}$ & 4 & 113,400 & 0.00016 & 16 & 16 & $0$ & $4A_{1}$ & $2^{116}\,1^{8}$ \\
105 & $\Phi_{1}^{4}\Phi_{2}^{4}$ & 4 & 3,150 & 0.00000 & 16 & 4 & $(\mathbb Z/2)^{2}$ & $D_{4}$ & $2^{108}\,1^{24}$ \\
\bottomrule
\end{tabular}

\caption{The five characteristic polynomials shared by several classes (all of positive rank), with the invariants that separate the classes: $\detp(1-w)=|\Sha|\cdot\Reg$ is the same within a group, but $\Reg=\disc E_8^w$ and $\Sha$ are not all equal within any group (for $\Phi_1^2\Phi_2^2\Phi_3\Phi_6$ two of the three classes share $(\Reg,\Sha)=(12,0)$ and differ only in the root type of $E_8^w$ and in the cycle type).  The pair (characteristic polynomial, cycle type on roots) separates all $112$ classes.}\label{tab:shared}
\end{table}

\begin{table}[ht]\centering\scriptsize

\begin{tabular}{@{}rrrr>{\raggedright\arraybackslash}p{6.9cm}>{\raggedright\arraybackslash}p{4.9cm}@{}}
\toprule
$r$ & classes & $\chi_w$'s & $P_W(\mathrm{mult}_1=r)$ & $\Sha$ groups (limit frequency given $r$) & root types of $E_8^w$ \\
\midrule
0 & 30 & 30 & 0.309527 & $0$ (0.5992), $(\mathbb Z/2)^{2}$ (0.2837), $(\mathbb Z/3)^{2}$ (0.07292), $(\mathbb Z/4)^{2}$ (0.03085), $(\mathbb Z/5)^{2}$ (0.005385), $(\mathbb Z/2)^{4}$ (0.004009), $(\mathbb Z/6)^{2}$ (0.003739), $(\mathbb Z/2)^{2}\times(\mathbb Z/4)^{2}$ (0.0001753), $(\mathbb Z/3)^{4}$ (2.077e-05), $(\mathbb Z/2)^{6}$ (1.039e-05), $(\mathbb Z/2)^{8}$ (4.637e-09) & -- (30) \\
1 & 31 & 30 & 0.464323 & $0$ (0.895), $(\mathbb Z/2)^{2}$ (0.09491), $(\mathbb Z/3)^{2}$ (0.008309), $(\mathbb Z/4)^{2}$ (0.001402), $(\mathbb Z/2)^{4}$ (0.0003895), $(\mathbb Z/2)^{6}$ (3.709e-07) & -- (19), $A_{1}$ (12) \\
2 & 24 & 21 & 0.186709 & $0$ (0.9693), $(\mathbb Z/2)^{2}$ (0.03003), $(\mathbb Z/3)^{2}$ (0.0006888), $(\mathbb Z/2)^{4}$ (2.906e-05) & $A_{1}$ (7), -- (6), $2A_{1}$ (6), $A_{2}$ (5) \\
3 & 12 & 11 & 0.035442 & $0$ (0.9923), $(\mathbb Z/2)^{2}$ (0.007654) & $3A_{1}$ (3), $A_{3}$ (3), $2A_{1}$ (2), $A_{2}$ (2), $A_{1}$ (1), $A_{2}+A_{1}$ (1) \\
4 & 8 & 7 & 0.003755 & $0$ (0.9988), $(\mathbb Z/2)^{2}$ (0.001204) & $D_{4}$ (3), $2A_{2}$ (1), $4A_{1}$ (1), $A_{3}$ (1), $A_{3}+A_{1}$ (1), $A_{4}$ (1) \\
5 & 3 & 3 & 0.000235 & $0$ (1) & $A_{5}$ (1), $D_{4}+A_{1}$ (1), $D_{5}$ (1) \\
6 & 2 & 2 & 0.000009 & $0$ (1) & $D_{6}$ (1), $E_{6}$ (1) \\
7 & 1 & 1 & 0.000000 & $0$ (1) & $E_{7}$ (1) \\
8 & 1 & 1 & 0.000000 & $0$ (1) & $E_{8}$ (1) \\
\bottomrule
\end{tabular}

\caption{All $112$ classes summarised by $r=\mult_1(w)$: number of classes and of distinct characteristic polynomials, $\WE$ mass, the groups $\Sha(E_D)=H^1(\Zhat,E_8)$ that occur (with their $\WE$ frequency given the rank; cf.\ Urabe's table \cite{Urabe1996MComp}) and the root types of the fixed lattice $E_8^w=E_D(K)$ (number of classes).}\label{tab:ranksummary}
\end{table}

\subsection{Class-level data}\label{sec:obspred}
For all $12{,}320$ curves of $\Fam(7)$ and for seeded samples of $7{,}500$ and $2{,}000$ curves drawn uniformly from all of $\Fam(11)$ and $\Fam(13)$ (seed $2026$), $w_D$ itself was computed exactly from the $240$ minimal sections (\S\ref{sec:sectionsmethod}); on every one of these curves it agrees with the point-count data (Appendix~\ref{app:chain}).  Table~\ref{tab:shagroups} gives the observed distribution of $\Sha(E_D)$ as a group (all ranks), with the $\WE$ mass of each group, and Table~\ref{tab:ambiguous} resolves the five ambiguous characteristic polynomials; the Mordell--Weil lattices $E_D(K)=E_8^{w_D}$, recorded as (rank, root type, $\Reg$), are tabulated in the ancillary table \texttt{sections\_mwl.tex}.  All observed $\Sha$ groups are of the form $H\times H$, and the group is not determined by its order: at $q=11$, $|\Sha|=16$ occurs as $(\Z/4)^2$ for $68$ curves and as $(\Z/2)^4$ for $24$.  Curves of positive rank with non-trivial $\Sha$: $385$ of $8{,}029$ at $q=7$ ($336$ with $(\Z/2)^2$, $49$ with $(\Z/3)^2$), $290$ of $4{,}874$ at $q=11$ ($269$, $19$, and $2$ with $(\Z/4)^2$), $73$ of $1{,}369$ at $q=13$ ($71$, $2$).  The curves of type $z^8-1$ split $1{,}120:154$ at $q=7$, $431:81$ at $q=11$ and $152:35$ at $q=13$ between $(A_7)''$ and $(A_7)'$ (class-size ratio $4:1$).  This is the arithmetic content of Corollary~\ref{cor:notLinv}: the curves in $(A_7)'$ are exactly those with an $\Fq$-rational minimal section (two of them, $\pm P$), and they have $\Sha\cong(\Z/2)^2$ and $\Reg=2$ where the others have $\Sha=0$ and $\Reg=8$.

\begin{table}[ht]\centering\small

\begin{tabular}{@{}lrrrr@{}}
\toprule
$\Sha(E_D)$ & $W(E_8)$ mass & $q=7$: $N$ & $q=11$: $N$ & $q=13$: $N$ \\
\midrule
$0$ & 0.82117 & 10,570 & 6,245 & 1,707 \\
$(\mathbb Z/2)^{2}$ & 0.13776 & 1,470 & 909 & 229 \\
$(\mathbb Z/3)^{2}$ & 0.02656 & 245 & 210 & 54 \\
$(\mathbb Z/4)^{2}$ & 0.01020 & 35 & 68 & 9 \\
$(\mathbb Z/5)^{2}$ & 0.00167 & 0 & 20 & 1 \\
$(\mathbb Z/2)^{4}$ & 0.00143 & 0 & 24 & 0 \\
$(\mathbb Z/6)^{2}$ & 0.00116 & 0 & 24 & 0 \\
$(\mathbb Z/2)^{2}\times(\mathbb Z/4)^{2}$ & 0.00005 & 0 & 0 & 0 \\
$(\mathbb Z/3)^{4}$ & 0.00001 & 0 & 0 & 0 \\
$(\mathbb Z/2)^{6}$ & $3.4\cdot10^{-6}$ & 0 & 0 & 0 \\
$(\mathbb Z/2)^{8}$ & $1.4\cdot10^{-9}$ & 0 & 0 & 0 \\
\bottomrule
\end{tabular}

\caption{$\Sha(E_D)$ as a group (all ranks): observed counts ($q=7$: all of $\Fam(7)$; $q=11,13$: the samples) and the $\WE$ mass of the classes with that torsion group.}\label{tab:shagroups}
\end{table}

\begin{table}[ht]\centering\scriptsize

\begin{tabular}{@{}rlrrrlrrrrr@{}}
\toprule
\# & $\chi_w$ & $r$ & $\disc E_8^w$ & $|\Sha|$ & $\Sha$ & $|C|/\sum|C'|$ & $q=7$: $N$ & $q=11$: $N$ & $q=13$: $N$ & rational minimal sections \\
\midrule
31 & $\Phi_{1}\Phi_{2}\Phi_{4}\Phi_{8}$ & 1 & 8 & 1 & $0$ & 0.800 & 1,120 & 431 & 152 & 0 \\
45 & $\Phi_{1}\Phi_{2}\Phi_{4}\Phi_{8}$ & 1 & 2 & 4 & $(\mathbb Z/2)^{2}$ & 0.200 & 154 & 81 & 35 & 2 \\
\addlinespace[2pt]
63 & $\Phi_{1}^{2}\Phi_{2}^{2}\Phi_{3}\Phi_{6}$ & 2 & 12 & 1 & $0$ & 0.750 & 252 & 93 & 41 & 2 \\
74 & $\Phi_{1}^{2}\Phi_{2}^{2}\Phi_{3}\Phi_{6}$ & 2 & 12 & 1 & $0$ & 0.167 & 105 & 38 & 19 & 0 \\
77 & $\Phi_{1}^{2}\Phi_{2}^{2}\Phi_{3}\Phi_{6}$ & 2 & 3 & 4 & $(\mathbb Z/2)^{2}$ & 0.083 & 21 & 9 & 3 & 6 \\
\addlinespace[2pt]
71 & $\Phi_{1}^{2}\Phi_{2}^{2}\Phi_{4}^{2}$ & 2 & 16 & 1 & $0$ & 0.889 & 126 & 66 & 17 & 0 \\
81 & $\Phi_{1}^{2}\Phi_{2}^{2}\Phi_{4}^{2}$ & 2 & 4 & 4 & $(\mathbb Z/2)^{2}$ & 0.111 & 0 & 3 & 0 & 4 \\
\addlinespace[2pt]
89 & $\Phi_{1}^{3}\Phi_{2}^{3}\Phi_{4}$ & 3 & 16 & 1 & $0$ & 0.947 & 14 & 27 & 7 & 4 \\
96 & $\Phi_{1}^{3}\Phi_{2}^{3}\Phi_{4}$ & 3 & 4 & 4 & $(\mathbb Z/2)^{2}$ & 0.053 & 0 & 0 & 0 & 12 \\
\addlinespace[2pt]
102 & $\Phi_{1}^{4}\Phi_{2}^{4}$ & 4 & 16 & 1 & $0$ & 0.973 & 0 & 0 & 0 & 8 \\
105 & $\Phi_{1}^{4}\Phi_{2}^{4}$ & 4 & 4 & 4 & $(\mathbb Z/2)^{2}$ & 0.027 & 0 & 0 & 0 & 24 \\
\bottomrule
\end{tabular}

\caption{Resolution of the five characteristic polynomials shared by several classes.  $|C|/\sum|C'|$: the size of the class relative to all classes with the same characteristic polynomial.  ``Rational minimal sections'' is the number of roots fixed by $w$, i.e.\ of $\Fq$-rational sections with $\deg x\le2$, $\deg y\le3$; it identifies the class.}\label{tab:ambiguous}
\end{table}

\section{Local invariants for $\deg D\le7$}\label{sec:local}

For $\deg D\ne5$ the fibre at infinity changes type, and for $2\le\deg D\le4$ it has a non-trivial Tamagawa number.  The bookkeeping is as follows.

\begin{theorem}[Theorem A(5); proven]\label{thm:local}
Let $D\in\Fq[t]$ be monic of degree $n$, $1\le n\le7$, with $\Delta_0=4+27D^2$ squarefree, and put $k=\lceil n/6\rceil$.  Then $\chi(\calO_{X_D})=\deg\omega=k$, all finite bad fibres are $\mathrm I_1$ with $c_v=1$, the fibre at $\infty$ is as in Table~\ref{tab:local}, and for rank $0$
\begin{equation}\label{eq:shageneral}
  |\Sha(E_D)|=L(E_D,1/q)\cdot|E_D(K)_{\mathrm{tors}}|^2\cdot q^{\chi(\calO_X)-1}\Big/\Bigl(c_\infty\prod_{v\ \mathrm{finite}}c_v\Bigr).
\end{equation}
\end{theorem}

\begin{table}[ht]\centering\footnotesize
\resizebox{\textwidth}{!}{%
\begin{tabular}{@{}rrlrrlrrrl@{}}
\toprule
$n$ & $j=6k-n$ & type at $\infty$ & $v_\infty(\Delta)$ & $f_\infty$ & $c_\infty$ ($D$ monic) & $\deg N$ & $\deg L$ & $\chi(\calO_X)$ & surface \\
\midrule
1 & 5 & II$^*$ & 10 & 2 & 1 & 4 & 0 & 1 & rational \\
2 & 4 & IV$^*$ & 8 & 2 & 3 & 6 & 2 & 1 & rational \\
3 & 3 & I$_0^*$ & 6 & 2 & $4$ ($q\equiv1$ mod 3), $2$ ($q\equiv2$ mod 3)$^{\dagger}$ & 8 & 4 & 1 & rational \\
4 & 2 & IV & 4 & 2 & 3 & 10 & 6 & 1 & rational \\
5 & 1 & II & 2 & 2 & 1 & 12 & 8 & 1 & rational \\
6 & 0 & I$_0$ (good: $Y^2=X^3+1$) & 0 & 0 & 1 & 12 & 8 & 1 & rational \\
7 & 5 & II$^*$ & 10 & 2 & 1 & 16 & 12 & 2 & K3 \\
\bottomrule
\end{tabular}}
\caption{The fibre at $\infty$ of $E_D$ for $D$ monic of degree $n\le7$ with $4+27D^2$ squarefree.  $^{\dagger}$: $c_\infty=1+\#\{T\in\Fq:T^3=-1\}$.}\label{tab:local}
\end{table}

\begin{proof}
Finite places: exactly as in \S\ref{sec:local1}, since $c_4=-48$ is a unit everywhere and $v(\Delta_0)\le1$.  At $\infty$ put $x=s^{-2k}X$, $y=s^{-3k}Y$: $Y^2=X^3+s^{4k}X+s^{6k-n}u(s)$ with $u(0)=1$, $j:=6k-n\in\{0,\dots,5\}$, $v_\infty(a_4)=4k\ge4$, $v_\infty(a_6)=j$, $v_\infty(\Delta)=\min(v(4s^{12k}),v(27s^{2j}u^2))=2j\le10<12$: the model is minimal.  Tate's algorithm with $p\ge5$ \cite{Tate75}, \cite[IV, \S9]{SilvermanATAEC} (conductor exponents from \cite[(8.5)]{UlmerPCMI}) gives: $j=0$: good reduction $Y^2=X^3+1$; $j=1$: II, $c=1$; $j=2$: IV, $c=3$ iff $a_6/s^2\equiv u(0)=1$ is a square mod $s$, always; $j=3$: I$_0^*$, $c=1+\#\{T\in\Fq:T^3+u(0)=0\}\in\{2,4\}$, $=4$ iff $-1$ has three cube roots in $\Fq$ iff $q\equiv1\pmod3$; $j=4$: IV$^*$, $c=3$; $j=5$: II$^*$, $c=1$.  For $j\ge1$ the reduction is additive and $p\ge5$, so $f_\infty=2$; for $j=0$, $f_\infty=0$.  Hence $\deg N=2n+2$ ($j\ge1$) or $2n$ ($j=0$), $\deg L=\deg N-4$; $\chi(\calO_X)=\deg\omega=k$ since the minimal discriminant has total degree $2n+2j=12k$; $n\le6$ gives a rational surface, $n=7$ a K3 surface.  $\omega=dx/2y=s^kdX/2Y$ at $\infty$, so $a_\infty=-k$ and $\sum_va_v\deg v=-k=-\chi$.  Finiteness of $\Sha$ and the refined formula hold for $n\le6$ by \S\ref{sec:shaorder} (Artin--Tate for rational surfaces) and for $n=7$ by the Tate conjecture for K3 surfaces over finite fields of characteristic $p\ge5$ \cite{Charles2013}, which by \S\ref{sec:BSDAT} makes $\Br(X_D)\cong\Sha(E_D)$ finite, together with Kato--Trihan; substituting $\tau=q^{1-\chi}\prod c_v$ into \eqref{eq:KT} gives \eqref{eq:shageneral}.
\end{proof}

\eqref{eq:shageneral} reduces to $|\Sha|=L(1/q)$ exactly when $|\mathrm{tors}|=1$, $\prod c_v=1$ and $\chi=1$, i.e.\ for $n=1$, $5$ and $6$ (with trivial torsion); for $n=7$ the formula is $|\Sha|=qL(1/q)\,|\mathrm{tors}|^2/\prod c_v$.  Torsion: for $n\ne6$ the fibre at $\infty$ is additive, so the prime-to-$p$ torsion of $E_D(K)$ embeds into $\Phi_\infty(\Fq)$; there is no $p$-torsion, since $j\notin K^p$ ($\Delta_0$ squarefree, or $p\nmid2n$).  Hence $E_D(K)_{\mathrm{tors}}=0$ for $n=1,5,7$ ($\Phi_\infty=0$); for $n=2,4$ ($c_\infty=3$) a point of order $3$ would have $x(P)\in\Fq[t]$ with $3x^4+6x^2+12Dx-1=0$, impossible as $3\deg x\ne n$; for $n=3$ ($\Phi_\infty\subseteq(\Z/2)^2$) a point of order $2$ is $(x_0,0)$ with $x_0=\alpha t+\beta$, $\alpha^3=-1$, $D=-(x_0^3+x_0)$, and then $4+27D^2=(3x_0^2+4)(3x_0^2+1)^2$ is not squarefree, while full $2$-torsion never occurs; for $n=6$ with $\Delta_0$ squarefree all fibres are irreducible and $E_D(\Kbar)\cong E_8$ is torsion-free as in \S\ref{sec:NSMWL}.  So the torsion is trivial under the hypotheses of Theorem~\ref{thm:local}.  If only $D$, and not $\Delta_0$, is required to be squarefree, then for $n=3$ exactly the $3q$ resp.\ $q$ curves ($q\equiv1$ resp.\ $2\bmod3$) with $D=-(x_0^3+x_0)$ have $E_D(K)_{\mathrm{tors}}\cong\Z/2$ (ancillary \texttt{torsion\_check.py}, \texttt{torsion\_f49.gp}).  \eqref{eq:shageneral} agrees with the Artin--Tate formula \eqref{eq:AT}.  At rank $0$, $P_2(X_D,T)=(1-qT)^2L(E_D,T)\prod_v\prod(1-(qT)^{\ell\deg v})$ by \eqref{eq:ZL} (inner product over the $\mathrm{Frob}_v$-orbits, of lengths $\ell$, on the components not meeting the zero section), $|\disc\NS(X_D)|=|E_D(K)_{\mathrm{tors}}|^{-2}\prod_v(\deg v)^{r_v}|\disc R_v^{\mathrm{Frob}_v}|$ with $R_v$ the root lattice of those components and $r_v$ the number of orbits, and $q^{\alpha(X_D)}=q^{\chi(\calO_{X_D})-1}$ is the power of $q$ in \eqref{eq:shageneral}; so the two formulas agree iff $c_v\prod\ell=|\disc R_v^{\mathrm{Frob}_v}|$ at every $v$.  This local identity holds for every fibre type occurring for $\deg D\le7$ (split and non-split $\mathrm I_m$, $m\le7$, and the types of Table~\ref{tab:local}); it is a finite check (ancillary \texttt{at\_kt\_local.sage}, which also covers $\mathrm I_m$ up to $m=12$, $\mathrm I_n^*$, III and III$^*$).

\subsection{Non-squarefree discriminants}\label{sec:nonsqfree}
If $\Delta_0$ is not squarefree, a zero of multiplicity $m$ gives a fibre $\mathrm I_m$ with $c_v=m$ (split) resp.\ $2$ or $1$ (non-split, $m$ even resp.\ odd), and $\deg L=\deg N-4$ drops by $\sum_v(m_v-1)\deg v$.  For $\deg D=5$ the trivial lattice becomes $T=\langle O,F\rangle\oplus\bigoplus_{\bar y}A_{m_{\bar y}-1}(-1)$, the Mordell--Weil lattice is the corresponding Oguiso--Shioda entry \cite{OguisoShioda}, Frobenius still acts on $\langle O,F\rangle^\perp=E_8(-1)$ through some $w'_D\in\WE$, and by \eqref{eq:ZL}
\[
  \Lunit(z)\cdot\prod_v(z^{\deg v}-1)^{\nu_v-1}(z^{\deg v}+1)^{g_v}=\chi_{w'_D}(z),
\]
so $\Lunit$ is still cyclotomic (as observed for the $\approx15\%$ of monic $D$ with $\Delta_0$ not squarefree at $q=13$) but is a characteristic polynomial of an element of $O(\mathrm{MWL})$, not necessarily of $\WE$; and $|\Sha|=L(1/q)|\mathrm{tors}|^2/\prod c_v$ at rank $0$ with $c_v=m_v$ or $2$ or $1$.  Such $D$ are $4{,}487$ of $7^5$, $1{,}210$ of $11^5$ and $55{,}055$ of $13^5$ (Table~\ref{tab:datasummary}); at $q=11$ only double roots at degree-$2$ places occur.  These curves are outside $\Fam(q)$ and outside the scope of Theorems~A--E.

\begin{lemma}[even-degree factors]\label{lem:evendeg}
Let $p\ge5$, $D\in\Fq[t]$ of degree $n\ge1$, $\Delta_0=4+27D^2$.  Since $-4/27=-3\cdot(2/9)^2$, $-4/27$ is a square in $\Fq$ iff $-3$ is, iff $q\equiv1\pmod3$.
\begin{enumerate}[label=\textup{(\alph*)},leftmargin=2em]
\item If $q\equiv2\pmod3$, every irreducible factor of $\Delta_0$ in $\Fq[t]$ has even degree; so $E_D$ has no bad place of odd degree, in particular none of degree $1$.
\item If $q\equiv1\pmod3$ and $a^2=-4/27$, then $\Delta_0=27(D-a)(D+a)$ with $\gcd(D-a,D+a)=1$, $\Delta_0$ is squarefree iff both factors are, and the bad places of degree $1$ are the $t_0\in\Fq$ with $D(t_0)=\pm a$, whose number has mean $2$ over all monic $D$ of degree $n$.
\end{enumerate}
\end{lemma}
\begin{proof}
(a) Let $g\mid\Delta_0$ be irreducible of odd degree $d$ and $t_0\in\mathbb F_{q^d}$ a root; then $D(t_0)^2=-4/27$, so $-4/27$ is a square in $\mathbb F_{q^d}$.  For $c\in\Fq^\times$, $c$ is a square in $\mathbb F_{q^d}$ iff $c^{(q^d-1)/2}=1$, and $c^{(q^d-1)/2}=(c^{(q-1)/2})^{1+q+\dots+q^{d-1}}$ with $1+q+\dots+q^{d-1}\equiv d\equiv1\pmod 2$, so a non-square of $\Fq$ stays a non-square in $\mathbb F_{q^d}$; contradiction.  (b) $(D-a)-(D+a)=-2a\in\Fq^\times$, and a product of coprime polynomials is squarefree iff both factors are.  For fixed $t_0$, $D(t_0)$ is a uniformly distributed function of the constant term, so each of the $q$ values $t_0$ has $D(t_0)=\pm a$ with probability $2/q$: mean $2$.
\end{proof}
At $q=7$ and $13$ ($\equiv1\pmod3$) the two-factor structure produces the $\approx2/q$ loss in \eqref{eq:counts} and degree-$1$ $\mathrm I_1$ places; at $q=11,23$ there are none.  Neither affects the $\WE$ limits of Theorem~D, which are the same for all $p$; it affects only which places carry the identities of Theorem~C.

\section{Uniformity in $p$ and the first-order term}\label{sec:uniform}

This section proves Theorems~U and~M of \S\ref{sec:main}, records a computed identity (Remark~\ref{rem:cm}), and states a conjectural first-order law (Conjecture~\ref{conj:first}) with its evidence and the gap to a theorem.

\emph{Notation.} For $c\in\Fq^\times$ let $\Fam^{(c)}(q)$ be the set of monic $D\in\Fq[t]$ of degree $5$ with $4c^3+27D^2$ squarefree, so $\Fam^{(1)}=\Fam$.  By Theorem~\ref{thm:general} with $A=c$, Theorems~A(1)--(4) and~C hold for $y^2=x^3+cx+D(t)$, $D\in\Fam^{(c)}(q)$, with the same proofs; $w_D\in\WE$ is its Frobenius and $N_C(q)=\#\{D\in\Fam^{(c)}(q):w_D\in C\}$.  Let $\chi_q$ be the quadratic character of $\Fq$, $\chi:=\chi_q(-3)$, $\chi_Z:=\chi_q(-3c)$, and $a\in\Fqbar$ with $27a^2=-4c^3$.  Then $4c^3+27D^2=27(D-a)(D+a)$, the discriminant locus is $Z=Z_+\cup Z_-\subset\A^5$ with $Z_\pm=\{\disc_t(D\mp a)=0\}$, and $a\in\Fq$ iff $\chi_Z=1$, since $-4c^3/27=-3c\,(2c/9)^2$; $\chi=1$ iff $\Fq$ contains the cube roots of unity.  All characters of $\WE$ are rational-valued.  $V$ is the reflection representation on $E_8\otimes\Q$ (character $\tr w$), and irreducible representations are named by their degrees.  From the character table: $\operatorname{Sym}^2V=1\oplus35$; $\Lambda^2V$ and $\Lambda^3V$ are irreducible, of degrees $28$ and $56$; the permutation character of $\WE$ on its $120$ reflections (by conjugation) is $\pi_{120}=1+35+84$, and $\pi_{120}(w)=\frac12\#\{\beta:w\beta=\pm\beta\}$ is the number of reflections commuting with $w$; the permutation character on the $240$ roots is $\mathrm{fix}=1+V+35+84+112$; and $V^{\otimes3}=3V+\Lambda^3V+112+2\cdot160$, with $84$, $112$, $160$ irreducible.  For an irreducible character $\rho$ put
\[
  \kappa_\rho(q):=-q^{-4}\sum_{D\in\Fam^{(c)}(q)}\rho(w_D)\quad(\rho\ne1),\qquad \kappa_1(q):=q^{-4}\bigl(q^5-|\Fam^{(c)}(q)|\bigr);
\]
then, by orthogonality and exactly,
\[
  N_C(q)=\frac{|C|}{|\WE|}\,q^5\Bigl(1-\frac{\gamma_C(q)}q\Bigr),\qquad \gamma_C(q):=\sum_\rho\rho(C)\,\kappa_\rho(q).
\]
In \S\ref{sec:thmU} only $c=1$ occurs.  Theorem~D(ii) gives $\kappa_\rho(q)=O_p(q^{1/2})$ for $\rho\ne1$.  A first-order law is a statement $\kappa_\rho(q)=\kappa^0_\rho(q)+O(q^{-1/2})$ with bounded $\kappa^0_\rho$.

\subsection{The torsor and uniformity in $p$: proof of Theorem~U}\label{sec:thmU}

\begin{proposition}[the $\WE$-torsor over $\Z{[1/6]}$; proven]\label{prop:torsor}
Let $U=\A^5_{\Z[1/6]}\smallsetminus V(\Delta)$ and $V\subset U\times\A^7$ be as in \S\ref{sec:monrep} and Lemma~\ref{lem:etale}.
\begin{enumerate}[label=\textup{(\alph*)},leftmargin=2em]
\item $V\to U$ is finite \'etale of degree $240$; in particular every geometric fibre has exactly $240$ points.
\item Fix a geometric point $\bar u$ of $U$ and a $\Z$-basis $b=(\beta_1,\dots,\beta_8)$ of the geometric Mordell--Weil lattice at $\bar u$ consisting of roots, i.e.\ a point of $V^8_{\bar u}$, where $V^8_U$ is the $8$-fold fibre power of $V$ over $U$.  The connected component $\tilde V$ of $V^8_U$ containing $b$ is a finite \'etale $\WE$-torsor over $U$.  For a $\Z[1/6]$-scheme $S$, $\tilde V\times_UU_S$ is connected iff $\rho_S(\pi_1(U_S))=\WE$; in particular $\tilde V_{\overline{\mathbb F}_p}$ is connected for every $p\ge5$, by Theorem~B(c).
\end{enumerate}
\end{proposition}
\begin{proof}
(a) is the proof of Lemma~\ref{lem:etale}, which applies verbatim over $\Z[1/6]$.  (b) $V^8_U$ is finite \'etale over $U$ by (a), and by \cite[Exp.~V]{SGA1} the connected components of $V^8_U\times_UU_S$ correspond to the orbits of $\pi_1(U_S,\bar u)$ on $V^8_{\bar u}$, on which it acts through $\rho_S$ (\S\ref{sec:monrep}).  The group $\WE$ acts faithfully and linearly on the lattice, so the stabiliser of the basis $b$ is trivial, and the orbit $\rho(\pi_1(U))\,b=\WE\cdot b$ (Theorem~B(a)) is a free transitive $\WE$-set.  The maps $wb\mapsto wgb$ ($g\in\WE$) are $\pi_1$-equivariant and act simply transitively on it, so $\tilde V$ is a $\WE$-torsor.  Over $U_S$ the orbit of $b$ is $\rho_S(\pi_1(U_S))\,b$, which is all of $\WE\cdot b$ iff $\rho_S$ is onto.
\end{proof}

\begin{proof}[Proof of Theorem~U]
\emph{Betti bound.} Let $X\subset\A^{62}_\Z$, with coordinates $a\in\A^5$, $z$, and $(x^{(k)},y^{(k)})\in\A^3\times\A^4$ for $k=1,\dots,8$, be defined by the $57$ polynomials $z\Delta(a)-1$ and, for each $k$, the seven coefficient equations of $y^{(k)2}=x^{(k)3}+x^{(k)}+D_a$.  Their degrees are $\le17$ and $\le3$, since $\deg\Delta=16$ (\S\ref{sec:equid}).  Moreover $X_{\Z[1/6]}=V^8_U$, and for $p\ge5$ the scheme $\tilde V_{\overline{\mathbb F}_p}$ is a connected component of $X_{\overline{\mathbb F}_p}$ (Proposition~\ref{prop:torsor}).  Katz's bound for an affine scheme over $\mathbb F_p$ cut out in $\A^N$ by $r$ equations of degree $\le d$, $\sum_i\dim H^i_c\le3(2+d)^{N+r}$ \cite[Remark after Lemma~13]{Katz2001} (a consequence of Theorem~10 and Lemma~13 there; numbering of the author's posted version), gives $\sum_ih^i_c(\tilde V_{\overline{\mathbb F}_p})\le\sum_ih^i_c(X_{\overline{\mathbb F}_p})\le3\cdot19^{119}$ for every $p\ge5$.

\emph{Equidistribution.} Fix $q=p^f$.  By Proposition~\ref{prop:torsor}, Theorem~B(c) and Lemma~\ref{lem:etale}, $\tilde V_{\Fq}\to U_{\Fq}$ is Galois with group $\WE=G_{\mathrm{arith}}=G_{\mathrm{geom}}$, and the Frobenius class at $D$ is that of $w_D$.  The proof of Theorem~\ref{thm:effcheb} gives
\[
  \Bigl|N_C(q)-\tfrac{|C|}{|\WE|}|\Fam(q)|\Bigr|\le\tfrac{|C|}{|\WE|}\,q^{9/2}\sum_{\rho\ne1}|\rho(C)|\,\sigma'(\rho),\qquad\sigma'(\rho):=\sum_{i\le9}h^i_c(U_{\Fqbar},\calF_\rho).
\]
For the finite map $\pi\colon\tilde V_{\Fqbar}\to U_{\Fqbar}$ we have $\pi_!\Ql=\pi_*\Ql=\bigoplus_\rho\calF_\rho^{\rho(1)}$, so $\sum_\rho\rho(1)\,h^i_c(U_{\Fqbar},\calF_\rho)=h^i_c(\tilde V_{\Fqbar})$ for every $i$.  With $|\rho(C)|\le\rho(1)$ the sum is therefore at most $\sum_{i\le9}h^i_c(\tilde V_{\overline{\mathbb F}_p})\le B$.  The monodromy enters only through $H^{10}_c(U_{\Fqbar},\calF_\rho)=0$ for $\rho\ne1$: if $\tilde V_{\overline{\mathbb F}_p}$ were disconnected, some $\rho\ne1$ would contribute a main term.  Next, $\#Z(\Fq)\le16q^4$ (\S\ref{sec:equid}).  For $q\ge32$ dividing the two estimates gives the third one with the constant $2B$; for $q<32$ its left side is at most $1<6/\sqrt{31}\le6|\WE|\,\frac{|C|}{|\WE|}\,q^{-1/2}$.  Finally, write $N_C(q)/|\Fam(q)|=m_C(1+\theta_Cd)$ with $m_C=|C|/|\WE|$, $|\theta_C|\le1$, and put $\bar\theta:=\sum_{C\subset S}m_C\theta_C/P_W(S)$.  Then $N_S(q)>0$, and
\[
  \E_q[\varphi\mid S]-\E_W[\varphi\mid S]=\frac{d\sum_{C\subset S}m_C\,\varphi(C)\,(\theta_C-\bar\theta)}{P_W(S)\,(1+\bar\theta d)},
\]
whose absolute value is at most $2\|\varphi\|_\infty d/(1-d)$.
\end{proof}

\begin{remark}[what Theorem~U does not say]\label{rem:Uscope}
The constant is astronomically large: $K\le6\cdot19^{119}\approx10^{153}$, so the bound is non-trivial only for $q>K^2$.  Theorem~U says nothing about any $q$ in our data, and nothing about the $1/q$ terms of \S\ref{sec:conj}.  Heuristically, for a tame cover $B_p$ is at least of the order of $|\WE|\cdot|\chi_c(U_{\overline{\mathbb F}_p})|$ (unless $\chi_c=0$), so no bound on Betti numbers alone can make this route say anything for $q$ below about $10^{18}$.  Theorem~U is a standard application of Katz's bounds and Deligne's theorem.  For the full parameter space of degree-$1$ del Pezzo surfaces, the rate $O(q^{-1/2})$ with an implied constant independent of $q$ is stated in \cite[Prop.~6.2]{BFL}.  The input specific to the slice is Theorem~B(c).  The characteristic-$0$ statement B(b), with constructibility (Remark~\ref{rem:uniformp}), would give the conclusion only for $p\ge p_1$ with $p_1$ ineffective.
\end{remark}

\subsection{Moments: proof of Theorem~M}\label{sec:thmM}

For every monic quintic $D\in\A^5(\Fq)$, in $Z$ or not, put
\begin{gather*}
  \mathrm{Sol}(D):=\#\{(x,y)\in\Fq[t]^2:\deg x\le2,\ \deg y\le3,\ y^2=x^3+cx+D\},\\
  T(D):=\sum_{t\in\Fq}a_{D(t)},\qquad S_2(D):=\sum_{t\in\mathbb F_{q^2}}a'_{D(t)},
\end{gather*}
where $a'_v:=-\sum_{x\in\mathbb F_{q^2}}\chi_{q^2}(x^3+cx+v)$ for $v\in\mathbb F_{q^2}$.  Let $M_k:=\sum_{b\in\Fq}a_b^k$.  We use four facts.

(F1) For $D\in\Fam^{(c)}(q)$: $\mathrm{Sol}(D)=\mathrm{fix}(w_D)$, $T(D)=-q\tr w_D$ and $S_2(D)=-q^2\tr(w_D^2)$.  The first is Theorem~A(2): the $\Fq$-rational minimal sections are the roots fixed by $w_D$.  The others are Theorem~C for $d=1,2$, since $-\sum_x\chi(x^3+cx+v)$ is the trace of the fibre: it is $\pm1$ at a node, $1$ at a node over $\mathbb F_{q^2}$ above a place of degree $1$, and $a_\infty=0$.  This is the only place where Theorems~A(2) and~C enter.  Nothing about $D\in Z$ is used beyond (F3).

(F2) For distinct $t_1,\dots,t_k\in\Fq$ with $k\le5$, and for $t\in\mathbb F_{q^2}\smallsetminus\Fq$, the maps $D\mapsto(D(t_i))_i\in\Fq^k$ and $D\mapsto D(t)\in\mathbb F_{q^2}$ are affine and surjective on $\A^5(\Fq)$ (Vandermonde), hence equidistributed.  The same holds on $\{(t-\beta)^2g+a:g\text{ monic cubic}\}$ for $k\le3$ points $t_i\ne\beta$ and for $t\notin\Fq$.  Moreover $M_1=0$ and $M_2=q^2-q-q(\chi+\chi_Z)$.  Indeed, $M_2=q\cdot\#\{f(x)=f(y)\}-q^2$ with $f(x)=x^3+cx$, and $f(x)=f(y)$ is the union of the diagonal ($q$ points) and the conic $x^2+xy+y^2=-c$ ($q-\chi$ points), which meet in the $1+\chi_Z$ points with $3x^2=-c$.

(F3) For every $D\in\A^5(\Fq)$: $|T(D)|\le20q$, $|S_2(D)|\le20q^2$ and $\mathrm{Sol}(D)\le3^7$.  Indeed, $c_4=-48c$ is a unit at every finite place.  So the Weierstrass surface of $y^2=x^3+cx+D$ has fibres of type $\mathrm I_n$ at finite places and II at $\infty$, only rational double points, and a minimal resolution $\tilde X$ that is a rational elliptic surface ($\sum_vv(\Delta)\deg v=12$).  For $Q\in\{q,q^2\}$ the trace formula gives $\#\tilde X(\mathbb F_Q)=1+\tau Q+Q^2$, with $\tau$ the trace of Frobenius on $\NS(\tilde X_{\Fqbar})\otimes\Q$ (as in \S\ref{sec:NSMWL}, Step~3), so $|\tau|\le10$.  The fibre over $t\in\mathbb F_Q$ has $Q+1-a^{(Q)}_{D(t)}+e_tQ$ points, where $a^{(q)}=a$, $a^{(q^2)}=a'$, and $0\le e_t\le n_t-1$ for an $\mathrm I_{n_t}$ fibre (count the rational components and nodes of the $n_t$-gon) and $e_t=0$ otherwise.  The fibre at $\infty$ has $Q+1$ points, and $\sum_t(n_t-1)\le8$ (Shioda--Tate).  Summing gives $S_Q(D)=Q(2+E-\tau)$ with $0\le E\le8$, where $S_q=T$ and $S_{q^2}=S_2$.  Finally, the points of $V_D(\Fqbar)$ are sections of $\tilde X$ disjoint from $O$, hence of height $\le2$.  There are finitely many, and B\'ezout bounds their number by $3^7$.

(F4) $Z_+$ and $Z_-$ are geometrically irreducible, being images of $(\beta,g)\mapsto(t-\beta)^2g\pm a$.  They are distinct: on the line $D=t^5+\varepsilon t^2+a_0$ ($\varepsilon=0$ for $p\ne5$, $\varepsilon=1$ for $p=5$) the only point of $Z_\pm$ is $a_0=\pm a$.  If $a\in\Fq$, then $\#Z_\pm(\Fq)=q^4$, since there are $q^n-q^{n-1}$ squarefree monic polynomials of degree $n\ge2$ over $\Fq$.  If $a\notin\Fq$, Frobenius exchanges $Z_+$ and $Z_-$, so $Z(\Fq)\subset Z_+\cap Z_-$.  In both cases $\#(Z_+\cap Z_-)(\Fq)\le320q^3$.  To see this, note that the discriminant of a monic quintic is weighted homogeneous of weight $20$, with $a_i$ of weight $5-i$.  So $P_\pm:=\disc_t(D\mp a)$ has degree $5$ in $a_1$, with leading coefficient $4^4$.  $P_+$ and $P_-$ have no common factor ($Z_\pm$ being distinct irreducible hypersurfaces), so their resultant in $a_1$ is a non-zero polynomial of degree $\le64$ in $(a_0,a_2,a_3,a_4)$.  It vanishes at $\le64q^3$ points of $\Fq^4$ (Schwartz--Zippel), and over each of these there are $\le5$ common roots $a_1$.

\begin{proof}[Proof of Theorem~M]
(a) We have $|\Fam^{(c)}(q)|=q^5-\#Z(\Fq)$, and by (F4) $\#Z(\Fq)=\#Z_+(\Fq)+\#Z_-(\Fq)-\#(Z_+\cap Z_-)(\Fq)$ is $2q^4+O(q^3)$ if $\chi_Z=1$ and $O(q^3)$ otherwise.

\emph{Sums over $Z$.} Suppose $\chi_Z=1$.  The map $(\beta,g)\mapsto(t-\beta)^2g+a$ ($\beta\in\Fq$, $g$ a monic cubic over $\Fq$) is onto $Z_+(\Fq)$ except at the $\le q^3$ points $D$ with $D-a=h^2\ell$, $h$ an irreducible quadratic.  It is injective except over the $\le q^3$ points $D$ with $D-a=(t-\beta_1)^2(t-\beta_2)^2\ell$.  Hence, by (F3) and (F4), for $f\in\{T,T^2,T^3,S_2\}$ with $|f|\le F$ on $\A^5(\Fq)$,
\[
  \sum_{Z(\Fq)}f=\sum_\pm\sum_{\beta,g}f\bigl((t-\beta)^2g\pm a\bigr)+O(q^3F)\quad(\chi_Z=1),\qquad\sum_{Z(\Fq)}f=O(q^3F)\quad(\chi_Z=-1).
\]

(c), (d), (f).  Expand $T^k=\sum_{t_1,\dots,t_k}\prod_ia_{D(t_i)}$.  By (F2) and $M_1=0$ only the terms with $t_1=\dots=t_k$ survive, so $\sum_{\A^5(\Fq)}T^k=q^5M_k$ for $k\le3$.  For $D=(t-\beta)^2g+a$ we have $T(D)=a_++Y$, with $a_+:=a_a=\pm1$ (the node over $t=\beta$) and $Y:=\sum_{t\ne\beta}a_{D(t)}$.  The same expansion over $g$ gives $\sum_gY=0$, $\sum_gY^2=(q-1)q^2M_2$ and $\sum_gY^3=(q-1)q^2M_3$.  Using $M_2=q^2+O(q)$ and $|M_3|\le2q^2+4q$ (from the second part of (f), proved below), we get $\sum_ZT=O(q^4)$, $\sum_ZT^2=(1+\chi_Z)q^6+O(q^5)$ and $\sum_ZT^3=O(q^6)$.  By (F1), $\sum_{\Fam^{(c)}(q)}(\tr w_D)^k=(-q)^{-k}\bigl(\sum_{\A^5(\Fq)}-\sum_{Z(\Fq)}\bigr)T^k$.  This gives (c) and (f), and also $\sum(\tr w_D)^2=q^3M_2-(1+\chi_Z)q^4+O(q^3)$, which is (d).

(e) For $t\in\mathbb F_{q^2}\smallsetminus\Fq$, $D(t)$ is equidistributed on $\mathbb F_{q^2}$ and $\sum_{v\in\mathbb F_{q^2}}a'_v=0$.  So only $t\in\Fq$ contribute to $\sum_{\A^5(\Fq)}S_2$, which equals $q^5m_2$ with $m_2:=\sum_{b\in\Fq}a'_b$.  For $b\in\Fq$, $a'_b=a_b^2-2q$ if $y^2=x^3+cx+b$ is smooth, and $a'_b=1=a_b^2$ at the $1+\chi_Z$ values with $4c^3+27b^2=0$ (a node is split over $\mathbb F_{q^2}$).  Hence
\[
  m_2=M_2-2q^2+2q(1+\chi_Z)=-q^2+(1-\chi+\chi_Z)q.
\]
For $D=(t-\beta)^2g+a$, the point $t=\beta$ contributes $a'_a=1$; after summing over $g$, each $t\in\Fq\smallsetminus\{\beta\}$ contributes $q^2m_2$ and each $t\notin\Fq$ contributes $0$.  So $\sum_{\beta,g}S_2=q^4+q(q-1)q^2m_2=-q^6+O(q^5)$, and $\sum_ZS_2=-(1+\chi_Z)q^6+O(q^5)$.  Therefore $\sum_{\Fam^{(c)}(q)}\tr(w_D^2)=-q^{-2}\bigl(q^5m_2+(1+\chi_Z)q^6\bigr)+O(q^3)$, which is (e).

(b) Compare the coefficients of $t^6$ and $t^5$ in $y^2-x^3-cx$.  They force $x_2\ne0$, since otherwise $y_3=0$ and the $t^5$-coefficient vanishes.  They also force $(x_2,y_3)=(s^2,s^3)$ with $s=y_3/x_2\in\Fq^\times$, and $y_2=(1+3s^4x_1)/(2s^3)$.  Then $(s,x_1,x_0,y_1,y_0)\in\Fq^\times\times\Fq^4$ is free and determines $D=y^2-x^3-cx$.  So $\sum_{\A^5(\Fq)}\mathrm{Sol}=(q-1)q^4$ exactly.  If $\chi_Z=1$, count the pairs (section, $\beta$) with $\beta\in\Fq$ and $(t-\beta)^2\mid y^2-x^3-cx-a$.  For fixed $(s,\beta)$, the map $(x_1,x_0,y_1,y_0)\mapsto(X,X',Y,Y'):=(x(\beta),x'(\beta),y(\beta),y'(\beta))$ is an affine bijection of $\Fq^4$.  The condition reads $Y^2=f(X)$, $2YY'=f'(X)X'$, with $f=X^3+cX+a=(X-r)^2(X-e)$.  The node $(r,0)$ gives $q^2$ solutions, and each of the $q-1-\chi_q(r-e)$ other affine points of $Y^2=f(X)$ gives $q$.  So there are $q(q-1)\bigl(q^2+q(q-1-\chi_q(r-e))\bigr)=2q^4+O(q^3)$ pairs.  Each $D\in Z_+(\Fq)$ is counted once for each rational double root of $D-a$, so once except for $O(q^3)$ points $D$; and $\mathrm{Sol}\le3^7$ by (F3).  Thus $\sum_{Z(\Fq)}\mathrm{Sol}=2(1+\chi_Z)q^4+O(q^3)$, using (F4) if $\chi_Z=-1$, and (b) follows from (F1).

(f), second part.  Let $j$ be the inclusion of the complement of the two geometric points $4c^3+27b^2=0$ in $\A^1_b$, and $R^1$ the first cohomology sheaf of $y^2=x^3+cx+b$ over it.  At a good $b$ we have $a_b^3=\tr(\Frob_b\mid\operatorname{Sym}^3R^1)+2qa_b$.  At a bad $b$ the fibre is nodal, so the local monodromy is a non-trivial unipotent and the stalk of $j_*\operatorname{Sym}^3R^1$ is a line with Frobenius eigenvalue $a_b^3$.  At $\infty$ the local monodromy has order $6$ (put $x=b^{1/3}X$, $y=b^{1/2}Y$ to get $Y^2=X^3+cb^{-2/3}X+1$).  Its eigenvalues on $R^1$ are $\zeta_6^{\pm1}$, so there are no invariants on $\operatorname{Sym}^3$.  All ramification is tame, as $p\ge5$.  The Zariski closure of the geometric monodromy contains a non-trivial unipotent element and has semisimple identity component \cite{DeligneWeilII}, so it is $\mathrm{SL}_2$.  Hence $H^0_c=H^2_c=0$, and the Grothendieck--Ogg--Shafarevich formula gives $\chi_c(\A^1_{\Fqbar},j_*\operatorname{Sym}^3R^1)=4\cdot(1-2)+2=-2$.  So $H^1_!:=H^1_c(\A^1_{\Fqbar},j_*\operatorname{Sym}^3R^1)$ is $2$-dimensional.  As there are no invariants at $\infty$, it is the cohomology of the middle extension to $\Pp$, which is pure of weight $4$ \cite{DeligneWeilII}.  The trace formula and $M_1=0$ give
\[
  M_3=\sum_b\tr\bigl(\Frob_b\mid(j_*\operatorname{Sym}^3R^1)_b\bigr)+2q\sum_{b\text{ good}}a_b=-\tr(\Frob_q\mid H^1_!)-2q\sum_{b\text{ bad}}a_b,
\]
and in particular $|M_3|\le2q^2+4q$.
\end{proof}

All implied constants above are absolute, so Theorem~M is uniform in $p$.  For $q\equiv3\pmod4$, $a_{-b}=-a_b$, so $M_3=0$.  For $c=1$ and such $q$, moreover, $\sum_D(\tr w_D)^3=0$ exactly (Proposition~\ref{prop:twist}(b)).  In terms of the $\kappa_\rho$, and using the decompositions listed at the start of this section, Theorem~M says that, up to $O(1/q)$,
\[
  \kappa_1=\kappa_{35}=1+\chi_Z,\quad\kappa_V=0,\quad\kappa_{\Lambda^2V}=\chi,\quad\kappa_{84}+\kappa_{112}=1,\quad\kappa_{\Lambda^3V}+\kappa_{112}+2\kappa_{160}=M_3/q^2.
\]
For $c=1$ and $q\equiv3\pmod4$ every odd $\kappa_\rho$ vanishes by Proposition~\ref{prop:twist}(b), so $\kappa_{84}=1+O(1/q)$.  The sign of the $\Lambda^2V$ coefficient is $\chi$ for every $c$, while that of $35$ is $\chi_Z$; Theorem~M therefore separates the two characters.  For $p\equiv2\pmod3$ both signs flip with the parity of $f$.  For example, $\E_q[(\tr w_D)^2]-1$ is $+1/5$ at $q=5$ and $-3/25$ at $q=25$ to first order; the exact values are $0.213$ and $-0.117$.

\emph{Checks (computed).} An independent enumeration verified, in $26$ runs with $q\le43$ (including $q=25$), all four sign pairs $(\chi,\chi_Z)$, $c=1$ and $c$ a non-square, the exact identities $\sum_{\A^5}T=0$, $\sum_{\A^5}T^2=q^5M_2$, $\sum_{\A^5}T^3=q^5M_3$, $\sum_{\A^5}\mathrm{Sol}=(q-1)q^4$, $\sum_{\A^5}S_2=q^5m_2$ and the formula for $M_2$.  It also verified the pair count at $q=7,11,13,19$.  The remainders (exact minus main term)$/q^3$ in (a)--(f) are at most $13.1$ in absolute value in all runs.  At $q=7$, $\sum_{\Fam(7)}\mathrm{fix}(w_D)=8{,}050$ agrees with the sections data of Section~\ref{sec:arithstat}.

\subsection{A computed CM identity}\label{sec:cm}

\begin{remark}[the CM identity; computed, not proved]\label{rem:cm}
Let $c=1$, and let $\eta(4z)^6=\sum_{n\ge1}\lambda_ne^{2\pi inz}$ be the newform of weight $3$ and level $16$ with CM by $\Q(i)$.  For all $428$ primes $5\le p<3000$ ($211$ of them with non-zero value) we computed $\tr(\Frob_p\mid H^1_!)=-M_3-2p\sum_{b\text{ bad}}a_b$ exactly and found
\[
  \tr(\Frob_p\mid H^1_!)=\chi_p(-3)\,p\,\lambda_p .
\]
For $5\le p\le17$ we also found $\tr(\Frob_{p^2}\mid H^1_!)=\tr(\Frob_p\mid H^1_!)^2-2\chi_{-4}(p)p^4$, as expected if $\det(\Frob_p\mid H^1_!)=\chi_{-4}(p)p^4$.  We have not proved the identity; a Faltings--Serre--Livn\'e comparison would be the natural route.  For $q=p^f$ put $\lambda_q:=\alpha^f+\beta^f$, where $\alpha+\beta=\lambda_p$ and $\alpha\beta=\chi_{-4}(p)p^2$.  If the identity holds, then $M_3/q^2=-\chi_q(-3)\lambda_q/q+O(1/q)$, and $\lambda_q=0$ for $q\equiv3\pmod4$.
\end{remark}

\subsection{The first-order law}\label{sec:conj}

\begin{conjecture}[first-order law]\label{conj:first}
Let $q=p^f$ with $p\ge5$ and $c\in\Fq^\times$.  For every conjugacy class $C\subset\WE$,
\begin{gather*}
  N_C(q)=\frac{|C|}{|\WE|}\,q^5\Bigl(1-\frac{\gamma^0_C(q)}q\Bigr)+O(q^{7/2}),\\
  \gamma^0_C(q):=\pi_{120}(C)+\chi_q(-3c)\,\chi_{\operatorname{Sym}^2V}(C)+\chi_q(-3)\,\chi_{\Lambda^2V}(C)+\frac{M_3}{q^2}\,\chi_{\Lambda^3V}(C).
\end{gather*}
Equivalently, up to $O(q^{-1/2})$: $\kappa_1=\kappa_{35}=1+\chi_q(-3c)$, $\kappa_{84}=1$, $\kappa_{\Lambda^2V}=\chi_q(-3)$, $\kappa_{\Lambda^3V}=M_3/q^2$, and $\kappa_\rho=0$ for the other $107$ irreducible $\rho$.
\end{conjecture}

For $c=1$, Remark~\ref{rem:cm} (computed) and $\chi_{\operatorname{Sym}^2V}+\chi_{\Lambda^2V}=\tr^2$ give
\[
  \gamma^0_C(q)=\pi_{120}(C)+\chi_q(-3)\bigl[(\tr w)^2-(\lambda_q/q)\,\chi_{\Lambda^3V}(w)\bigr]+O(1/q)\qquad(w\in C).
\]
For $q\equiv3\pmod4$ this is $\pi_{120}\pm\tr^2$, with $+$ for $q\equiv1$ and $-$ for $q\equiv2\pmod3$.  Since $\E_W[\gamma^0_C]=1+\chi_q(-3c)$, normalising by $|\Fam^{(c)}(q)|$ instead of $q^5$ replaces $\gamma^0_C$ by $\gamma^0_C-1-\chi_q(-3c)$.  The conjecture is consistent with Theorem~M, which proves it for $\rho=1,V,35,\Lambda^2V$ (with $O(1/q)$ in place of $O(q^{-1/2})$) and for the two sums displayed after the proof of Theorem~M.  Its content is the value of $\kappa_\rho$ for the other irreducibles, including $84$, $112$, $160$ and $\Lambda^3V$ individually.  For the rank statistics and $c=1$ it predicts, for example, $q(M_0(q)+\mu_0)=-1.0164+0.106\,M_3/q^2+O(q^{-1/2})$ for $q\equiv1\pmod3$ and $+0.0484+0.106\,M_3/q^2+O(q^{-1/2})$ for $q\equiv2\pmod3$, where $M_3=0$ for $q\equiv3\pmod4$ (exact rational coefficients in the ancillary files).

\begin{remark}[heuristic origin of the four terms]\label{rem:heuristic}
By the trace formula and Deligne's weight bound, the $q^4$-terms of $\sum_D\rho(w_D)$ come from $\mathrm{Gr}^W_8$ of $H^9_c(U_{\Fqbar},\calF_\rho)$ and of $H^8_c$.
\begin{enumerate}[label=\textup{(\roman*)},leftmargin=2em]
\item (Proved; standard.) By Poincar\'e duality, $\mathrm{Gr}^W_8H^9_c(U,\calF_\rho)$ is dual to $\mathrm{Gr}^W_2H^1(U,\calF_\rho^\vee)(5)$, which is $\Hom_{\WE}\bigl(\rho,\mathrm{Gr}^W_2H^1(\tilde V_{\Fqbar})\bigr)(5)$.  By the Kummer sequence, $\mathrm{Gr}^W_2H^1(\tilde V_{\Fqbar},\Ql)=M\otimes\Ql(-1)$, where $M:=(\calO(\tilde V_{\Fqbar})^\times/\Fqbar^\times)\otimes\Q$ is the unit module of the cover.  $M$ embeds in the permutation module on the boundary divisors of a normal compactification.  So Frobenius acts on it through permutations, the weight-$8$ eigenvalues of $H^9_c$ are $q^4$ times roots of unity, and a non-Tate first-order term such as the $\Lambda^3V$ term can only come from $H^8_c$, i.e.\ from transcendental classes.
\item (Heuristic.) The data suggest $M\cong\pi_{120}\oplus\operatorname{Sym}^2V\oplus\Lambda^2V$, with Frobenius acting on the three summands trivially, by $\chi_Z$ and by $\chi$, and $H^8_c$ contributing $\Lambda^3V\otimes H^1_!$.  First, $\pi_{120}$: in Shioda's description of the full family by $\WE$-invariants \cite{Shioda1991}, the $120$ root linear forms are units on the cover.  (This is only a guide, since Shioda's model is not used in positive characteristic.)  Second, $\operatorname{Sym}^2V\otimes\chi_Z$: at a generic point of $Z_\pm$ two $\mathrm I_1$ fibres merge into an $\mathrm I_2$ fibre, and the local monodromy is the reflection in the vanishing root (Picard--Lefschetz).  A unit separating the two components is $\disc_t(D-a)/\disc_t(D+a)$, and Frobenius fixes $Z_+$ and $Z_-$ iff $\chi_Z=1$.  Third, $\Lambda^2V\otimes\chi$: compactify $\{a_4=0\}\cong\A^4$ ($p\ge7$) in $\mathbb P(1,12,18,24,30)$.  Near the divisor at infinity a rescaling degenerates $E_D$ to the $j=0$ surface $Y^2=X^3+D(T)$, and the local monodromy $h$ has order $6$, $h^3=-1$ and characteristic polynomial $\Phi_6^4$.  So $h$ is $\zeta_6$-regular in Springer's sense: its eigenspace has dimension $4=\#\{12,18,24,30\}$, the number of degrees of $\WE$ divisible by $6$ \cite[Thm.~1.9]{Bessis}.  Frobenius acts on the conjugates of $h$ by $h\mapsto h^q$, trivially if $q\equiv1$ and by $h\mapsto h^{-1}$ if $q\equiv2\pmod3$.  The $C_{\WE}(h)$-invariant line of $\Lambda^2V$, spanned by the complex structure $J_h$, changes sign under $h\mapsto h^{-1}$.  Fourth, $\Lambda^3V\otimes H^1_!$ would need transcendental $H^2$ of the compactified cover, which we have not identified.
\end{enumerate}
\end{remark}

\emph{Evidence} (computed data; the comparison is numerical).  Exact class counts $N_C(q)$ for all of $\Fam(q)$ were computed at the twelve prime powers $5\le q\le41$, and at six pairs $(q,c)$ with $c$ a non-square (ancillary \texttt{finite\_q/}).  The class of $w_D$ is determined by the pair $(\chi_{w_D},\mathrm{fix}(w_D))$, which separates the $112$ classes.  Here $\chi_{w_D}$ comes from the Lefschetz sums $S_d$, $d\le4$ (Theorem~C), and $\det w_D$; $\mathrm{fix}(w_D)$ comes from a direct count of the rational minimal sections.  $\det w_D$ is needed only when $e_4(w_D)=0$, for $25$--$29\%$ of the curves.  There it was taken from a root-number rule that is not proved, but agrees with an independent determination on every curve where one exists: all such curves at $q=5,7$, the section samples at $q=11,13$, the curves of $\Fam(11)$, $\Fam(13)$ whose $\varepsilon$ is fixed by $p_5$, and random subsets at $17\le q\le31$.  The counts agree with the Frobenius classes of Section~\ref{sec:arithstat} on all of $\Fam(7)$ and on the samples at $q=11,13$.  They agree with the point-count data on all of $\Fam(11)$, $\Fam(13)$ and on the $q=23$ sample.  They agree with degree-$5$ Lefschetz sums on random subsets for $17\le q\le31$; there is no degree-$5$ check at $q=37,41$.  They also agree with an independent count of $|\Fam(q)|$, and satisfy $N_C=N_{-C}$ exactly for $q\equiv3\pmod4$.  Table~\ref{tab:kappa} compares the $\kappa_\rho(q)$ with the conjecture, and Table~\ref{tab:cfamily} tests the two signs separately.  The coefficients of the five named irreducibles approach the conjectured values in both residue classes mod~$3$.  The largest $|\kappa_\rho|$ among the other $107$ is at most $0.33$ for $q\ge17$, and about $0.2$ at $q=37,41$.  The non-periodic odd part is carried by $\Lambda^3V$ alone: for instance $q=17$ and $q=29$, both $\equiv5\pmod{12}$, give $-1.76$ and $+1.45$.  The law also accounts for the value $q(M_0(q)+\mu_0)=-0.116$ at $q=17$ (predicted $-0.14$): this is the CM term.

\begin{table}[ht]\centering\small
\begin{tabular}{@{}rr rrrr rr rr@{}}
\toprule
$q$ & $q\bmod12$ & $\kappa_{1}$ & $\kappa_{35}$ & $\kappa_{84}$ & $\kappa_{\Lambda^2V}$ & $\kappa_{\Lambda^3V}$ & pred. & $\max_{\text{other}}|\kappa_\rho|$ & $\|R\|_2$ \\
\midrule
5 & 5 & $+0.16$ & $-0.02$ & $+1.09$ & $-1.01$ & $-1.17$ & $-1.20$ & $0.88$ & $2.120$ \\
7 & 7 & $+1.87$ & $+1.31$ & $+0.47$ & $+0.78$ & $0$ & $0$ & $0.64$ & $1.499$ \\
11 & 11 & $+0.08$ & $0$ & $+1.08$ & $-0.98$ & $0$ & $0$ & $0.64$ & $1.655$ \\
13 & 1 & $+1.93$ & $+1.59$ & $+0.69$ & $+0.87$ & $-0.67$ & $-0.77$ & $0.44$ & $1.142$ \\
17 & 5 & $+0.06$ & $-0.00$ & $+1.05$ & $-1.00$ & $-1.76$ & $-1.76$ & $0.17$ & $0.610$ \\
19 & 7 & $+1.95$ & $+1.71$ & $+0.78$ & $+0.91$ & $0$ & $0$ & $0.33$ & $0.679$ \\
23 & 11 & $+0.04$ & $-0.00$ & $+1.04$ & $-1.00$ & $0$ & $0$ & $0.14$ & $0.369$ \\
25 & 1 & $+1.96$ & $+1.78$ & $+0.82$ & $+0.93$ & $+0.51$ & $+0.56$ & $0.26$ & $0.702$ \\
29 & 5 & $+0.03$ & $-0.00$ & $+1.03$ & $-1.00$ & $+1.45$ & $+1.45$ & $0.31$ & $0.519$ \\
31 & 7 & $+1.97$ & $+1.81$ & $+0.84$ & $+0.94$ & $0$ & $0$ & $0.25$ & $0.895$ \\
37 & 1 & $+1.97$ & $+1.85$ & $+0.88$ & $+0.95$ & $+1.80$ & $+1.89$ & $0.19$ & $0.475$ \\
41 & 5 & $+0.02$ & $-0.00$ & $+1.02$ & $-1.00$ & $+0.44$ & $+0.44$ & $0.20$ & $0.764$ \\
\bottomrule
\end{tabular}

\caption{First-order coefficients $\kappa_\rho(q)$ for $\Fam(q)$ ($c=1$), from the exact class counts, for the five irreducibles of Conjecture~\ref{conj:first}, together with the largest $|\kappa_\rho|$ over the other $107$.  Conjectured values: $(\kappa_1,\kappa_{35},\kappa_{84},\kappa_{\Lambda^2V})=(2,2,1,+1)$ for $q\equiv1$ and $(0,0,1,-1)$ for $q\equiv2\pmod3$; $\kappa_{\Lambda^3V}$ as in the column ``pred.''; all others $0$.  ``pred.'' is $-\chi_q(-3)\lambda_q/q$ (Remark~\ref{rem:cm}), the form used for the out-of-sample predictions.  It differs from $M_3/q^2$ by $O(1/q)$, and only for $q\equiv1\pmod3$: $M_3/q^2=-0.46,+0.72,+1.78$ at $q=13,25,37$.  $\kappa_{\Lambda^3V}=0$ exactly for $q\equiv3\pmod4$ (Proposition~\ref{prop:twist}(b)).  $\|R\|_2$ is the Euclidean norm, over all $112$ irreducibles, of $\kappa_\rho(q)$ minus its conjectured value.}\label{tab:kappa}
\end{table}

\begin{table}[ht]\centering\scriptsize
\begin{tabular}{@{}rr rr rrrrr r rrr@{}}
\toprule
$q$ & $c$ & $\chi_q(-3c)$ & $\chi_q(-3)$ & $\kappa_{1}$ & $\kappa_{35}$ & $\kappa_{84}$ & $\kappa_{\Lambda^2V}$ & $\kappa_{\Lambda^3V}$ & $M_3^{(c)}/q^2$ & $\|R\|_2$ & alt.\ 1 & alt.\ 2 \\
\midrule
11 & 2 & $+1$ & $-1$ & $+1.92$ & $+1.53$ & $+0.71$ & $-0.83$ & $0$ & $0$ & $1.66$ & $2.46$ & $2.92$ \\
17 & 3 & $+1$ & $-1$ & $+1.94$ & $+1.68$ & $+0.75$ & $-0.89$ & $-0.83$ & $-0.71$ & $0.83$ & $2.05$ & $2.68$ \\
23 & 5 & $+1$ & $-1$ & $+1.96$ & $+1.76$ & $+0.81$ & $-0.92$ & $0$ & $0$ & $0.52$ & $1.99$ & $2.67$ \\
7 & 3 & $-1$ & $+1$ & $+0.12$ & $-0.01$ & $+1.08$ & $+0.98$ & $0$ & $0$ & $1.10$ & $2.26$ & $2.96$ \\
13 & 2 & $-1$ & $+1$ & $+0.07$ & $-0.00$ & $+1.06$ & $+0.99$ & $+1.84$ & $+1.85$ & $1.16$ & $2.31$ & $3.01$ \\
19 & 2 & $-1$ & $+1$ & $+0.05$ & $-0.00$ & $+1.04$ & $+1.00$ & $0$ & $0$ & $0.54$ & $2.07$ & $2.85$ \\
\bottomrule
\end{tabular}

\caption{The families $y^2=x^3+cx+D$ with $c$ a non-square in $\Fq$, for which $\chi_q(-3c)=-\chi_q(-3)$.  The $\operatorname{Sym}^2V$ part ($\kappa_1$, $\kappa_{35}$) follows $\chi_q(-3c)$, i.e.\ whether $Z_\pm$ are defined over $\Fq$, and the $\Lambda^2V$ part follows $\chi_q(-3)$, the inertia at $\infty$.  $\|R\|_2$ is as in Table~\ref{tab:kappa}, against Conjecture~\ref{conj:first} with $\kappa_{\Lambda^3V}=M_3^{(c)}/q^2$.  In ``alt.~1'' the sign of $\Lambda^2V$ is replaced by $\chi_q(-3c)$; in ``alt.~2'' the sign of $\operatorname{Sym}^2V$ is replaced by $\chi_q(-3)$.}\label{tab:cfamily}
\end{table}

\emph{Out-of-sample predictions.}  The law was formulated from the data at $q\le31$.  Its predictions for $q=41$ were recorded before the $q=41$ counts had been computed (file timestamps of the ancillary data and scripts). For $q=37$ the counts had already been computed, though according to our working notes not yet examined, when the predictions were recorded, so the $q=37$ comparison is the weaker test.  The law (in the form of Table~\ref{tab:kappa}, normalised by $|\Fam(q)|$) predicted, at $q=37$ ($\lambda_{37}=-70$): $\kappa_{\Lambda^3V}=+1.892$, $q\cdot\mathrm{TV}=1.10$, $q(M_0+\mu_0)=-0.84$, $q(M_1+\mu_1)=-0.56$, $q(P_0-P_W(0))=+0.30$.  Here TV is the total-variation distance between the empirical class law and $|C|/|\WE|$, and $P_0$ the proportion of rank~$0$.  The data gave $+1.796$, $1.098$, $-0.861$, $-0.531$, $+0.302$.  At $q=41$ ($\lambda_{41}=18$) the predictions were $+0.439$, $0.45$, $+0.09$, $-0.07$, $+0.16$, and the data gave $+0.440$, $0.518$, $+0.290$, $-0.230$, $+0.219$.  At $q=41$ the class-level prediction for $\Lambda^3V$ is accurate to $10^{-3}$, but the four aggregate statistics missed by $0.06$ to $0.20$, well outside the agreement at $q=37$.  Note that $41\equiv1\pmod5$, like $q=11$, where misses of the same kind occur.  At $q=11$, for example, $q\cdot\mathrm{TV}=0.718$ against $0.435$ and $q(P_0-P_W(0))=0.445$ against $0.146$.  The residual at the three $q\equiv1\pmod5$ in range is $\|R\|_2=1.655$, $0.895$, $0.764$ at $q=11,31,41$, against $0.37$--$0.52$ at $q=23,29,37$.  It decays roughly like $q^{-1/2}$; the value expected in advance under that hypothesis was $0.78$ at $q=41$.  This favours a lower-order ($\mathrm{weight}\le7$) contribution inside the error term of Conjecture~\ref{conj:first} over an additional first-order term at $q\equiv1\pmod5$, but three values of $q$, of which only $11$ and $41$ lie in the same class mod~$3$, cannot exclude the latter.

\emph{Error model.} Conjecture~\ref{conj:first} asserts $\kappa_\rho(q)=\hat\kappa_\rho(q)+O(q^{-1/2})$ for every irreducible $\rho$, where $\hat\kappa_\rho(q)$ is the conjectured first-order coefficient; equivalently $N_C(q)=(|C|/|\WE|)\,q^5\bigl(1-c_C(q)/q\bigr)+O(q^{7/2})$, the error coming from contributions of weight at most $7$ to the cohomology of the local systems $\mathcal F_\rho$ (heuristic). Under this model $\sqrt q\,\|R(q)\|_2$ stays bounded. In Table~\ref{tab:kappa} it lies between $1.8$ ($q=23$) and $5.5$ ($q=11$) for all twelve $q\le41$; its three largest values occur at the three $q\equiv1\pmod5$ ($5.5$, $5.0$, $4.9$ at $q=11,31,41$), and the next at $q=5$ ($4.7$). By contrast $q\,\|R(q)\|_2$ grows, from $8.5$ at $q=23$ to $31$ at $q=41$, so the data do not support an error term smaller than $q^{-1/2}$. To first order the aggregate statistics are fixed linear combinations of the $\kappa_\rho$, so their deviations from the predicted first-order values are also of order $q^{-1/2}$; the $q=41$ misses ($0.06$ to $0.20$, with $\|R\|_2=0.76$) are of that size \textup{(Numerical; the model is heuristic.)} Neither order is proved; the gap is (G1).

\emph{What is not explained.}  (i) The prime $p=5$: $q=5$ has the largest residual ($\|R\|_2=2.12$), and $q(M_0+\mu_0)=+0.65$ against a prediction of $-0.07$, while $q=25$ fits.  (ii) For $q\equiv1\pmod3$ the approach is slow, roughly $\kappa_{35}\approx2-6/q$, $\kappa_{84}\approx1-5/q$ and $\kappa_{567}\approx-6.8/q$ (numerical).  These look like second-order Tate terms, and they make the $q\equiv1\pmod3$ rank statistics approach their predictions only at rate about $1/q$ (e.g.\ $q(M_0+\mu_0)=-0.93$ against $-1.25$ at $q=7$).  (iii) The $q\equiv1\pmod5$ excess just described.  All negative statements (e.g.\ ``$\kappa_\rho\to0$ for the other $107$'') refer to the tested range $q\le41$.

\emph{The gap to a theorem.}  Deligne's bound gives $O(q^{9/2})$ in Conjecture~\ref{conj:first}, which makes a first-order statement empty; a proof needs an error $O(q^{7/2})$, or at least $o(q^4)$.  Two things are missing:
\begin{enumerate}[label=\textup{(G\arabic*)},leftmargin=3em]
\item no weight $9$: $\mathrm{Gr}^W_9H^9_c(U_{\overline{\mathbb F}_p},\calF_\rho)=0$ for every $\rho\ne1$; by duality, $\mathrm{Gr}^W_1H^1(\tilde V_{\overline{\mathbb F}_p})$ has no $\rho$-isotypic part;
\item the weight-$8$ parts: $\mathrm{Gr}^W_8H^9_c$ is given by the unit module $M$ of Remark~\ref{rem:heuristic}(i), and one must show $M\cong\pi_{120}\oplus\operatorname{Sym}^2V\oplus\Lambda^2V$ with the stated Frobenius action; and $\mathrm{Gr}^W_8H^8_c$ must be identified with the $\rho$-part of a Tate twist of $\Lambda^3V\otimes H^1_!$.
\end{enumerate}
Both are known only on the span of Theorem~M.  Suppose (G1) and (G2) hold for all $p\ge5$.  Then the part of weight $\le7$ in $\sum_i(-1)^i\tr(F\mid H^i_c(U,\calF_\rho))$ is at most $\sigma'(\rho)q^{7/2}$, and $\sum_\rho\rho(1)\sigma'(\rho)\le B$ by the proof of Theorem~U.  So the conjecture would follow with an absolute implied constant.  Two routes seem possible: compute the class group and the transcendental $H^2$ of a compactified cover, or extend Theorem~M to all products of power sums of total degree $\le5$.  The second needs the higher moments of the one-parameter family and reaches only the constituents of $V^{\otimes k}$, $k\le5$.

\section{Open questions}\label{sec:open}

\begin{enumerate}[label=\textup{(Q\arabic*)},leftmargin=3em,itemsep=2pt]
\item \emph{The second-order term.} Is Conjecture~\ref{conj:first} true?  Section~\ref{sec:conj} reduces it to two statements, both known only on the span of Theorem~M.  The first, (G1), is the absence of weight $9$ in $H^9_c(U_{\overline{\mathbb F}_p},\calF_\rho)$ for $\rho\ne1$.  The second, (G2), is the identification of the weight-$8$ parts of $H^9_c$ (the unit module of the $\WE$-cover) and of $H^8_c$ (a transcendental piece $\Lambda^3V\otimes H^1_!$).  Also open are a proof of the CM identity of Remark~\ref{rem:cm}, and an explanation of the lower-order terms at $p=5$, at $q\equiv1\pmod3$ and at $q\equiv1\pmod5$.  Any exact formula for $N_C(q)$ must respect the divisibility by $q$ and the twist symmetry of \S\ref{sec:twist}.
\item \emph{Uniformity in $p$.} Answered by Theorem~U: the constant in Theorem~\ref{thm:effcheb} is bounded uniformly in $p$, so $N_C(q)/|\Fam(q)|\to|C|/|W|$ as $q\to\infty$ through all prime powers of characteristic $\ge5$, in particular through the primes.  The constant obtained, $K\le6\cdot19^{119}$, is useless at any computable $q$.  An explicit bound that is meaningful for $q\le10^3$, say (cf.\ \cite[p.~2]{LoughranTrepalin} on effectivity), remains open.  Heuristically (Remark~\ref{rem:Uscope}), it cannot come from bounds on Betti numbers alone.
\item \emph{The K3 regime.} For $\deg D=7$ the surface $X_D$ is K3, $H^2$ has a transcendental part, $\Lunit$ is not integral, and there is no finite monodromy group.  What replaces Theorems~C--D?  One expects the rank-conditioned trace statistics to be governed by the (infinite) monodromy of the transcendental lattice, i.e.\ a Katz--Sarnak regime for the $12$-dimensional $L$-function, with the algebraic part of $\NS$ still contributing a finite Weyl-group piece.
\item \emph{Other slices and other Weyl groups.} Theorem~\ref{thm:general} covers $y^2=x^3+Ax+B$ with $\deg A\le3$, $\deg B=5$.  Every entry of the Oguiso--Shioda list \cite{OguisoShioda} with trivial lattice $T$ and Mordell--Weil lattice $L$ gives a family whose Frobenius lives in $O(L)$ (a subgroup of $W(E_7)$, $W(E_6)$, $W(D_5)$, \dots), with its own rank law, $\Sha$ law and trace constants; for del Pezzo surfaces of degree $\le3$ the equidistribution is \cite[Thm.~1.9]{BFL}.  Which slices have full monodromy, and is there a uniform statement across the list?
\item \emph{$\Sha$ as a module.} On $H^1(G,U_8)=H^1(\Zhat,E_8)\cong N/(1-w)N^*$ Urabe's cup-product pairing is $\langle\bar x,\bar y\rangle=-((1-w)^{-1}x,y)\bmod\Z$ \cite[Lemma~1.15]{Urabe1996MComp}, alternating and non-degenerate \cite[Thm.~1.10]{Urabe1996MComp}, and Milne \cite{Milne70} identifies the cup-product pairing on $H^1(G,\Pic\Xbar)$ with Tate's pairing on the prime-to-$p$ part of $\Br(X)$.  Under \eqref{eq:shachain}, $\Sha(E_D)$ also carries the Cassels--Tate pairing.  Do these pairings agree up to sign?  (For the Artin--Tate pairing on $\Br(X)$ of an arbitrary surface over a finite field the alternating property is Feng's theorem \cite{Feng2017}; see also Urabe \cite{Urabe96b}.)
\item \emph{Higher place degrees.} For $d\ge2$ the mean of $a_v$ over places of degree exactly $d$ is not a class function (Theorem~C); its rank-conditioned averages at $q=7,11,13,23$ are in the ancillary table \texttt{lefschetz.tex}.  Do they converge as $q=p^f\to\infty$, and is there a clean description of the limit?  It would involve the joint distribution of $(w_D,\sum_{\deg v=1}a_v^2,\dots)$, i.e.\ of the Frobenius class together with lower-degree fibre data.
\end{enumerate}

\appendix

\section{Computational details, reproducibility and data}\label{app:comp}

\subsection{Pipelines}\label{app:pipelines}
\begin{sloppypar}
\emph{Point counts} (SageMath 10.8, PARI/GP 2.17, numpy; \texttt{lpoly.py}, \texttt{tables.py}, \texttt{run\_engine.py}, \texttt{run\_q7\_deg5.py}).  Trace tables $c\mapsto a(c)$ of $y^2=x^3+x+c$ over $\mathbb F_{q^d}$ are built by PARI's \texttt{cardinality} (and, for $q^d\le4\cdot10^4$, also by a numpy character sum $-\sum_x\chi(x^3+x+c)$); split/non-split at nodes by the square test on $3x_0$.  Places of degree $d$ are Frobenius orbits, the power sums $p_n$ are exact integers, and Newton's identities give $c_0,\dots,c_K$ ($K=5$).  $\deg L=\deg N-4$ is taken from the conductor (Theorem~A(1), Table~\ref{tab:local}); the functional equation completes the polynomial with $\varepsilon$ from a pair $(c_i,c_{8-i})$ both known, else from the middle coefficient, else from the product of local root numbers (tame formulas: $w_\infty=(-1/q)$ for II, II$^*$, I$_0^*$, $(-3/q)$ for IV, IV$^*$; $w_v=-1$ iff split multiplicative).  For $c_1=\dots=c_4=0$ Theorem~A alone forces $\varepsilon=-1$ ($z^8+1=\Phi_{16}$ is not a $\WE$ characteristic polynomial).  Every $L$-polynomial is checked to satisfy the functional equation, $|\alpha_j|=q$, $c_i/q^i\in\Z$, full cyclotomicity, $\varepsilon=(-1)^{\rank}$ and vanishing of the coefficients beyond $\deg L$.  All $16{,}807$ monic $D$ at $q=7$ were also computed with $K=8$ (all places of degree $\le8$, no functional-equation input), with identical results, and so were all $3{,}025$ curves of $\Fam(5)$ (by the FFT pipeline below, \texttt{enumerate\_q5.py}); at $q=11,13$ the non-squarefree-$D$ part of $\Fam(q)$ was recomputed with $K=6$, with complete agreement (ancillary \texttt{k6\_recount\_F.sage}), and the root-number fallback agrees with $\varepsilon$ wherever the coefficients determine it (more than $380{,}000$ cases).  No complete ($K=8$) computation was made for all of $\Fam(11)$, $\Fam(13)$.

\emph{Lefschetz sums by FFT} (\texttt{lefschetz\_fft.py}; used to find the curves of Tables~\ref{tab:certificate} and~\ref{tab:lines}, not in the proof of Theorem~B).  The power sums $p_n=\sum_{\alpha\in\mathbb F_{q^n}}\sum_{x\in\mathbb F_{q^n}}\chi_n(x^3+x+D(\alpha))$ (Theorem~C read backwards) are computed with the inner sums tabulated for all $c\in\mathbb F_{q^n}$ at once by a real FFT over $(\Z/q)^n$ (the additive convolution $a(c)=-\sum_uf(u)\chi(u+c)$ with $f(u)=\#\{x:x^3+x=u\}$) and rounded from float64 (observed maximal deviation from an integer $\le5\cdot10^{-11}$ at $N=11^8$); $K=8$ for $q\le13$, and $K=6$ plus the orthogonality identity $\chi_w(z)=\det(w)z^8\chi_w(1/z)$ at $q=23$, with $\varepsilon$ read off a non-zero pair $(c_i,c_{8-i})$ or, for the type $z^8-1$ where all computed coefficients vanish, forced by Theorem~A (the column ``$\varepsilon$ from'' of Table~\ref{tab:certificate}).  The tables agree bit for bit with those of the point-count pipeline where both exist and with PARI on $300$ random entries of every table; \texttt{verify\_certificate.py} (pure integer arithmetic) re-derives every row of Table~\ref{tab:certificate} from the recorded power sums and re-checks the exclusion logic.

\emph{Sections} (\texttt{sections\_core.py}, \texttt{run\_sections.py}, \texttt{lattice\_invariants.py}, \texttt{weyl\_classes.py}, \texttt{merge.py}, \texttt{gap\_conjugacy\_check.py}, \texttt{twist\_symmetry\_check.py}; for Theorem~B the independent implementation \texttt{sections\_w.sage}): \S\ref{sec:sectionsmethod}.

\emph{$\WE$} (\texttt{we8\_classes.sage}; \texttt{weyl\_e8\_group.py}, \texttt{compute\_certificate.py}, \texttt{verify\_certificate.py}, \texttt{curve\_specialization.py}, \texttt{shared\_charpoly\_classes.py}): the $112$-class table and the limit constants (exact rationals) by Sage's \texttt{WeylGroup(['E',8])}; independently, the $240$-point permutation group in GAP with its normal and maximal subgroups; the class sizes per characteristic polynomial agree between the two.

\emph{Analysis} (\texttt{analysis\_F5.py}, \texttt{we8.py}, \texttt{make\_figures.py}, \texttt{make\_class\_tables.py}, \texttt{make\_monodromy\_tables.py}): all means are exact rationals, printed rounded; every table and figure is generated by these scripts from the per-curve files, and every table fragment carries a provenance comment, except \texttt{finite\_q\_kappa.tex} and \texttt{finite\_q\_cfamily.tex} (Section~\ref{sec:uniform}), which are written by \texttt{finite\_q/kappa\_table.py}.
\end{sloppypar}

\subsection{The Frobenius class from the $240$ minimal sections}\label{sec:sectionsmethod}
The $L$-polynomial gives only $\chi_{w_D}$; to obtain the class of $w_D$ we compute $w_D$ itself.  All arithmetic is exact (finite fields, integer and rational linear algebra, Smith normal forms; no floating point).
(1) \emph{Minimal sections.} Comparing coefficients of $t^0,\dots,t^6$ in $y^2=x^3+x+D$ with $x=x_2t^2+x_1t+x_0$, $y=y_3t^3+\dots+y_0$ gives seven polynomial equations over $\Fq$; the quotient algebra $A=\Fq[x_i,y_j]/I$ has $\Fq$-dimension $240$ for every curve (asserted), so the scheme of minimal sections is reduced and consists of the $240$ roots (Theorem~A(2)).  (2) \emph{Frobenius orbits without root finding.} Multiplication matrices of the seven coordinates on $A$ ($240\times240$ over $\Fq$; Gr\"obner basis via Singular's \texttt{slimgb}) are computed; $A$ is a product of finite fields, one per Frobenius orbit (residue field $\mathbb F_{q^d}$ $\leftrightarrow$ orbit of length $d$), and the decomposition is obtained by repeatedly splitting along kernels of irreducible factors of the characteristic polynomial of a coordinate (or a random $\Fq$-combination) restricted to a block; a block on which some element has irreducible characteristic polynomial of full degree is a field (asserted: $\dim\ker h(M)=\deg h\cdot(\text{multiplicity})$, which fails if $A$ were non-reduced).  (3) \emph{Explicit points.} With $N=\operatorname{lcm}$ of the block degrees $=\ord(w_D)\le30$, all $240$ points are written down in $\mathbb F_{q^N}[t]$, verified to satisfy the equation and to be distinct; the Frobenius permutation $x\mapsto x^q$ on coordinates is read off, and its cycle type is asserted equal to the block degrees.  (4) \emph{Height pairing and $w_D$.} For distinct minimal sections $\langle P,Q\rangle=1-\bar P\cdot\bar Q$, and $\bar P\cdot\bar Q$ is the number of common zeros on $\Pp$ (with multiplicity) of the binary forms $x_P-x_Q$ (degree $2$) and $y_P-y_Q$ (degree $3$), i.e.\ $\deg\gcd$ of the affine parts plus the minimum of the orders at $\infty$; this is the local intersection number because sections pass through smooth points of the Weierstrass model, which is the minimal regular model.  Eight roots with non-singular Gram matrix are chosen, all $240$ roots receive rational coordinates, the $\Z$-span is asserted even unimodular of rank $8$ (hence $E_8$), integer coordinates are taken, and $w_D$ is the integer matrix with $w_D\,c(P)=c(\sigma P)$ for all $240$ roots (asserted for all $P$; asserted isometry).  Sanity: every root has exactly $1/56/126/56/1$ roots at pairing $2/1/0/-1/-2$.  (5) \emph{Invariants and class.} From $(G,\text{roots},w_D)$: characteristic polynomial, $\det$, $\mult_1$, cycle types on the $240$ roots and on the $2160$ norm-$4$ vectors, Smith normal form of $1-w_D$ ($\Sha$ as a group), fixed lattice with Gram matrix, discriminant and root type, $\detp(1-w_D)$.  The class is looked up by (characteristic polynomial, cycle type on roots), which separates the $112$ classes, and the full invariant key of $w_D$ is asserted equal to that of the class representative.
This was run for all $12{,}320$ curves of $\Fam(7)$, a seeded sample of $7{,}500$ curves drawn uniformly from all of $\Fam(11)$ (seed $2026$; $693$ of them have $D$ not squarefree) and a seeded sample of $2{,}000$ curves drawn uniformly from all of $\Fam(13)$ (seed $2026$).

\subsection{Verification}\label{app:chain}
\begin{sloppypar}
Besides the checks of Appendix~\ref{app:pipelines}: (V1) every $\Lunit$ is one of the $106$ characteristic polynomials of $\WE$ (all curves, all $q$).  (V2) Theorem~C: the pipeline's power sums $p_d(E_D)$, $d\le5$, equal $q^dp_d(w_D)$ on every curve of every data set; the Lefschetz sums recomputed from the per-place traces with separately written code equal $-q^dp_d(w_D)$ on all of $\Fam(7)$ ($d\le8$), on $300$ random curves per $q$ at $q=11,13,23$ ($d\le5$) and on the non-squarefree-$D$ parts of $\Fam(11)$, $\Fam(13)$ ($d\le6$); $\bar a_1(D)=-\tr w_D$ on every curve.  (V3) Sections against point counts, on each of the $21{,}820$ curves of \S\ref{sec:sectionsmethod}: $\chi_{w_D}=\Lunit$, $\mult_1(w_D)=\rank$, $\det w_D=\varepsilon(E_D)$, $\bar a_1(D)=-\tr w_D$, $\detp(1-w_D)=|\Sha|\cdot\disc E_8^{w_D}$, $\Sha\cong H\times H$, and at rank~$0$ $\det(1-w_D)=L(E_D,1/q)$; the two computations share no code, so this is a curve-by-curve check of \eqref{eq:LdetW}.  For up to $30$ curves per observed class at $q=7,11$ ($3{,}619$ curves), GAP's \texttt{IsConjugate} confirmed the class.  (V4) Theorem~B: the maximal subgroups were also computed on $\WE/\{\pm1\}\cong O_8^+(2){:}2$ and lifted, with the same result, and the certificate is checked by a Sage-free integer script.  (V5) $N_C=N_{-C}$ on all of $\Fam(7)$ and $\Fam(11)$, and $q\mid N_C$ at $q=7,11,13$.  The per-check counts are in the ancillary tables \texttt{trace\_identity\_check.tex} and \texttt{sections\_crosscheck.tex}.
\end{sloppypar}

\subsection{Ancillary files and data availability}\label{app:data}
\begin{sloppypar}
The arXiv ancillary files (described in their \texttt{README.md}) contain: the per-curve tables for $\Fam(7)$ and the $q=23$ sample ($L$-coefficients, type, rank, $\varepsilon$, $|\Sha|$, Lefschetz sums, mean traces); the sections files \texttt{frob\_class\_q7.csv}, \texttt{frob\_class\_q11\_sample.csv}, \texttt{frob\_class\_q13\_sample.csv} (per curve: $w_D$, class, cycle type on roots, $\Sha$ group, fixed lattice with Gram matrix and root type, regulator) and the $112$-class table \texttt{weyl\_e8\_classes.csv}; the monodromy files (certificate, maximal and normal subgroups, the two lines); the $\WE$ limit tables; all scripts named above; the table fragments \texttt{tables/*.tex}, including those not printed here (\texttt{type\_freq\_q*}, \texttt{classes\_rank0\_q*}, \texttt{lefschetz}, \texttt{sections\_mwl}, \texttt{rank\_dist}, \texttt{sha\_dist}, \texttt{trace\_identity\_check}, \texttt{sections\_crosscheck}); the scripts, with their outputs, behind individual statements: \texttt{sections\_w.sage} (\S\ref{sec:certificate}), \texttt{checks/scripts/urabe\_no1\_explicit.sage} (\S\ref{sec:HxH}), \texttt{deg7\_c1.gp} (Remark~\ref{rem:prop13}), \texttt{at\_kt\_local.sage}, \texttt{torsion\_check.py}, \texttt{torsion\_f49.gp} (Section~\ref{sec:local}), \texttt{enum\_family.gp} (\eqref{eq:counts}, Remark~\ref{rem:whichplaces}) and \texttt{k6\_recount\_F.sage}; \texttt{CORRECTIONS.md} (\S\ref{sec:changed}); and the folder \texttt{finite\_q/} for Section~\ref{sec:uniform}, described in its \texttt{README.md}.  That folder contains the exact class counts $N_C(q)$ for all of $\Fam(q)$, $5\le q\le41$, and for the $c$-families, with the pipeline that produced them, and the $\WE$ character table.  It also contains \texttt{kappa\_table.py} (Tables~\ref{tab:kappa} and~\ref{tab:cfamily}, and the out-of-sample numbers of \S\ref{sec:conj}), \texttt{theorem\_M/} (an independent enumerator and the checks of Theorem~M, including the pair count) and \texttt{cm\_check/} (Remark~\ref{rem:cm}).  Trace tables (up to $2$ GB) are regenerable from the scripts and are not shipped.  The per-curve tables for $\Fam(11)$ and $\Fam(13)$ (about $115$ MB uncompressed) and the intermediate data of the recomputation of the first version (about $170$ MB) exceed arXiv's size limit; they are available from the author on request, are regenerated by the included scripts, and are listed with checksums in the ancillary file \texttt{NOT\_SHIPPED.md}.
\end{sloppypar}

\subsection*{Acknowledgments}
Computations, drafting and verification for this paper were carried out with AI assistance (Anthropic's Claude) under the author's direction; all computational results of this version were recomputed independently.  The finite-$q$ pattern in the class distribution at $q=7,\dots,23$ that led to Section~\ref{sec:uniform} was first noticed in an AI-assisted independent review of an earlier draft of this paper.  The author is responsible for the content.


\begin{thebibliography}{99}

\bibitem{ATLAS}
J.~H.~Conway, R.~T.~Curtis, S.~P.~Norton, R.~A.~Parker and R.~A.~Wilson, \emph{Atlas of Finite Groups}, Oxford University Press, 1985.

\bibitem{BFL}
B.~Banwait, F.~Fit\'e and D.~Loughran,
\emph{Del Pezzo surfaces over finite fields and their Frobenius traces},
Math. Proc. Cambridge Philos. Soc. \textbf{167} (2019), 35--60; arXiv:1606.00300.

\bibitem{BBLLD}
J.~Bober, A.~R.~Booker, M.~Lee and D.~Lowry-Duda,
\emph{Murmurations of modular forms in the weight aspect},
Algebra Number Theory \textbf{20} (2026), 299--332; arXiv:2310.07746.

\bibitem{Bergvall2020}
O.~Bergvall, \emph{Seven points in general linear position}, arXiv:2007.16014.

\bibitem{Bergvall2024}
O.~Bergvall, \emph{Frobenius actions on del Pezzo surfaces of degree~$2$}, Innov. Incidence Geom. \textbf{21} (2024), 1--10; arXiv:2211.12855.

\bibitem{BergvallGounelas}
O.~Bergvall and F.~Gounelas, \emph{Cohomology of moduli spaces of del Pezzo surfaces}, arXiv:1904.10249.

\bibitem{Bessis}
D.~Bessis, \emph{Finite complex reflection arrangements are $K(\pi,1)$}, arXiv:math/0610777.

\bibitem{Bourbaki}
N.~Bourbaki, \emph{Groupes et alg\`ebres de Lie}, Chapitres IV--VI, Hermann, Paris, 1968.

\bibitem{Carter}
R.~W.~Carter, \emph{Conjugacy classes in the Weyl group}, Compositio Math. \textbf{25} (1972), 1--59.

\bibitem{Charles2013}
F.~Charles, \emph{The Tate conjecture for K3 surfaces over finite fields}, Invent. Math. \textbf{194} (2013), 119--145; arXiv:1206.4002.

\bibitem{ConwaySloane}
J.~H.~Conway and N.~J.~A.~Sloane, \emph{Sphere Packings, Lattices and Groups}, 3rd ed., Grundlehren 290, Springer, 1999.

\bibitem{CossecDolgachev}
F.~Cossec and I.~Dolgachev, \emph{Enriques Surfaces I}, Progress in Math. 76, Birkh\"auser, 1989.

\bibitem{Cowan}
A.~Cowan, \emph{Murmurations and ratios conjectures}, arXiv:2408.12723.

\bibitem{Das}
R.~Das, \emph{Arithmetic statistics on cubic surfaces}, arXiv:2002.11183.

\bibitem{DeligneWeilI}
P.~Deligne, \emph{La conjecture de Weil. I}, Publ. Math. IH\'ES \textbf{43} (1974), 273--307.

\bibitem{DeligneWeilII}
P.~Deligne, \emph{La conjecture de Weil. II}, Publ. Math. IH\'ES \textbf{52} (1980), 137--252.

\bibitem{EGAIV}
A.~Grothendieck and J.~Dieudonn\'e, \emph{\'El\'ements de g\'eom\'etrie alg\'ebrique IV}, Publ. Math. IH\'ES \textbf{20, 24, 28, 32} (1964--67).

\bibitem{Ekedahl}
T.~Ekedahl, \emph{An effective version of Hilbert's irreducibility theorem}, in \emph{S\'eminaire de Th\'eorie des Nombres, Paris 1988--1989}, Progr. Math. \textbf{91}, Birkh\"auser, Boston, 1990, pp.~241--249.

\bibitem{Feng2017}
T.~Feng, \emph{\'Etale Steenrod operations and the Artin--Tate pairing}, Compos. Math. \textbf{156} (2020), 1476--1515; arXiv:1706.00151.

\bibitem{Groth}
A.~Grothendieck, \emph{Le groupe de Brauer III: Exemples et compl\'ements}, in \emph{Dix expos\'es sur la cohomologie des sch\'emas}, North-Holland, Amsterdam, 1968, pp.~88--188.

\bibitem{HLOP2022}
Y.-H.~He, K.-H.~Lee, T.~Oliver and A.~Pozdnyakov, \emph{Murmurations of elliptic curves}, Experimental Math. \textbf{34} (2025), no.~3, 528--540; arXiv:2204.10140.

\bibitem{HeckmanLooijenga}
G.~Heckman and E.~Looijenga, \emph{The moduli space of rational elliptic surfaces}, in \emph{Algebraic Geometry 2000, Azumino}, Adv. Stud. Pure Math. \textbf{36} (2002), 185--248; arXiv:math/0010020.

\bibitem{JKZ}
F.~Jouve, E.~Kowalski and D.~Zywina, \emph{An explicit integral polynomial whose splitting field has Galois group $W(E_8)$}, J. Th\'eor. Nombres Bordeaux \textbf{20} (2008), 761--782; arXiv:0801.1733.

\bibitem{KT}
K.~Kato and F.~Trihan, \emph{On the conjectures of Birch and Swinnerton-Dyer in characteristic $p>0$}, Invent. Math. \textbf{153} (2003), 537--592.

\bibitem{Karras2026}
L.~Karras, \emph{On the inverse Galois problem for del Pezzo surfaces of degree 1}, arXiv:2604.02036.

\bibitem{Katz2001}
N.~M.~Katz, \emph{Sums of Betti numbers in arbitrary characteristic}, Finite Fields Appl. \textbf{7} (2001), 29--44.

\bibitem{KatzSarnak}
N.~M.~Katz and P.~Sarnak, \emph{Random Matrices, Frobenius Eigenvalues, and Monodromy}, AMS Colloquium Publ. 45, 1999.

\bibitem{KowalskiSieve}
E.~Kowalski, \emph{The Large Sieve and its Applications}, Cambridge Tracts in Math. 175, Cambridge Univ. Press, 2008.

\bibitem{KumarShioda}
A.~Kumar and T.~Shioda, \emph{Multiplicative excellent families of elliptic surfaces of type $E_7$ or $E_8$}, Algebra Number Theory \textbf{7} (2013), 1613--1641; arXiv:1204.1531.

\bibitem{LLOP}
M.~Lal\'in, K.-H.~Lee, T.~Oliver and A.~Pozdnyakov, \emph{Murmurations of quadratic and cubic characters over function fields}, arXiv:2608.01337.

\bibitem{LoughranTrepalin}
D.~Loughran and A.~Trepalin, \emph{Inverse Galois problem for del Pezzo surfaces over finite fields}, Math. Res. Lett. \textbf{27} (2020), 845--853; arXiv:1811.06785.

\bibitem{Matsumura}
H.~Matsumura, \emph{Commutative Ring Theory}, Cambridge Univ. Press, 1986.

\bibitem{Milne70}
J.~S.~Milne, \emph{The Brauer group of a rational surface}, Invent. Math. \textbf{11} (1970), 304--307.

\bibitem{Milne1975}
J.~S.~Milne, \emph{On a conjecture of Artin and Tate}, Ann. of Math. \textbf{102} (1975), 517--533.

\bibitem{MilneADT}
J.~S.~Milne, \emph{Arithmetic Duality Theorems}, 2nd ed., BookSurge, 2006.

\bibitem{MilneEC}
J.~S.~Milne, \emph{\'Etale Cohomology}, Princeton Univ. Press, 1980.

\bibitem{Nikulin}
V.~V.~Nikulin, \emph{Integral symmetric bilinear forms and some of their applications}, Math. USSR Izv. \textbf{14} (1980), 103--167.

\bibitem{OguisoShioda}
K.~Oguiso and T.~Shioda, \emph{The Mordell--Weil lattice of a rational elliptic surface}, Comment. Math. Univ. St. Pauli \textbf{40} (1991), 83--99.

\bibitem{PoonenStoll}
B.~Poonen and M.~Stoll, \emph{The Cassels--Tate pairing on polarized abelian varieties}, Ann. of Math. \textbf{150} (1999), 1109--1149.

\bibitem{SS}
W.~Sawin and A.~V.~Sutherland, \emph{Murmurations for elliptic curves ordered by height}, arXiv:2504.12295.

\bibitem{SchuettShioda}
M.~Sch\"utt and T.~Shioda, \emph{Elliptic surfaces}, in \emph{Algebraic Geometry in East Asia -- Seoul 2008}, Adv. Stud. Pure Math. \textbf{60} (2010), 51--160; arXiv:0907.0298.

\bibitem{SchuettShiodaBook}
M.~Sch\"utt and T.~Shioda, \emph{Mordell--Weil Lattices}, Springer, Singapore, 2019.

\bibitem{SerreCourse}
J.-P.~Serre, \emph{A Course in Arithmetic}, GTM 7, Springer, 1973.

\bibitem{SerreGC}
J.-P.~Serre, \emph{Galois Cohomology}, Springer, 1997.

\bibitem{ShephardTodd}
G.~C.~Shephard and J.~A.~Todd, \emph{Finite unitary reflection groups}, Canad. J. Math. \textbf{6} (1954), 274--304.

\bibitem{Solomon1963}
L.~Solomon, \emph{Invariants of finite reflection groups}, Nagoya Math. J. \textbf{22} (1963), 57--64.

\bibitem{SGA1}
A.~Grothendieck, \emph{Rev\^etements \'etales et groupe fondamental (SGA 1)}, LNM 224, Springer, 1971.

\bibitem{SGA4h}
P.~Deligne, \emph{Cohomologie \'etale (SGA 4\textonehalf)}, LNM 569, Springer, 1977.

\bibitem{Shioda1989}
T.~Shioda, \emph{Mordell--Weil lattices and Galois representation. I}, Proc. Japan Acad. Ser. A \textbf{65} (1989), 268--271.

\bibitem{Shioda1990}
T.~Shioda, \emph{On the Mordell--Weil lattices}, Comment. Math. Univ. St. Pauli \textbf{39} (1990), 211--240.

\bibitem{Shioda1991}
T.~Shioda, \emph{Construction of elliptic curves with high rank via the invariants of the Weyl groups}, J. Math. Soc. Japan \textbf{43} (1991), 673--719.

\bibitem{ShiodaUsui}
T.~Shioda and H.~Usui, \emph{Fundamental invariants of Weyl groups and excellent families of elliptic curves}, Comment. Math. Univ. St. Pauli \textbf{41} (1992), 169--217.

\bibitem{SilvermanATAEC}
J.~H.~Silverman, \emph{Advanced Topics in the Arithmetic of Elliptic Curves}, GTM 151, Springer, 1994.

\bibitem{Stacks}
The Stacks project authors, \emph{The Stacks project}, \texttt{https://stacks.math.columbia.edu}.

\bibitem{Tate75}
J.~Tate, \emph{Algorithm for determining the type of a singular fiber in an elliptic pencil}, in \emph{Modular Functions of One Variable IV}, LNM 476, Springer, 1975, pp.~33--52.

\bibitem{TateBourbaki}
J.~Tate, \emph{On the conjectures of Birch and Swinnerton-Dyer and a geometric analog}, S\'em. Bourbaki \textbf{306} (1966).

\bibitem{UlmerPCMI}
D.~Ulmer, \emph{Elliptic curves over function fields}, in \emph{Arithmetic of $L$-functions}, IAS/Park City Math. Ser. 18, AMS, 2011, pp.~211--280; arXiv:1101.1939.

\bibitem{UlmerCRM}
D.~Ulmer, \emph{Curves and Jacobians over function fields}, in \emph{Arithmetic Geometry over Global Function Fields}, Adv. Courses Math. CRM Barcelona, Birkh\"auser, 2014, pp.~281--337; arXiv:1203.5573.

\bibitem{Urabe1996MComp}
T.~Urabe, \emph{Calculation of Manin's invariant for Del Pezzo surfaces}, Math. Comp. \textbf{65} (1996), no.~213, 247--258, with Supplement pp.~S15--S23.

\bibitem{Urabe96b}
T.~Urabe, \emph{The bilinear form of the Brauer group of a surface}, Invent. Math. \textbf{125} (1996), 557--585.

\bibitem{VAZ}
A.~V\'arilly-Alvarado and D.~Zywina, \emph{Arithmetic $E_8$ lattices with maximal Galois action}, LMS J. Comput. Math. \textbf{12} (2009), 144--165; arXiv:0803.3063.

\bibitem{WachsQ}
D.~Wachs, \emph{BSD invariants and murmurations of elliptic curves}, arXiv:2603.04604.


\bibitem{Zub}
N.~Zubrilina, \emph{Murmurations}, Invent. Math. \textbf{241} (2025), 627--680; arXiv:2310.07681.

\end{thebibliography}
\end{document}